\documentclass[12pt]{article}
\usepackage[cp1251]{inputenc}   % Windows
\usepackage[english, russian]{babel}
\usepackage{amssymb}
\usepackage{amsmath}
\newenvironment{proof}[1][Proof]{\noindent\textbf{#1.} }{\ \rule{0.5em}{0.5em}}
\newtheorem{definition}{Definition}[section]
\newtheorem{theorem}{Theorem}[section]
\newtheorem{lemma}{Lemma}[section]
\newtheorem{remark}{Remark}[section]
\newtheorem{corollary}{Corollary}[section]

\numberwithin{equation}{section}

\begin{document}

\begin{center}
{\Large %{New applications of the Egorychev method of coeffficients //}
{New applications of the Egorychev method of coefficients of integral
representation and calculation of combinatorial sums} }
\end{center}

\bigskip
\bigskip

%\bigskip New applications of the Egorychev method of coeffficients
\begin{center}
Maksim Davletshin\footnote{davmaks@gmail.com},
Georgy Egorychev\footnote{gegorych@mail.ru} and
Vyacheslav Krivokolesko\footnote{krivokolesko@gmail.com}
\end{center}

%E. Zima \footnote{%
%ezima@wlu.ca}
\begin{center}
\textbf{Abstract}
\end{center}

%{\small
Here we present the new applications of the Egorychev method of
coef-ficients of integral representations and computation of combinatorial
sums developed by the author at the end of 1970's and its recent
applications to the algebra and the theory of holomorphic functions in
$\mathbb{C}^{n}$ and others.
%}

\renewcommand{\contentsname}{Сontents}
\tableofcontents

\addcontentsline{toc}{section}{Introduction}
\section*{Introduction}

%\part{\S\ 1. Introduction}
\noindent

At the end of the 1970's G.P. Egorychev developed the method of
coefficients, which was successfully applied to many combinatorial problems %sums
\cite {Egor13, Egor23, Egor1,Egor34} and \cite {Leont}. %
Here we present the (see section 2.1)
short description of the Egorychev method and its recent applications to
several problems of enumeration and summation in various fields of mathematics: algebra, the
theory of integral
%representations in\textit{\ }$\mathbb{C}^{n}$and the theory of approximation.
representations in $\mathbb{C}^{n}$ and the theory of approximation.

Enumerative combinatorial problems in algebra have been considered for a
long time
starting from the well-known estimates for the number of Sylow subgroups, the
number of fixed-order subgroups in a finite p-group (G. Frobenius, Ph. Hall,
etc.).
For example, the ranks of the $n$-factors $M_{q}(n)$ for the
lower central series of
the free group $\Phi $ with $q$ generators
\begin{equation*}
M_{q}(n)=\frac{1}{n}\sum \eta (d) q^{n/d}
\end{equation*}
were computed by M.Jr \ Hall (\cite{Hall59}, Theorem 11.2.2).

%In Part 1 we have solved two enumerative combinatorial problems in the ring theory.
%In section 1 we have solved two enumerative combinatorial problems in the ring theory.
%In \S\ 2 we have obtained two simple formulas for the number of quadrics and symmetric
In section 2.2 we obtained two simple formulas for the number of quadrics and symmetric
forms of modules over local rings \cite{Star2004, EgorZima(2008)}.

In section 3 we solved
the interesting enumeration problem of the number of $D$-ideals of ring
$R_{n}(K,J) $ in lattices by means of
the inclusion-conclusion method \cite{Egor(2009), Davl(2011), DavEgLev}.
Thus it was possible to us with the help of the Egorychev method to
calculate the difficult
6-multiple (!) combinatorial sum in the closed form.

In section 4 we solved several summation problems. In sections 4.1, 4.2 and section 4.3 we
found the simple new computation %proof 
and generalization of the several multiple
combinatorial sums, which originally arose in the theory of holomorphic functions in
$\mathbb{C}^{n} $ \cite{th, th1, EDK2011}.

In section 5 we found the new short computation of multiple sum to the
theory of cubature
formulas \cite{XeoXu(2000), Egor(2012)}.

During the work on the article we added a number of corrections and the additions
improving
earlier known results.%\bigskip
%\section{Part 1. Two enumerative combinatorial problems in algebra\protect \medskip}
\section{The method of coefficients and its algebraic applications}

\noindent

Hans Rademacher \cite{Radem} noted, that the applications of the method
of generating functions is usually connected  with the use of operations over the
Laurent series and the Dirichlet series. Earlier G.P Egorchev had 
developed the method of integral representations and calculation of
combinatorial sums of various type (the inference rules and
the completeness Lemma; see, \cite{Egor1, Egor13, Egor16, Egor23, Egor34}, and \cite{Leont})
connected with the use of the theory of analytic
functions, the theory of multiple residues in $\mathbf{C}^{n}$ and the
formal power Laurent series over $\mathbb{C}.$
%\subsection{The Egorychev method of coefficients\protect\medskip}

\subsection{The Egorychev method of coefficients}

%\subsubsection{The computational scheme of the method of coeffficients\protect\medskip}
\subsubsection{The computational scheme of the method of coefficients}

\noindent
The general scheme of the Egorychev method of integral representation and
computation of sums can be followed up by the following steps \cite{Egor1}.

\textbf{1.} \textit{Assignment of a table of integral representations of
combinatorial numbers.}

\textit{For example}, \textit{the binomial coefficients} $\binom{n}{k},\,$
$n,\,k=0,1,\ldots ,$%
\begin{equation}
\binom{n}{k}=\mbox{\bf res}_{w}(1+w)^{n}w^{-k-1}=\frac{1}{2\pi i}%
\int\limits_{0<|w|=\varrho }(1+w)^{n}w^{-k-1}dw,
\label{M1}
\end{equation}%
\begin{equation}
\binom{n+k-1}{k}=\mbox{\bf res}_{w}(1-w)^{-n}w^{-k-1}=
\frac{1}{2\pi i}\int\limits_{0<|w|=\varrho <1}\frac{(1-w)^{-n}}{w^{k+1}}dw.
\label{M3}
\end{equation}

{\it{The Stirling numbers $ S_{2}(n, k) $ of second kind,  $n, k=0, 1,\ldots $
(\cite{Egor1}, p. 273)$:S_{2}(0,0):=1 $ and}}
\begin{equation}
S_{2}(n,k)=\mbox{\bf res}_{w}(-1+\exp w)^{n}w^{-k-1}=
\frac{1}{2\pi i}\int\limits_{0<|w|=\varrho }\frac{(-1+\exp w)^{n}}{w^{k+1}}dw.
\label{M4}
\end{equation}

{\it{The Kronecker symbol $\delta (n,k),\, n,k=0,\,1,\,\ldots,$}}
\begin{equation}
\delta (n,k)=\mbox{\bf res}_{w}w^{-n+k-1}.
\label{M5}
\end{equation}

\textbf{2.} {\it{Representation of the summand $a_{k}$ of the
original sum $\sum_{k}a_{k}$ by a sum of product of combinatorial numbers.}}

\textbf{3.} {\it{Replacement of the combinatorial numbers by their integrals.}}

\textbf{4.} {\it{Reduction of products of integrals to multiple integral.}}

\textbf{5.} {\it{Interchange of the order of summation and integration. This
gives the integral representation of original sum with the kernel
represented by a series. The use of this transformation requires to
deform the domain of integration in such a way as to obtain the series under
the integral which converges uniformly on this domain saving the value of
the integral.}}

\textbf{6.} {\it{Summation of the series under the integral sign. As a
rule, this series turns out to be a geometric progression \cite{Hardy49}.
 This gives the integral representation of the original sum with
the kernel in closed form.}}

\textbf{7.} {\it{The computation of the resulting integral by means of
tables of integrals, iterated integration, the theory of one and
multidimensional residues, or new methods.}}

%\subsubsection{Operations with formal power series and the inference rules}

%\noindent
%Hans Rademacher \cite{Radem} has noted, that the applications of the method
%of generating functions is connected usually with use of operations over the
%Laurent series and the Dirichlet series. Earlier the author had been
%developed the method of integral representations and calculation of
%combinatorial sums of various type (see, \cite{Egor1,Egor13,Egor16,Egor23},
%and also \cite{Leont'ev2006}) connected with use the theory of analytic
%functions, the theory of multiple residues in $\mathbf{C}^{n}$ and the
%formal power Laurent series over $\mathbb{C}.$

\subsubsection{Laurent power series and the inference rules: definition  and properties of $\textbf{res}$
operator}
\noindent
%\paragraph{Laurent power series: definition\ and\ properties\ of\ the
%residue\ operator\textbf{\protect\large \protect\smallskip }}

%\noindent
Using the $\mbox{\bf res}$ concept and its properties the idea of integral
representations can be extended on sums that allow computation with the help
of formal Laurent power series of one and several variables over $\mathbb{C}.$
The $\mbox{\bf res}$ concept is directly connected with the classic
concept of residue in the theory of analytic functions and which may be used
with series of various types. This connection enabled us to express
properties of $\mbox{\bf res}$ operator analogous to properties of residue
in the theory of analytic functions. This in turn allows us to unify the
scheme of the method of integral representations of sums independently of
what kind of series -- convergent or formal -- is used (separately, or
jointly) in the process of computation of a particular sum.

In this section we shall restrict ourselves explaining only one-dimensional
case, although in further computations the $\mbox{\bf res}$ concept will
also be used for multiple series. Besides, the one-dimensional case is of
interest both in itself and in the computation of multiple integrals ($%
\mbox{\bf res}$) in terms of repeated integrals.%{\Large \medskip }

%\subsubsection{Definition and properties of the $\mbox{\bf res}$ operator}
\noindent
%\noindent
Let $L$ be the set of formal Laurent power series over $\mathbb{C}$
containing only finitely many terms with negative degrees. The {\it{order}}
of the monomial $c_{k}w^{k}$ is $k$. The {\it{order}} of the series
$C(w)=\sum_{k}c_{k}w^{k}$ from $L$ is minimal order of monomials with nonzero
coefficient.

Let $L_{k}$ denote the set of series of order
$k,\,L=U_{k=-\infty }^{\infty }L_{k}.$

Two series $A(w)=\sum_{k}a_{k}w^{k}$ and
$B(w)=\sum_{k}b_{k}w^{k}$ from $L$ are equal if $a_{k}=b_{k}$ for all $k.$ We
can introduce in $L$ operations of addition, multiplication, substitution,
inversion and differentiation \cite{Cartan61, Evgr2, Henrici}.

The ring $L$
is a field \cite{Post}.

Let $f(w),\,\psi (w) \in L_{0}.$ Below we shall use
the following notations:
$$h(w)=wf(w)\in L_{1},\, l(w)=\frac{w}{\psi (w)}\in L_{1},\,
z'(w)=\frac{d}{dw}z(w),\,\overline{h}=\overline{h}(z)\in L_{1}$$
-- the inverse series of the series $z=h(w) \in L_{1}.$

{\it{ For $C(w)\in L$ define the formal residue as}}
\begin{equation}
\mbox{\bf res}_{w}C(w)=c_{-1}.
\label{M11}
\end{equation}
Let $A(w)=\sum_{k}a_{k}w^{k}$ be the {it{generating function}}
for the sequence $\{a_{k}\}.$ Then%
\begin{equation}
a_{k}=\mbox{\bf res}_{w}A(w)w^{-k-1},\ k=0,1,....  \label{M12}
\end{equation}%
For example, one of the possible representations of the binomial coefficient
is%
\begin{equation}
\binom{n}{k}=\mbox{\bf res}_{w} (1+w)^{n}w^{-k-1},\ k=0,1,...,n.
\label{M13}
\end{equation}
There are several properties, ({\it{inference rules}}) for the
$\mbox{\bf res}$ operator which immediately follow from its definition and
properties of operations in formal Laurent power series over $\mathbb{C}$.

We list only a few of them which will be used in this article. Let
$A(w)=\sum_{k}a_{k}w^{k}$ and $B(w)=\sum_{k}b_{k}w^{k}$ be the generating
functions from L.

\textbf{Rule 1} (Removal).%
\begin{equation}
\mbox{\bf res}_{w}A(w)w^{-k-1}=\mbox{\bf res}_{w}B(w)w^{-k-1}\text{for\thinspace\ all }k\text{ iff }A(w)=B(w).
\label{M14}
\end{equation}

\textbf{Rule 2} (Linearity). {\it{For any $\alpha,\,\beta $ from} $\mathbb{C}$}
\begin{equation*}
\alpha \mbox{\bf res}_{w}A(w)w^{-k-1}+\beta \mbox{\bf res}_{w}B(w)w^{-k-1}=
\end{equation*}
\begin{equation}
=\mbox{\bf res}_{w}\left((\alpha A(w)+\beta B(w))w^{-k-1}\right).
\label{M15}
\end{equation}
{\it{By induction from (\ref{M15}) follows, that operators $ \sum $ and $\mbox{\bf res} $ are commutative}}.

\textbf{Rule 3} (Substitution).

a)\, {\it{For $w \in L_{k}\,(k\geq 1)$ and $A(w)$ any element of $L,$ }}

b)\, {\it{for $A(w)$ polynomial and $w$ any element of $ L $ including a constant}}%
\begin{equation}
\sum_{k}w^{k}\mbox{\bf res}_{z}\left( A(z)z^{-k-1}\right) =[A(z)]_{z=w}=A(w).
\label{M16}
\end{equation}

\textbf{Rule 4} (Inversion).

{\it{For $f(w)$ from $L_{0}$}}
\begin{equation}
\sum_{k}z^{k}\mbox{\bf res}_{w}\left( A(w)f(w)^{k}w^{-k-1}\right) =
\left[A(w)/f(w)h^{\prime }(w)\right]_{w=\overline{h(z)}},
\label{M17}
\end{equation}%
{\it {where $z=h(w)=wf(w)\in L_{1}.$}}

\textbf{Rule 5} (Change of variables).

{\it{If $f(w)\in L_{0},$ then}}%
\begin{equation}
\mbox{\bf res}_{w}\left( A(w)f(w)^{k}w^{-k-1}\right)=
\mbox{\bf res}_{z}\left( \left[ A(w)/f(w)h^{\prime }(w)\right]_{w=\overline{h(z)}}z^{-k-1}\right) ,
\label{M18}
\end{equation}%
{\it {where $z=h(w)=wf(w)\in L_{1}.$}}

{\textbf{Rule 6} (Differentiation).%
\begin{equation}
k\,\mbox{\bf res}_{w}A(w)w^{-k-1}=\mbox{\bf res}_{w}A^{\prime }(w)\,w^{-k}.
\label{M19}
\end{equation}}%

\subsubsection{Connection with  the theory of analytic functions}

\noindent

If a formal power series $A(w)\in L$ converges in a punctured neighborhood
of zero, then the definition of $\mbox{\bf res}_{w}A(w)$ coincides with the
usual definition of $\mbox{\bf res}_{w=0}A(w)$, used in the theory of
analytic functions.

The formula (\ref{M12}) is an analog of the well known integral Cauchy formula
\begin{equation*}
a_{k}=\frac{1}{2\pi i}\int_{|w|=\rho} A(w)w^{-k-1}dw
\end{equation*}%
for the coefficients of the Taylor series in a punctured neighborhood of
zero. The substitution rule (\ref{M16}) of the $res$ operator is a direct
analog of the famous Cauchy theorem. Similarly, it is possible to introduce
the definition of formal residue at the point of infinity, the logarithmic
residue and the theorem of residues (all necessary concepts and results in
the theory of residues in one and several complex variables
\cite{Cartan61, Evgr2, AizenYuzh83, Egor1}).

 Moreover, it is easy to see that each rule of
the $\mbox{\bf res}$ operator can be simply proven by reduction to the known
formula in the theory of residues for corresponding rational
function.%\bigskip

\subsection {Simple formulas for the number of quadrics and
  symmetric forms of modules over local rings}

\noindent

Simple formulas for the number of quadrics and symmetric forms of modules
over local rings are found analytically. This simplification of known
formulas of Levchuk and Starikova is achieved using the method of coefficients
and leads to the number-theoretical identities of new type. The problem of
algebraic interpretation of new formulas is posed.%\medskip

\subsubsection{Thе algebraic statement of a problem}

\noindent

Let $R$ be a local ring with $2 \in R^{*}$. In \cite{Star2004} quadrics of a
projective space $RP_{n-1}$ associated to a free module of rank $n$ over $R$
were enumerated. Every invertible symmetric matrix over $R$ is congruent to
a diagonal matrix \cite{Mil1985}. Observe that the Mobius, Minkowski and
Laguerre's classical geometries have natural representations by projective
lines over algebras \cite{Benz73}. The fundamental theorem of projective
geometry over a field was extended to projective spaces over rings
\cite{O69}.

Consider the case of a local ring $R$ with main maximal ideal of nilpotent
step $s$, where $2\in R^{\ast }$ and $|R^{\ast }:R^{\ast2}|=2$.
The following enumerative formulas of classes of projectively congruence
quadric and symmetric forms of projective space $RP_{n-1}$ are given \cite{Lev2004,Star2004}.
Denote by $N(n,s)$ the number of all such classes. Then
\begin{equation*}
N(n,s)\!=\!\sum_{m=1}^{n}\!\sum_{q=1}^{\min (m,s)}\!
%\binom{s}{q}\!\left\{\binom{-1+\frac{m}{2}}{q-1}^{'}+
\binom{s}{q}\!\left\{\binom{-1+\frac{m}{2}}{q-1}'+
\!\sum_{(n_{1},\ldots ,n_{q})\in \Omega_{q}(m)}\left\lfloor \frac{1}{2}\prod_{j=1}^{q}(n_{j}+1)\right\rfloor \right\} ,
\end{equation*}%
\begin{equation}
{\ if}\,(1+R^{\ast 2})\subset R^{\ast 2},
\label{levs2}
\end{equation}%
\begin{equation*}
N(n,s)=\sum_{m=1}^{n}\sum_{q=1}^{\min (m,s)}\binom{s}{q}2^{q-1}
\left\{\binom{-1+\frac{m}{2}}{q-1}^{\prime }+\binom{m-1}{q-1}\right\} ,
\end{equation*}%
\begin{equation}
{\ if}\,R^{\ast }\cap (1+R^{2})\nsubseteq R^{\ast 2}.
 \label{levs1}
\end{equation}

Here $\binom{p}{q}^{\prime }$ is equal to $\binom{p}{q}$ for nonnegative
integers $p$ and $q$, and 0 otherwise; $\Omega _{q}(m)$ denotes the set of
all ordered partitions of the number $m$ in $q$ parts:
$n_{1}+\ldots+n_{q}=m $.

These formulas can be simplified using the integral
representation technique.%\medskip

\subsubsection{The analytic solution of a problem}

The main result of this paper is formulated in the following theorems.

\begin{theorem}.
\label{th1}
If $(1+R^{*2}) \subset R^{*2}$ then
\begin{equation}
\label{levs4}
N(n,s) = T(n,2s) + T(\left\lfloor \frac{n}{2} \right\rfloor,s),
\end{equation}
where
\begin{equation*}
T(n,p) = -\frac{1}{2} + \frac{1}{2}\binom{n+p}{p}.
\end{equation*}
\end{theorem}

\begin{theorem}.
\label{th2} If $R^{*} \cap (1+R^2) \nsubseteq R^{*2}$ then
\begin{equation}
\label{levs3}
N(n,s) = S(n,s) + S\left({\left\lfloor \frac{n}{2} \right\rfloor,\,s}\right),
\end{equation}
where
$$ S(n,s) = -\frac{1}{2}+\sum_{q=0}^{s} 2^{q-1}\binom{s}{q}\binom{n}{q}, $$
 and
for fixed $s$ and $n \rightarrow \infty$
\begin{equation}
\label{assym1}
S(n,s) \propto -\frac{1}{2} + {2}^{s-1} \frac{n^s}{s!},
\end{equation}
\begin{equation}
\label{assym2}
N(n, s) \propto -1 + \frac{1}{2}(2^s+1)\frac {n^s}{s!}
\end{equation}
\end{theorem}

The following section containing the proof of the theorem \ref{th1} in
particular solves difficult triple summation problem with the inner sum
taken by partitions and summand involving integer part function. The second
section gives a proof of theorem \ref{th2} including asymptotic solution of
the problem. Final remarks can be found in Conclusion.

{\textbf {Proof of the theorem }\,\ref{th1}}

%\begin{proof}{Theorem \ref{th1}}

Proof of formula (\ref{levs4}) for the sum (\ref{levs2}) splits into the
sequence of intermediate simpler statements of similar type, formulated in
lemmas \ref{lemma1}--\ref{lemma4}. In each of them we obtain with the help
of the method of coefficients and combinatorial techniques integral
representations for intermediate combinatorial expressions, that define sums
(\ref{levs2}). Write $N(n,s)$ in (\ref{levs2}) as
\begin{equation}
 \label{sp1}
N(n,s) = S_1(n,s) + S_2(n,s),
\end{equation}
where
\begin{equation}
\label{S1}
S_{1}(n,s) =\sum_{m=1}^{n}\sum_{q=1}^{\min ( m,s)}\binom{s }{q}\binom{-1+\frac{m}{2}}{q-1}^{\prime },
\end{equation}
\begin{equation}
\label{S2}
S_{2}(n,s) = \sum_{m=1}^{n} \sum_{q=1}^{\min(m,s)} \binom{s}{q} S_3(n,s,m,q),
\end{equation}
and
\begin{equation}
\label{S3}
S_{3}(n,s,m,q) =\frac{1}{2}\sum_{( n_{1},...,n_{q}) \in \Omega_{q}(m)}\left[\frac{1}{2}\prod_{j=1}^{q}(n_{j}+1)\right],
\end{equation}
where $\Omega_{q}(m)$ is a set of ordered sequences of integers
$n_1,\ldots,n_q$, such that $n_1+\ldots+n_q=m,\, n_{i}\geq 1$ for $i=1,\ldots,q$.

\begin{lemma}.
\label{lemma1}
\begin{equation}
\label{S4}
S_{1}(n, s) =-1+\binom{s+[\frac{n}{2}]}{s}.
\end{equation}
\end{lemma}

%\textbf{Proof.}
\begin{proof}
 We have
\begin{equation*}
S_{1}(n,s)=\sum_{k=1}^{[\frac{n}{2}]}\sum_{q=0}^{\min (2k,s)}\binom{s}{q}\binom{-1+k}{q-1}\!=
\!\sum_{k=0}^{[\frac{n}{2}]-1}\sum_{q=0}^{\min (2k+2,s)}\binom{s}{q}\binom{k}{k-q+1}=
\end{equation*}

\begin{equation*}
(using \,(\ref{M13})\, twice)
\end{equation*}

\begin{equation*}
=\sum_{k=0}^{[\frac{n}{2}]-1}\sum_{q=0}^{\min (2k+2,s)}\mbox{\bf res}_{x,y}\{(1+x)^{s}(1+y)^{k}/x^{q+1}y^{k-q+2}\}=
\end{equation*}%
\begin{equation*}
=\sum_{k=0}^{[\frac{n}{2}]-1}\mbox{\bf res}_{y}\left\{ \frac{(1+y)^{k}}{y^{k+2}}%
\left[ \sum_{q=0}^{\infty }y^{q}\mbox{\bf res}_{x}(1+x)^{s}/x^{q+1}\right]
\right\}.
\end{equation*}

In last expression using substitution rule within square brackets with the
change $x=y$ we have
\begin{equation*}
S_{1}(n,s)\!=\!\sum_{k=0}^{[\frac{n}{2}]-1}\mbox{\bf res}_{y}
\left\{ \left. \frac{(1+y)^{k}}{y^{k+2}}\cdot \left[ (1+x)^{s}\right] \right\vert _{x=y}\right\} =
\end{equation*}%
\begin{equation*}
=\sum_{k=0}^{[\frac{n}{2}]-1}\mbox{\bf res}_{y}\{(1+y)^{s}(1+y)^{k}/y^{k+2}\}=
\end{equation*}%
\begin{equation*}
=\mbox{\bf res}_{y}\left\{\sum_{k=0}^{[\frac{n}{2}]-1}(1+y)^{s}(1+y)^{k}/y^{k+2}\right\} =
\end{equation*}%
\begin{equation*}
\mbox{ (formula of the geometric progression in $k$)}
\end{equation*}%
\begin{equation*}
=\mbox{\bf res}_{y}\{(1+y)^{s}y^{-2}[1-(1+y)^{[\frac{n}{2}]}/y^{[\frac{n}{2}]}]/[1-(1+y)/y]\}=
\end{equation*}%
\begin{equation*}
=-\mbox{\bf res}_{y}(1+y)^{s}y^{-1}+\mbox{\bf res}_{y}(1+y)^{s+[\frac{n}{2}]}y^{[n/2]+1}=-1+\binom{s+[\frac{n}{2}]}{s}.
 \end{equation*}
\end{proof}
%$\hfill{\vrule height 3pt width 5pt depth 2pt}$
%\end{proof}
Now we will obtain an integral representation for sum $S_{3}(n, s, m, q)$
starting with its summand $\left[1/2\prod_{j=1}^{q}(n_{j}+1)\right]. $

Note that
$\Omega_{q}(m) =\Omega_{q}^{\prime}(m) \bigcup \Omega_{q}^{\prime\prime }(m)$,
 where
\begin{equation*}
\Omega_{q}^{\prime }(m)=\{(n_1,\ldots,n_q):
\end{equation*}
\begin{equation*}
n_1+...+n_q=m, n_{i} \geq 1, \,\text {for}\, i=1,\ldots,q, \text{and}%
\,\exists\, j\, \text {such that }\, n_j \text{is odd}\},
\end{equation*}
\begin{equation*}
\Omega_{q}^{\prime \prime }( m)=\{(n_1,\dots,n_q) : n_1+...+n_q=m,\, n_{i}\geq 1\,
 \text{for}\, i=1,\ldots,q, \text {and} n_{j} \text{ is even}\, \forall
\end{equation*}
and $\Omega_{q}^{\prime }(m) \bigcap \Omega_{q}^{\prime\prime}(m) =
\emptyset $. Then
\begin{equation}
\label{prod}
\left[\frac{1}{2}\prod_{j=1}^{q}(n_{j}+1)\right]=
\begin{cases}
\frac{1}{2}\prod_{j=1}^{q}(n_{j}+1), & \text{if }n\in \Omega_{q}^{\prime}(m), \\
\frac{1}{2}(-1+\prod_{j=1}^{q}(n_{j}+1)), & \text{ if }n\in \Omega_{q}^{\prime
\prime}(m).%
\end{cases}%
\end{equation}

\begin{lemma}.
\label{lemma1a}
The following equalities are valid
\begin{equation}
\label{A}
\sum_{n\in \Omega_{q}(m)}[1/2\prod_{j=1}^{q}(n_{j}+1)]=
\frac{1}{2}\sum_{n\in \Omega_{q}(m)}\prod_{j=1}^{q}(n_{j}+1)-\frac{1}{2}|\Omega_{q}^{\prime \prime}(m)|,
\end{equation}
\begin{equation}
\label{B}
\!\!\sum_{n\in \Omega_{q}(m)}\prod_{j=1}^{q}(n_{j}+1)=\mbox{\bf res}%
_{x}\{f^{q}(x) /x^{m+q+1}\}= \mbox{\bf res}_{x}\{\frac{(-1+(1-x)^{-2})^{q}}{%
x^{m+1}}\},
\end{equation}
\begin{equation}
\label{C}
|\Omega_{q}^{\prime \prime }( m) |=\mbox{\bf res}_{x}%
\{(1-x^{2})^{-q}/x^{m-2q+1}\}=
\begin{cases}
\binom{\frac{m}{2}-1 }{q-1}, & \text{ if }\,m \,\text{ is even,} \\
0, & \text{if}\,m \text{ is odd.}%
\end{cases}%
\end{equation}
\end{lemma}

%\textbf{Proof}.
\begin{proof}
Using (\ref{prod}) we have
\begin{equation*}
\sum_{n\in \Omega_{q}(m)}[\frac{1}{2}\prod_{j=1}^{q}(n_{j}+1)]=
\sum_{n\in\Omega_{q}^{\prime}(m)}[\frac{1}{2}\prod_{j=1}^{q}(n_{j}+1)] +
\sum_{n\in\Omega_{q}^{\prime \prime}(m) }[\frac{1}{2}\prod_{j=1}^{q}(n_{j}+1)]=
\end{equation*}
\begin{equation*}
=\sum_{n\in \Omega_{q}^{\prime}(m)} \frac{1}{2}\prod_{j=1}^{q}(n_{j}+1)+
\sum_{n\in\Omega_{q}^{\prime \prime}(m)}\frac{1}{2}(-1+\prod_{j=1}^{q}(n_{j}+1))=
\end{equation*}
\begin{equation*}
=\frac{1}{2}\sum_{n\in
\Omega_{q}(m)}\prod_{j=1}^{q}(n_{j}+1)-\frac{1}{2}|\Omega_{q}^{\prime\prime}(m) |,
\end{equation*}
which gives (\ref{A}).

Proof of (\ref{B})--(\ref{C}) uses combinatorial properties of generating
functions of partitions of the set $\Omega_{q}(m) $.
Denote $c_{n}=(n+1)\geq2 $. Since the generating function for sequence $\{c_n\}$ is
\begin{equation*}
f(x)=\sum\limits_{n=1}^{\infty }(n+1)x^{n+1}=
-x+\sum\limits_{n=0}^{\infty}(n+1)x^{n+1}=
\end{equation*}
\begin{equation*}
=-x+x[(1-x)]^{^{\prime }}=-x+x(1-x)^{-2}=x\left(-1+(1-x)^{-2}\right),
\end{equation*}%
then we get
\begin{equation*}
\sum_{n\in \Omega_{q}(m)}\prod_{j=1}^{q}(n_{j}+1)=
\mbox{\bf res}_{x}\{f^{q}(x)/x^{m+q+1}\}=
 \mbox{\bf res}_{x}\{(-1+(1-x)^{-2})^{q}/x^{m+1}\}.
\end{equation*}

It is easy to see that
\begin{equation*}
|\Omega_{q}^{\prime \prime}(m)|=\mbox{\bf res}_{x} \{ R^{q}(x) /x^{m+q+1} \},
\end{equation*}
where $R^{q}(x)=x^3+x^5+\ldots...=x^3(1-x^2)^{-1}. $ Thus
\begin{equation*}
|\Omega_{q}^{\prime \prime }(m)|=
\mbox{\bf res}_{x}\{[x^3(1-x^2)^{-1}]^{q}/x^{m+q+1}\}=
\mbox{\bf res}_{x}\{[(1-x^{2})^{-q}/x^{m-2q+1}\},
\end{equation*}%
which is equivalent to (\ref{C}).
\end{proof}
\begin{lemma}.
\label{lem5} Let
\begin{equation}
\label{412}
S_{4}(n, s) = \frac{1}{2}\sum_{m=1}^{n}\sum_{q=1}^{\min(m, s)}\binom{s }{q}
\sum_{(n_{1},\ldots,n_{q}) \in \Omega_{q}(m)}\prod_{j=1}^{q}(n_{j}+1).
\end{equation}
Then
\begin{equation}
\label{412a}
S_{4}(n, s)=-\frac{1}{2}+\frac{1}{2}\binom{2s+n}{n}.
\end{equation}
\end{lemma}%
\begin{proof}

Since
\begin{equation*}
\binom{s}{q}=\mbox{\bf res}_{z}\{(1+z)^{s}/z^{q+1}\}
\end{equation*}%
from (\ref{412}) and (\ref{B})we have
\begin{equation*}
S_{4}(n,s)=\frac{1}{2}\sum_{m=1}^{n}\sum_{q=0}^{\min (m,s)}{\mbox{\bf res}}_{z}\left\{\frac{(1+z)^{s}}{z^{q+1}}\right\}
\mbox{\bf res}_{x}\left\{\frac{\left(-1+(1-x)^{-2}\right)^{q}}{x^{m+1}}\right\} =
\end{equation*}%
\begin{equation*}
(the \, substitution\, rule \, of \, the\, index\, q,\,the\, change\,z=x^{2}/(1-x^{2}))
\end{equation*}
\begin{equation*}
=\frac{1}{2}\sum_{m=1}^{n}{\mbox{\bf res}}_{x}\left\{\left(1+(1+(1-x)^{2})\right)^{s}/x^{m+1}\right\}=
\end{equation*}
\begin{equation*}
=\frac{1}{2}\sum_{m=1}^{n}\mbox{\bf res}_{x}\{(1-x)^{-2s}/x^{m+1}\}
=\frac{1}{2}\mbox{\bf res}_{x}\{\sum_{m=1}^{n}(1-x)^{-2s}/x^{m+1}\}=
\end{equation*}%
\begin{equation*}
(\mbox{summation by the index }m)=
\end{equation*}%
\begin{equation*}
=\frac{1}{2}{\mbox{\bf res}}_{x}\{(1-x)^{-2s}x^{-2}(1-x^{-n}/(1-x^{-1})\}=
\end{equation*}%
\begin{equation*}
=-\frac{1}{2}{\mbox{\bf res}}_{x}\{(1-x)^{-2s-1}x^{-1}(1-x^{-n})\}=-\frac{1}{2}+\frac{1}{2}\binom{2s+n}{n}.
\end{equation*}
\end{proof}
%$\hfill {\vrule height3ptwidth5ptdepth2pt}$

\begin{lemma}.
\label{lemma4}
Let
\begin{equation}
S_{5}(n,s):=-\frac{1}{2}\sum_{m=1}^{n}\sum_{q=1}^{\min (m,s)}\binom{s}{q}|\Omega_{q}^{\prime \prime }(m)|.
\label{414}
\end{equation}%
Then%
\begin{equation}
S_{5}(n,s)=\frac{1}{2}-\frac{1}{2}\binom{2s+[n/2]}{2s}.
\label{414a}
\end{equation}
\end{lemma}
\begin{proof}
%\textbf{Proof.}
Since
$$\binom{s}{q}={\mbox{\bf res}}_{z}\{(1+z)^{s}/z^{q+1}\}$$
from (\ref{414}) and (\ref{C}) we have
\begin{equation*}
S_{5}(n,s)=-1/2\sum_{m=1}^{n}\sum_{q=0}^{\min (m,s)}
{\mbox{\bf res}}_{z}\{(1+z)^{s}/z^{q+1}\}\mbox{\bf res}_{x}\{(1-x^{2})^{-q}/x^{m-2q+1}\}=
\end{equation*}
\begin{equation*}
(\text{the\, substitution\,of\,the\,  rule\, of\, the\, index\, q})
%(the\, substitution\, rule\, of\, the\, index\, q\,)
\end{equation*}
\begin{equation*}
=-\frac{1}{2}\sum_{m=1}^{n}{\mbox{\bf res}}_{x}\{\left(1+(1-x^{2})^{-1}x^{2}\right)^{s}/x^{m+1}\}=
\end{equation*}%
\begin{equation*}
=-\frac{1}{2}\sum_{m=1}^{n}{\mbox{\bf res}}_{x}\{(1-x^{2})^{-s}/x^{m+1}\}=
\end{equation*}%
\begin{equation*}
=-\frac{1}{2}{\mbox{\bf res}}_{x}\{\sum_{m=1}^{n}(1-x^{2})^{-s}/x^{m+1}\}=
(\mbox{a\, summation by the index }m)=
\end{equation*}%
\begin{equation*}
=-\frac{1}{2}{\mbox{\bf res}}_{x}\{(1-x^{2})^{-s}x^{-2}(1-x^{-n})/(1-x^{-1})\}=
\end{equation*}%
\begin{equation*}
=\frac{1}{2}{\mbox{\bf res}}_{x}\{(1-x^{2})^{-s-1}(1-x^{-n})(1+x)x^{-1}\}=
\end{equation*}%
\begin{equation*}
=\frac{1}{2}-\frac{1}{2}{\mbox{\bf res}}_{x}(1-x^{2})^{-s-1}x^{-n-1}-\frac{1}{2}{\mbox{\bf res}}%
_{x}(1-x^{2})^{-s-1}x^{-n}.
\end{equation*}%
Since%
\begin{equation*}
-\frac{1}{2}{\mbox{\bf res}}_{x}(1-x^{2})^{-s-1}x^{-n-1}=%
\begin{cases}
-\frac{1}{2}\binom{s+[\frac{n}{2}]}{s}, & \text{{if} }n\text{ is even;} \\
0, & \text{if }\,n\text{ is odd,}%
\end{cases}%
\end{equation*}%
\begin{equation*}
-\frac{1}{2}{\mbox{\bf res}}_{x}(1-x^{2})^{-2s-1}x^{-n}=%
\begin{cases}
-\frac{1}{2}\binom{s+[\frac{(n-1)}{2}]}{s}, & \text{{if} }n\text{ is odd;} \\
0, & \text{{if} }n\text{ is even,}%
\end{cases}%
\end{equation*}%
we have%
\begin{equation*}
S_{5}(n,s)=\frac{1}{2}-\frac{1}{2}\binom{s+[\frac{n}{2}]}{s}.
\end{equation*}%
\end{proof}
%\medskip $\hfill {\vrule height3ptwidth5ptdepth2pt}$

\textbf{Proof of the theorem} \ref{th2}\medskip

Analogously to the theorem \ref{th1}, formula (\ref{levs3}) for the sum
(\ref{levs1}) is immediately corollary of formulas (\ref{ssum6}), (\ref{ssum7})
and (\ref{ssum}).

Formula (\ref{assym1}) for the sum $S(n, s)$ follows from
the classical tauberian theorem (see, for example, \cite{Hardy}): if the
series $A(z) = a_0 + \sum_{k=1}^{\infty} a_k z^k $ converges for $|z|<1$,
the limit $\lim_{z \rightarrow 1} (1-z)^{s+1} A(z) = B $ exists for some
$s\geq 0 $, and $k (a_k-a_{k-1}) >-c (k=1,2,\ldots)$ for some positive
constant $c $, then
$\lim_{k \rightarrow \infty} a_k k^{-s} = B/\Gamma(s+1).$

Formula (\ref{assym2}) follows from (\ref{assym1}) and (\ref{levs3}).

Write $N(n,s)$ in (\ref{levs1}) as%
\begin{equation}
N(n,s)=S_{6}(n,s)+S_{7}(n,s),
\label{sp11}
\end{equation}%
where%
\begin{equation*}
S_{6}(n,s)=\sum_{m=1}^{n}\sum_{q=1}^{\min (m,s)}\binom{s}{q}2^{q-1}\,
\binom{m-1}{q-1},
\end{equation*}%
and%
\begin{equation*}
S_{7}(n,s)=\sum_{m=1}^{n}\sum_{q=1}^{\min (m,s)}\binom{s}{q}2^{q-1}
\binom{-1+m/2}{q-1}^{\prime }.
\end{equation*}

\begin{lemma}.
 \label{lem6a} Let the sum
\begin{equation}
S(n,s)=-\frac{1}{2}+\sum_{q=0}^{s}2^{q-1}\binom{s}{q}\binom{n}{q}.
\label{ssum}
\end{equation}%
Then the sum $S_{6}(n,s)$ has the following integral representations%
\begin{equation*}
S_{6}(n,s)=-\frac{1}{2}+\frac{1}{2}\mbox{\bf res}_{y}(1+2y)^{s}(1+y)^{n}y^{-n-1}=
\end{equation*}%
\begin{equation}
=-\frac{1}{2}+\frac{1}{2}\mbox{\bf res}_{z}(1+z)^{s}(1-z)^{-s-1}z^{-n-1},
\label{intrep}
\end{equation}%
\begin{equation}
S_{6}(n,s)=S(n,s),
\label{ssum6}
\end{equation}%
and
\begin{equation}
\label{ssum7}
S_{7}(n,s)=S\left([\frac{n}{2}],s\right).
\end{equation}
\end{lemma}
\begin{proof}
%\textbf{Proof.}
We have
\begin{equation*}
S_{6}(n,s)=\sum_{m=1}^{n}\sum_{q=1}^{\min (m,s)}\binom{s}{q}2^{q-1}
\binom{m-1}{q-1}=\sum_{m=1}^{n}\sum_{q=0}^{\min (m,s)}\binom{s}{q}2^{q-1}\binom{m-1}{m-q}=
\end{equation*}
\begin{equation*}
(using\text{ }(\ref{M13})\text{ }twice)
\end{equation*}
\begin{equation*}
=\sum_{m=1}^{n}\sum_{q=0}^{\min (m,s)}2^{q-1}\mbox{\bf res}_{x,y}\{(1+x)^{s}(1+y)^{m-1}/x^{q+1}y^{m-q+1}\}=
\end{equation*}%
\begin{equation*}
(the\text{ }summation\text{ }with\text{ }respect\, to\, the\text{ }index%
\text{ }q,\text{ }substitution\text{ }rule\text{ }and\text{ }the\text{ }%
\end{equation*}%
\begin{equation*}%
change\text{ }x=2y\in L_{1})
\end{equation*}%
\begin{equation*}
=\frac{1}{2}\sum_{m=1}^{n}\mbox{\bf res}_{y}\{(1+2y)^{s}(1+y)^{m-1}/y^{m+1}\}=
\end{equation*}%
\begin{equation*}
\frac{1}{2}\mbox{\bf res}_{y}\{\sum_{m=1}^{n}(1+2y)^{s}(1+y)^{m-1}/y^{m+1}\}=
\end{equation*}%
\begin{equation*}
(the\text{ }summation\text{ }by\text{ }the\text{ }index\text{ }m,\text{ }%
formula\text{ }of\text{ }the\text{ }geometric\text{ }progression)
\end{equation*}%
\begin{equation*}
=\frac{1}{2}\mbox{\bf res}_{y}\{(1+2y)^{s}y^{-2}[1-(1+y)^{n}/y^{n}]/[1-(1+y)/y]\}=
\end{equation*}%
\begin{equation*}
=-\frac{1}{2}\mbox{\bf res}_{y}(1+2y)^{s}y^{-1}+1/2\,\mbox{\bf res}%
_{y}(1+2y)^{s}(1+y)^{n}y^{-n-1}=
\end{equation*}%
\begin{equation*}
=-\frac{1}{2}+\frac{1}{2}\mbox{\bf res}_{y}(1+2y)^{s}(1+y)^{n}y^{-n-1}.
\end{equation*}%
Further%
\begin{equation*}
S_{6}(n,s)=-\frac{1}{2}+\frac{1}{2}\mbox{\bf res}_{y}(1+2y)^{s}(1+y)^{n}y^{-n-1}=
\end{equation*}%
\begin{equation*}
=-\frac{1}{2}+\sum_{q=0}^{s}2^{q-1}\binom{s}{q}\mbox{\bf res}_{y}%
\{(1+y)^{n}/y^{n-s+1}\}=
\end{equation*}%
\begin{equation*}
=-\frac{1}{2}+\sum_{q=0}^{s}2^{q-1}\binom{s}{q}\binom{n}{q},
\end{equation*}%
and%
\begin{equation*}
S_{6}(n,s)=-\frac{1}{2}+\frac{1}{2}\mbox{\bf res}_{y}(1+2y)^{s}(1+y)^{n}y^{-n-1}.
\end{equation*}

The second formula in (\ref{intrep}) is obtained from the first one with the
change of variables $z=y/(1+y)$ by \textbf{Rule 3}.

The proof of (\ref{ssum}) is analogous to the proof of (\ref{ssum6}).
\end{proof}

\begin{corollary}. In notions of (\ref{levs1}) the following combinatorial identity is valid%
\begin{equation*}
\sum_{k=1}^{[\frac{n}{2}]}\sum_{q=1}^{\min (2k,s)}\binom{s}{q}2^{q-1}
\binom{-1+k}{q-1}+\sum_{m=1}^{n}\sum_{q=1}^{\min (m,s)}2^{q-1}\binom{s}{q}\binom{m-1}{q-1}=
\end{equation*}%
\begin{equation*}
=-1+\sum_{q=0}^{s}2^{q-1}\binom{s}{q}\left( \binom{n}{q}+\binom{[\frac{n}{2}]}{q}\right).
\end{equation*}%
\medskip
\end{corollary}

\textbf{Conclusion}

A simplification of formulas usually brings new information on the structure
of objects of enumeration.

For example, simplification of known formulas for
$R_{q}^{3}(n) $ from Sokolov (1969) allowed to understand better
the structure of the enumerable regular words (commutators) on known
Shirshov bases of a free Lie algebra. This allows to solve the Kargapolov
problem of computing the ranks $R_{q}^{k}(n) $ of the factors
for the lower central series of a free solvable group of step $k$ with $q$
generators for arbitrary $k$ (Egorychev, 1972). This, in turn, allowed to
solve analogous problem for a free polynilpotent group (Gorchakov and
Egorychev, 1972) and for free groups in varieties (Egorychev, 1977). Another
answer to the same problem of Kargapolov was suggested in Petrogradsky in
1999. The reader can find more detailed statement of this question and
corresponding references to the literature in (\cite{Egor1}, pp. 129--132
and 207--222).

Simplicity found in formulae of Theorems 1 and 2 poses in
\cite{EgorZima(2008)}\ the following problem: \textit{Give an independent
algebraic proof and interpretation of formulas (\ref{levs4}) and
(\ref{levs3}) for the number of quadrics on the projective space }
$RP_{n-1}$\textit{, }$n>2$\textit{, over local ring }$R$
\textit{\ with the maximal ideal nilpotent of a class }$s$\textit{.}

The full answer to this problem was given O.V. Starikova and A.V.
Svistunova A.V. in the article (2011, \cite{StSv11}), in which the by algebraic
proof (interpretation) of identities \textit{(\ref{levs4}) and (\ref{levs3})
}is given in the theory of local rings of the specified type. This, in turn
O.V. Starikova tried understood better the structure of studied
objects and to solve more difficult enumeration problems for various classes
of projective equivalent quadrics over local rings (2013, \cite{St2013}).

%\section{\S 3 The enumeration of $\mathcal{D}$-invariant ideals of ring $ R_{n}(K,J)$
%on lattices}

\section{ The enumeration of $\mathcal{D}$-invariant ideals of ring
$R_{n}(K,J)$ on lattices}

\subsection {The combinatorial statement of a problem}

In the description of ring ideals $R_{n}(K,J)$ uses the definition of
$T$-border $A,\ A=A(T;\mathcal{L},\mathcal{L}^{\prime })$ (\cite{KuzLev2000}, Definition 2.1).
The $T$-border depends on the $J$-submodule $T$ in $K$ and
the pair of sets of matrix elements:%
\begin{equation*}
\mathcal{L}=\left\{(i_1, j_1),\,(i_2, j_2),...,(i_r, j_r) \right \},r\geq 1,
\end{equation*}
\begin{equation}
1 \leq j_{1}<j_{2}<...<j_{r}\leq n,\,\, 1\leq i_{1}<i_{2}<...<i_{r}\leq n;
\label{M8}
\end{equation}%
\begin{equation*}
\mathcal{L}^{\prime }=\left\{ (1,j_{r}),\, (k_{1}, m_{1}),\,
%\mathcal{L}'=\left\{ (1,j_{r}),\, (k_{1}, m_{1}),\,
(k_{2},m_{2}),\, ...\ ,(k_{q},m_{q}),\ (i_{1},n)\right\} ,\ q\geq 0,
\end{equation*}%
\begin{equation}
j_{r} \leq m_{1}<m_{2}<...<m_{q}\leq n,\, 1\leq k_{1}<k_{2}<...<k_{q}\leq i_{1}.
\label{M9}
\end{equation}%
We define the pair $(\mathcal{L},\mathcal{L}^{\prime })$ as "the set of degree angles $n$".

Let $\mathcal{L}(i,j),\, i,\,j\in \overline{1,n},$ be the set of all sequences
of type $\mathcal{L}$ of arbitrary length $r,\, r\geq 1,$ where
$i_{1}=i,\,j_{r}=j,$ and $\mathcal{L}^{\prime }(i,j)$, $i,j\in \overline{1,n},$ is
the set of all sequences of $\mathcal{L}^{\prime }$ type (including the
empty set) of any length $q,\,q\geq 0.$ It determined by the initial
conditions $i_{1}=i,\,j_{r}=j$.

In the article \cite{KuzLev2000} it is shown

\textbf{Theorem \cite{KuzLev2000} }

{\it{Let $ n\in N.$  Then the number
$N(R)$ of all ideals of the ring $R_{n}(K,\, J) $ is equal to
\[
{\Omega }(n)=\left| L\times L'\right|=
\sum_{i=1}^{n}\sum_{j=1}^{n}|\mathcal{L}(i,j)|\cdot |\mathcal{L}^{\prime }(i,j)|.
\]%
}}

Then we prove %Далее мы докажем
\begin{theorem}
\label{th3}
%\textbf{Theorem.}
The number of all ideals of the ring $R_{n}(K,\, J) $ is equal to
\begin{equation}
\label{M1a}
\Omega (n)=
%\sum_{i=1}^{n}\sum_{j=1}^{n}|\mathcal{L}(i,j)|\cdot |\mathcal{L}^{\prime }(i,j)|=
(2n-1)\binom{2n-2 }{n-1}.
\end{equation}%
\end{theorem}

Then $\Omega (n)$ is the number of sets of angles
$(\mathcal{L},\mathcal{L}^{\prime })$ of degree $n$ , and
$\Omega ^{+}(n)$ is the number of all sets
angles $(\mathcal{L},\mathcal{L}^{\prime })$ of degree
$n$ with $i>j$ for all $(i,j)\in \mathcal{L}.$

Note that an each sequence$ \{(i_{1},j_{1}),\,(i_{2},j_{2})
,...(i_{r},j_{r})\}$ of  $\mathcal{L}$\ type (analoqously of
 $\mathcal{L}'$ type) of length $r,\,r\geq 1,$ is one-to-one correspondence
of \ the increasing path $(i_{1},j_{1}),\,( i_{2},j_{2})
,...,(i_{r},j_{r})$\ with the\ $(r-1)$-diagonal steps on a
rectangular lattice from the point $(i_{1}, j_{1}) $ to the
point $(i_{r},j_{r}) $, see Fig. 1.

\bigskip According to (\cite{KuzLev2000}, Theorem 2.2), for a strongly
maximal ideal $J$ of $K$ any ideal of the ring $R_{n}(K,J)(n\geq 2)$ is
generated by unique $T$-boundary $A$. Regard the $n\times n$ matrix as a
square array of $n^{2}$ points $(i,j)$ (matrix positions). Associating
"staircases" with $\mathcal{L}$ and $\mathcal{L}'$ we can compare
with the ideal generated by the $T$-boundary $A$ the matrix

\begin{center}
%\bc
\setlength{\unitlength}{5mm}
\begin{picture}(12,10)
\thicklines
\put(-.1,4){\makebox(0,0){$\left(\rule[-2cm]{0pt}{4.5cm}\right.$}}
\put(11.1,4){\makebox(0,0){$\left.\rule[-2cm]{0pt}{4.5cm}\right)$}}
\put(-1,4){\makebox(0,0){$i_1$}} \put(-1,3){\makebox(0,0){$i_2$}}
\put(-1,2){\makebox(0,0){$\vdots$}}
\put(-1,1){\makebox(0,0){$i_r$}} \put(1,-1){\makebox(0,0){$j_1$}}
\put(2,-1){\makebox(0,0){$j_2$}}
\put(3,-1){\makebox(0,0){$\cdots$}}
\put(5,-1){\makebox(0,0){$j_r$}} \put(6,9.5){\makebox(0,0){$m_1$}}
\put(7,9.5){\makebox(0,0){$m_2$}}
\put(8,9.5){\makebox(0,0){$\cdots$}}
\put(10,9.5){\makebox(0,0){$m_q$}}
\put(12,8){\makebox(0,0){$k_1$}} \put(12,7){\makebox(0,0){$k_2$}}
\put(12,6){\makebox(0,0){$\vdots$}}
\put(12,5){\makebox(0,0){$k_q$}} \put(0,4){\line(1,0){1}}
\put(1,4){\line(0,-1){1}} \put(1,3){\line(1,0){1}}
\put(2,3){\line(0,-1){1}} \put(5,0){\line(0,1){1}}
\put(5,1){\line(-1,0){1}} \put(5,9){\line(0,-1){1}}
\put(5,8){\line(1,0){1}} \put(6,8){\line(0,-1){1}}
\put(6,7){\line(1,0){1}} \put(7,7){\line(0,-1){1}}
\put(9,5){\line(1,0){1}} \put(10,5){\line(0,-1){1}}
\put(10,4){\line(1,0){1}} \put(1,1){\makebox(0,0){$(T)$}}
\put(.5,4.5){\makebox(0,0){$\cal L$}}
\put(8,2){\makebox(0,0){$(JT)$}}
\put(8.5,5.5){\makebox(0,0){$\cdots$}}
\put(9.5,5.5){\makebox(0,0){${\cal L}'$}}
\put(10,7){\makebox(0,0){$(J^2T)$}} \thinlines
\multiput(1.1,4)(.5,0){18}{\line(1,0){.3}}
\multiput(1.4,3.95)(.5,0){17}{.}
\multiput(5,1.1)(0,.5){14}{\line(0,1){.3}}
\multiput(4.9,1.45)(0,.5){13}{.}
\end{picture}
\\[1cm]
Figure 1.
\end{center}

We now distinguish $\mathcal{D}$-invariant ideals of the ring $R_{n}(K,J)$.

For any additive subgroup $F$ of $K$ we denote by $N_{ij}(F)$ (resp. by
$Q_{ij}(F)),$ the additive group of $R_{n}(K,J)$ generated by sets $Fe_{km}$
for all $(k,m)\geq (i,j)$ (resp. $(k,m)>(i,j))$ where $(i,j)<(k,m)$, if
$i\leq k,\;m\leq j$ and $(i,j)\neq (k,m)$.

\subsubsection{The enumeration of a number of ways in a rectangular lattice
and the method of inclusion-conclusion }

%The following theorem prooves \cite{DavEgLev}{Theorеm 3},
%[ \cite{DavEgLev}{Theorm 3}, Theorm 3

\begin{theorem}
\label{teo1}
For numbers $\Omega^{+}(n),$ $n=3,4,\ldots $ the following
combinatorial formula is valid:
\begin{equation*}
\Omega^{+}(n)=\sum_{i=2}^{n}\sum_{j=1}^{i-1}\left\{\sum_{r\geq1}\left\{\sum_{k_{1}=0}^{r-1}
\sum_{k_{2}=0}^{k_{1}}\sum_{s=r-k_{2}-1}^{2r-k_{1}-k_{2}-2} \binom{i-1}{k_{1}}\binom{j-i}{s}
\binom{n-j}{k_{2}} \times \right.\right.
\end{equation*}
\begin{equation*}
\times\binom{s}{2r-s-k_{1}-k_{2}-2}\cdot \frac{(k_{1}-k_{2}+1)}{(s-r+k_{1}+2)} +
\end{equation*}
\begin{equation*}
+\sum_{k_{1}=0}^{r-1}\sum_{k_{2}=0}^{k_{1}}%
\sum_{s=r-k_{2}-1}^{2r-k_{1}-k_{2}-2} \binom{i-1}{k_{1}}\binom{j-i}{s}
\binom{n-j}{k_{2}}\binom{s}{2r-s-k_{1}-k_{2}-2} \times
\end{equation*}
\begin{equation*}
\left.\left. \times \frac{(k_{2}-k_{1}+1)}{(s-r+k_{2}+2)}\right\}\binom{%
2s-2r+k_1+k_2+2}{s-r+k_2+1} \right\}\cdot\binom{i-1+n-j}{i-1}+
\end{equation*}
\begin{equation*}
+\sum_{i=2}^{n}\sum_{j=i}^{n-1}\left\{\sum_{r\geq1}\left\{%
\sum_{k_{1}=0}^{r-1} \sum_{k_{2}=0}^{k_{1}}\sum_{s=r-k_2-1}^{2r-k_1-k_2-2}
\binom{i-1}{k_{1}}\binom{j-i}{s}\binom{n-j}{k_{2}} \right.\right. \times
\end{equation*}
\begin{equation*}
\times\binom{s}{2r-s-k_{1}-k_{2}-2}\cdot\frac{(k_1-k_2+1)}{(s-r+k_1+2)}+
\end{equation*}

\begin{equation*}
+\sum_{k_{1}=0}^{r-1}\sum_{k_{2}=0}^{k_{1}}%
\sum_{s=r-k_{2}-1}^{2r-k_{1}-k_{2}-2} \binom{i-1}{k_{1}}\binom{j-i}{s}
\binom{n-j}{k_{2}}\binom{s}{2r-s-k_1-k_2-2}\times
\end{equation*}

\begin{equation}
\label{M10}
\left. \left. \times \frac{(k_2-k_1+1)}{(s-r+k_2+2)}\right\}
\cdot\binom{2s-2r+k_1+k_2+2}{s-r+k_1+1}\right\}\cdot \binom{i-1+n-j}{i-1}
\end{equation}
\end{theorem}

For the proof of the theorem we give a number of definitions, lemmas and theorems.
%Для доказательства теоремы приведем ряд определений, лемм и теорем.

Let $\overline{\mathcal{L}}^{(n)}(i,j)$ be the set of all sequences
$\{(i_{1},j_{1}),\{(i_{2},j_{2}),\ldots ,\{(i_{r},j_{r})\},$ $r\geq 1$, of
%type (\ref{staircase-1}) which $i_{1}=i,$ $j_{r}=j$, and $j_{t}\leq i_{t}$
type (\ref{M8}) which $i_{1}=i,$ $j_{r}=j$, and $j_{t}\leq i_{t}$
from all $t=0,1,\ldots ,r$.

Let $\overline{\overline{\mathcal{L}}}%
_{r}^{(n)}(i,j)$ be the set of all sequences$\{\{(i_{1},j_{1}),%
\{(i_{2},j_{2}),\ldots ,\{(i_{r},j_{r})\}$, $r\geq 1$, of type (\ref{M8})
 by fixed $r,\,r\geq 1$.

Let $\overline{\overline{\mathcal{L}}%
}_{r}^{(n)}(i,j)$ be the set of all sequences $\{(i_{1},j_{1}),%
\{(i_{2},\,j_{2}),\ldots ,\{(i_{r},\,j_{r})\}$ of type (\ref{M8}) by
%fixed $r,\,r\geq 1$, of type (\ref{staircase-1}) which $i_{1}=i,\,j_{r}=j$, and
fixed $r,\,r\geq 1$, of type (\ref{M9}) which $i_{1}=i,\,j_{r}=j$, and
$j_{t}\leq i_{t}$ from all $t=0,1,\ldots ,r$. It is easy to note that%
\begin{equation}
\overline{\mathcal{L}}^{(n)}(i,j)=
\overline{\overline{\mathcal{L}}}^{(n-1)}(i-1,j)=
\label{1}
\end{equation}%
\begin{equation}
=\bigcup_{r\geq 1}\overline{\overline{\mathcal{L}}}_{r}^{(n-1)}(i-1,r).
\label{2}
\end{equation}%
By the formulas (\ref{M8}), (\ref{M9}) and (\ref{2}) %(Lemma 5$_{1}$)
 we have
\[
\Omega^{+}(n)=\sum_{i,j=1}^{n}|\overline{\mathcal{L}}^{(n)}(i,j)|\cdot
\left|\mathcal{L}'(i,\,j)\right| =
\]%
\[
=\sum_{i=2}^{n}\sum_{j=1}^{n-1}|\overline{\overline{\mathcal{L}}}^{(n-1)}(i-1,j)| \cdot \binom{i-1+n-j}{i-1}=
\]%
\begin{equation*}
=\sum_{i=2}^{n}\sum_{j=1}^{i-1}\left|\overline{\overline{\mathcal{L}}}^{(n-1)}(i-1,j)\right|
\cdot\binom{i-1+n-j}{i-1}+
\end{equation*}
\begin{equation}
+\sum_{i=2}^{n}\sum_{j=i}^{n-1}\left|\overline{\overline{\mathcal{L}}}^{(n-1)}(i-1,j)\right|
\cdot \binom{i-1+n-j}{i-1}.
\label{3}
\end{equation}

Let $A=\{j_{1},j_{2},\ldots ,j_{r-1},j\}$, $B=\{i,i_{2},\ldots ,i_{r}\},$ be
two sequences of positive integers of type $(1)$ which $1\leq i_{1}=i-1,$ $%
j_{r}=j\leq n-1$ and $j_{t}\leq i_{t}$ from all $t=0,1,\ldots ,r$.

Let us consider two cases: $n-1\geq i-1\geq j\geq 0$ or $0\leq i-1<j\leq n-1.$.

\textit{1 case}: $n-1\geq i-1\geq j\geq 0$.

%\textbf{Lemma 1.} \textit{If }$i-1\geq j$, $i=1,2,\ldots ,n$, \textit{then the following formula is valid}%
\begin{lemma}
\label{Lmma1}
If $i-1\geq j$, $i=1,2,\ldots ,n$, then
the following formula is valid
\begin{equation}
\left|\overline{\overline{\mathcal{L}}}^{(n-1)}(i-1,j)\right|=\binom{n-i+j-1}{j-1}.
\label{4}
\end{equation}
\end{lemma}

%\textit{Proof.} As $i-1\geq j$ then $A\cap B=\varnothing $ and we have
\begin{proof}

 As $i-1\geq j$ then $A\cap B=\varnothing $ and we have

\[
\left|\overline{\overline{\mathcal{L}}}_{r}^{(n-1)}(i-1,j)\right|=
\binom{n-i}{r-1}\binom{j-1}{r-1}.
\]
%By the formula (\ref{staircase-2}) we get
By the formula for (\ref{2}) we get
\[
\left|\overline{\overline{\mathcal{L}}}^{(n-1)}(i-1,j)\right|
=\sum_{r\geq 1}\left|\overline{\overline{\mathcal{L}}}_{r}^{(n-1)}(i-1,j)\right|=\sum_{r\geq 1}\binom{n-i}{r-1}
\binom{j-1}{r-1}=
\]
\[
=\binom{n-i+j-1}{j-1}.
\]

\textit{2 case}: $0\leq i-1<j\leq n-1.$
\end{proof}

We use the known combinatorial result %(\textit{Feller, Ch. 3, paragraph 1,Theorem 1})
which in terms of rectangular paths in a rectangular lattice may
be formulated as

%\textbf{Lemma 2} ({\it{Feller, Ch. 3, paragraph 1, Theorem 1}}).
\begin{lemma} ({\it{Feller, Ch. 3, paragraph 1, Theorem 1}}).
\label{Lmma1a}
%{\it{The number of all increasing rectangular paths in a rectangular lattice }}
%({\it{without diagonal steps}}) \textit{from the origin to the point}
%$(X,Y),\, X\geq Y$, {\it{is equal to}}%
The number of all increasing rectangular paths in a rectangular  lattice ( without diagonal steps)
from the origin to the point $(X,Y),\, X\geq Y$, is equal to
\begin{equation}
\Phi (X,Y)=\frac{X-Y+1}{X+1}\binom{X+Y}{Y}.
\label{5}
\end{equation}
\end{lemma}
%\begin{proof}

%\end{proof}
%\end{proof}

%\textbf{Lemma 3}. {\it{If $i-1<j,\, i=1,2,\ldots ,n$, than the following formula is valid}}%
\begin{lemma}
\label{Lmma2}
If $i-1<j,\, i=1,2,\ldots ,n$, than the following formula is valid
\[
\left|\overline{\overline{\mathcal{L}}}_{r}^{(n-1)}(i-1,j)\right|
=\sum_{r\geq1}\left\{\sum_{k_{1}=0}^{r-1}\sum_{k_{2}=0}^{k_{1}}
\sum_{s=r-k_{2}-1}^{2r-k_{1}-k_{2}-2}
\binom{i-1}{k_{1}}\binom{j-i}{s}%
\binom{n-j}{k_{2}}\right.\times
\]%
\[
\times\binom{s}{2r-s-k_{1}-k_{2}-2}\cdot\frac{k_{1}-k_{2}+1}{s-r+k_{1}+2}
\cdot\binom{2s-2r+k_{1}+k_{2}+2}{s-r+k_{2}+1}+
\]%
\[
+\sum_{k_{1}=0}^{r-1}\sum_{k_{2}=0}^{k_{1}}%
\sum_{s=r-k_{2}-1}^{2r-k_{1}-k_{2}-2}\binom{i-1}{k_{1}}\binom{j-i}{s}%
\binom{n-j}{k_{2}}\binom{s}{2r-s-k_{1}-k_{2}-2}\times
\]%
\begin{equation}
\left.\times\frac{k_{2}-k_{1}+1}{s-r+k_{2}+2}\cdot\binom{2s-2r+k_{1}+k_{2}+2}{s-r+k_{1}+1}\right\},
\label{6}
\end{equation}%
where as usual the binomial coefficients $\tbinom{a}{b}$ are equal zero, if
the integers $a,b$ such that $a,b\geq a\geq 0$ or $b<0$.
\end{lemma}

\begin{proof}

%\textit{Proof. }
Indeed, let $0\leq i-1<j\leq n-1$ and $A=\{j_{1},j_{2},%
\ldots ,j_{r-1},j\}$ and $B=\{i,i_{2},\ldots ,i_{r}\}$ be two sequences of
positive integers of type $(1)$ which $1\leq i_{1}=i-1,$ $j_{r}=j\leq n-1$
and $j_{t}\leq i_{t}$, from all $t=0,1,\ldots ,r$.

%Let us 
Let  
$I_{1}=\{0,1,\ldots ,i-2\}$, $I_{2}=\{i-1\},\,I_{3}=\{i,i+1,\ldots,j-1\}$,
$I_{4}=\{4\},$ $I_{5}=\{j+1,j+2,\ldots ,n\},$ $\left|A\cap I_{1}\right|=k_{1}$,
$0\leq k_{1}\leq r-1$, $\left|B\cap I_{5}\right|=k_{2}$, $0\leq k_{2}\leq r-1.$

%Let us 
Let  
$\left| A\cap I_{3}\right|=r-k_{1}-1,$ $\left|B\cap I_{3}\right|=r-k_{2}-1,$
$\left| (A\cup B)\cap I_{3}\right|=s,$ $\max (r-k_{1}-1,r-k_{2}-1)\leq s$ $\leq 2r-k_{1}-k_{2}-2.$

Then $A\cap B\subseteq I_{3},$ and $\left|A\cap B\right|=\left| A\cap I_{3}\right|
+\left|B\cap I_{3}\right|-
\left|(A\cup B)\cap I_{3}\right|=$ $2r-s-k_{1}-k_{2}-2.$ The number of different
points of the sets $A\cap I_{3}$ and $B\cap I_{3}$ is equal to $\left|A\cap I_{3}\right|
+\left| B\cap I_{3}\right|-2\left|A\cap B\right|=(r-k_{1}-1)+(r-k_{2}-1)-2(2r-s-k_{1}-k_{2}-2)=$
$2s-2r+k_{1}+k_{2}+2$ including $s-r+k_{2}+1$ points from the set $A\cap
I_{3} $ and $s-r+k_{1}+1$ points from the set $B\cap I_{3}$. It is clear
that the number of possible choices of $A$ and $B$ from fixed integers $%
k_{1},k_{2}$ and $s$, where $k_{1}=\left\vert A\cap I_{1}\right\vert $, $%
s=\left\vert (A\cup B)\cap I_{3}\right\vert $, $k_{2}=\left\vert B\cap
I_{5}\right\vert ,$ $0\leq k_{1}\leq r-1,$ $0\leq k_{2}\leq r-1$, $\max
(r-k_{1}-1,r-k_{2}-1)s2r-k_{1}-k_{2-2,}$ is equal to%
\begin{equation*}
N_{r,n}(k_{1},k_{2},s)=\binom{i-1}{k_{1}}\binom{j-i}{s}\binom{n-j}{k_{2}}\binom{s}{2r-s-k_{1}-k_{2}-2}\times
\end{equation*}%
\begin{equation}
\times
\binom{n-j}{k_{2}}\binom{s}{2r-s-k_{1}-k_{2}-2}\Psi (s-r+k_{1}+1,s-r+k_{2}+1).
\label{7}
\end{equation}%
Here $\Psi (s-r+k_{2}+1,s-r+k_{1}+1)$ is the number of possible partitions
of the $2s-2r+k_{1}+k_{2}+2$ distinct elemens of the set $(A\cup B)\cap
I_{3} $ among which $X=s-r+k_{2}+1$ elements belong to the set $A\cap I_{3}$
and $Y=s-r+k_{1}+1$ elements belong to the set $B\cap I_{3}$, while the
elements of the sets $A$ and $B$ satisfy to the condition $(1_{1})$.

It is
clear that the number $N_{r,n}(k_{1},k_{2},s)=0$ if the integer $s\geq 0$
satisfies the inequalities $j-i<s<2r-s-k_{1}-k_{2}-2$ as in this case the
corresponding binomial coefficient $\tbinom{j-i}{s}$ or
$\tbinom{s}{2r-s-k_{1}-k_{2}-2}$ is equal to zero.

Determine the number $\Psi (s-r+k_{2}+1,s-r+k_{1}+1)$. To do this let us
arrange selecting integers $l_{1},l_{2},\ldots ,l_{2s-2r+k_{1}k_{2}+2}$ of
different elements of the sets $A\cap I_{3}$ and $B\cap I_{3}$ in increasing
order.

%Let us
Let   
$k_{1}\geq k_{2}$.

Clearly, $l_{1}\in A\cap I_{2}$.
Correspondingly, we take a step from the point $(0,\,0)$ to the point $(1,\,0)$.
If $l_{2}\in A\cap I_{3}$, then the next step is the $(2,\,0)$, while if
$l_{2}\in B\cap I_{3}$, then it is to the point $(1,\,1)$, and so on.

The
result is one possible path of the indicated form from the point $(0,0)$ to
the point $(X,Y)$. The total number of such paths is $\Phi (X,Y),$ if $X\geq
Y$, and it is $\Phi (Y,X),$ if $X\leq Y$. Thus by the formula (\ref{5}) we have%
\[
\Psi (s-r+k_{1}+1,s-r+k_{2}+1)=\Phi (s-r+k_{1}+1,s-r+k_{2}+1)=
\]%
\begin{equation}
=\frac{k_{1}-k_{2}+1}{s-r+k_{1}+2}\binom{2s-2r+k_{1}k_{2}+2}{s-r+k_{2}+1},%
\text{ \ if }k_{1}\geq k_{2},
\label{8}
\end{equation}%
\[
\Psi (s-r+k_{1}+1,s-r+k_{2}+1)=\Phi (s-r+k_{2}+1,s-r+k_{1}+1)=
\]%
\begin{equation}
=\frac{k_{2}-k_{1}+1}{s-r+k_{2}+2}\binom{2s-2r+k_{1}+k_{2}+2}{s-r+k_{1}+1},%
\text{ \ if }k_{2}\geq k_{1}.
\label{9}
\end{equation}

By the formulas (\ref{M8}), (\ref{M9}) and (\ref{2}) we get%
\[
N_{r,n}(k_{1},k_{2},s)=\binom{i-1}{k_{1}}\binom{j-i}{s}\binom{n-j}{k_{2}}%
\binom{s}{2r-s-k_{1}-k_{2}-2}\times
\]%
\begin{equation}
\times\frac{k_{1}-k_{2}+1}{s-r+k_{1}+2}\binom{2s-2r+k_{1}+k_{2}+2}{s-r+k_{2}+1},%
\text{ \ if }k_{1}\geq k_{2},
\label{10}
\end{equation}%
\[
N_{r,n}(k_{1},k_{2},s)=\binom{i-1}{k_{1}}\binom{j-i}{s}\binom{n-j}{k_{2}}%
\binom{s}{2r-s-k_{1}-k_{2}-2}\times
\]%
\begin{equation}
\times\frac{k_{2}-k_{1}+1}{s-r+k_{2}+2}\binom{2s-2r+k_{1}+k_{2}+2}{s-r+k_{1}+1},%
\text{ \ if }k_{2}\geq k_{1}.
\label{11}
\end{equation}%
As
\[
\left|\overline{\overline{\mathcal{L}}}^{(n-1)}(i-1,j) \right |
=\sum\limits_{r\geq 1}\left |\overline{\overline{\mathcal{L}}}_{r}^{(n-1)}(i-1,r)\right|=
\]
\[
=\sum\limits_{r\geq1}\sum_{k_{1}=1}^{r-1}
\sum_{k_{2}=1}^{r-1}\sum_{s=r-k_{2}+1}^{2r-k_{1}-k_{2}-2}N_{r,n}(k_{1},k_{2},s)=
\]
\[
=\sum_{r\geq1}\left\{\sum_{k_{1}=0}^{r-1}\sum_{k_{2}=0}^{k_{1}}
\sum_{s=r-k_{2}+1}^{2r-k_{1}-k_{2}-2}N_{r,n}(k_{1},k_{2},s)\right.+
\]
\[
\left.+\sum_{k_{1}=1}^{r-1}\sum_{k_{2}=0}^{r-2}%
\sum_{s=r-k_{2}+1}^{2r-k_{1}-k_{2}-2}N_{r,n}(k_{1},k_{2},s)\right\}
\]
tnen by the
formulas (\ref{10}) and (\ref{11}) we get the formula
(\ref{6})

\[
S_{3}(r,n)=\sum_{k=1}^{r-1}\binom{i-1}{k}\binom{n-j}{k}\cdot\frac{1}{j-i+1}\cdot\binom{j-i+1}{r-k}\binom{j-i+1}{r-k-1}.
\]

Futher we will spend (\ref{6}) summation on parameteres $s,\,k_{1},\,k_{2}$ and $r$.

\end{proof}

%\textbf{Lemma 4 }(\textit{cp. Egorychev, formula} (2.36)). \textit{Let}
%$m,\,a,\,b$ \textit{are positive integers, such that} $m\geq a+b,\,a\geq b.$
%\textit{Then the following summation formula is valid}%

\begin{lemma}%(cp. Egorychev, formula (2.36)).
\label{Lmm3}
Let $m,\,a,\,b$ be positive integers, such that $m\geq a+b,\,a\geq b.$
Then the following summation formula is valid
\begin{equation*}
S=\sum_{s=a}^{a+b}\binom{m}{s}\binom{s}{a+b-s}\cdot\frac{1}{s-b+1}\binom{2s-a-b}{s-a}=
\end{equation*}
\begin{equation}
=\frac{1}{m+1}\binom{m+1}{a+1}\binom{m+1}{b};
\label{12}
\end{equation}
\end{lemma}

%\textit{Proof.} We have%
\begin{proof}
 We have
\[
S=\sum_{s=a}^{a+b}\binom{m}{s}\binom{s}{a+b-s}\cdot\frac{1}{s-b+1}\binom{2s-a-b}{s-a}
\]%
\[
=\sum_{s=a}^{a+b}\frac{m!}{s!(m-s)!}\cdot\frac{s!}{(a+b-s)!(s-b+1)!}\cdot\frac{1}{s-b+1}
\cdot\frac{(2s-a-b)!}{(s-a)!(s-b)!}=
\]%
\[
=\sum_{s=a}^{a+b}\frac{m!}{(m-s)!(a+b-s)!(s-b+1)!(s-a)!}=
\]
\[
=\frac{1}{m+1}\binom{m+1}{a+1}\sum_{s=a}^{a+b}\binom{m-a}{s-a}\binom{a+1}{a+b-s}=
\]
\[
=\frac{1}{m+1}\binom{m+1}{a+1}\sum_{s=0}^{b}\binom{m-a}{s}\binom{a+1}{b-s}
=\frac{1}{m+1}\binom{m+1}{a+1}\binom{m+1}{b}.
\]

Having spent a summation on $s$ in right part of the (\ref{6}) of the formula (\ref{12})  we get.

\end{proof}

%\textbf{Lemma 5}. If $i-1<j,\,i=1,2,\ldots ,n,$ then the
%following summation formula is valid

\begin{lemma}.
\label{Lmm4}
If $i-1<j,\,i=1,2,\ldots ,n,$ then the
following summation formula is valid
\[
\left|\overline{\overline{\mathcal{L}}}^{(n-1)}(i-1,j)\right|=
\]%
\[
=\sum_{r\geq 1}\left\{\sum_{k_{1}=0}^{r-1}\sum_{k_{2}=0}^{k_{1}}\binom{i-1}{k_{1}}
\binom{n-j}{k_{2}}\frac{(k_1-k_2+1)}{(j-i+1)}\binom{j-i+1}{r-k_{2}}\binom{j-i+1}{r-k_{1}-1}+\right.
\]%
\[
+\left.\sum_{k_{1}=0}^{r-1}\sum_{k_{2}=0}^{r-1}\binom{i-1}{k_{1}}\binom{n-j}{k_{2}}
\frac{(k_{2}-k_{1}+1)}{(j-i+1)}\binom{j-i+1}{r-k_{1}}\binom{j-i+1}{r-k_{2}-1}\right\}.
\]%
\bigskip
\end{lemma}

\subsubsection{The calculation of the difficult 6-multiple combinatorial sum $\Omega^{+}(n)$ with the help of the Egorychev method of coefficients}

\textbf{Theorem 4 (\cite{Dav2009}). }\textit{For the number}$\Omega^{+}(n)$\textit{\ the following equalities are valid}:
\[
\Omega^{+}(n)=2^{2n-1}+(n-1)\binom{2n-2}{n-1}-\frac{4}{n}\binom{2n}{n-2}-\binom{2n}{n}=
\]%
\begin{equation}
=2^{2n-1}+\frac{2(2n-3)!}{(n-2)!(n+2)!}(n^{4}-2n^{3}-27n^{2}+20n-4).
\label{A73}
\end{equation}

Let $\Omega^{+}(n)=T_{1}+T_{2}+T_{3},$ where
\begin{equation}
T_{1}:=\sum_{i=2}^{n}\sum_{j=1}^{i-1}{\binom{n-i+j-1}{j-1}}{\binom{i-1+n-j}{i-1}},
\label{25}
\end{equation}%
\begin{equation*}
T_{2}:=\sum_{i=2}^{n}\sum_{j=i}^{n-1}\binom{i-1+n-j}{i-1}\left(\binom{n-i+j}{j} - \binom{n-i+j}{j+1}\right)=
\end{equation*}%
\begin{equation}
=\sum_{i=2}^{n}\sum_{j=i}^{n-1}{\binom{i-1+n-j}{i-1}}
\mbox{\bf res}_{u}\left\{\frac{(1+u)^{n+j-i}\cdot(1-u)}{u^{j+2}}\right\},
\label{26}
\end{equation}%
\begin{equation*}
T_{3}:=-\sum_{i=2}^{n}\sum_{j=i}^{n-1}\sum_{s=0}^{j-i-3}\binom{i-1+n-j}{i-1}
\binom{i-1+n-j}{s+i}\binom{2j-2i}{j-i-s-3}+
\end{equation*}%
\begin{equation}
\label{27}
+\sum_{i=2}^{n}\sum_{j=i}^{n-1}\sum_{s=0}^{j-i-1}\binom{i-1+n-j}{i-1}\binom{i-1+n-j}{s+i}\binom{2j-2i}{j-i-s+1}.
\end{equation}

%\begin{exercise}
\begin{lemma}.
\label{Lem2}
For any natural $n>2$ the following combinatorial identity is
valid%
\begin{equation}
T_{1}=\sum_{i=2}^{n}\sum_{j=1}^{i-1}{\binom{n-i+j-1}{j-1}}{\binom{n+i-j-1}{%
i-1}}{=}\text{ }\left( n-1\right) {\binom{2n-2}{n-1}}.
\label{28}
\end{equation}%
\medskip
\end{lemma}
%\end{exercise}

\begin{proof}
Using the known relations and integral representations for the
corresponding binomial coefficients, we obtain
\begin{equation*}
T_{1}=\sum_{i=2}^{n}\sum_{j=1}^{i-1}\binom{n-i+j-1}{j-1}\binom{n+i-j-1}{n-j}
=\sum_{i=2}^{n}\sum_{j=0}^{i-2}\binom{n-i+j}{j}\binom{n+i-j}{n-j-1}
\end{equation*}%
\begin{equation*}
=\sum_{i=2}^{n}\sum_{j=0}^{i-2}\{\mbox{\bf res}_{x}\frac{(1-x)^{-n+i-1}}{x^{j+1}}
\mbox{\bf res}_{y}\frac{(1-y)^{-i-2}}{y^{n-j}}\}.
\end{equation*}%
Using the formula of a geometric progression in the sum on the index $j$, we
get
\begin{equation*}
\sum_{i=2}^{n}\mbox{\bf res}_{xy}\{[\frac{(1-x)^{-n+i-1}(1-y)^{-i-2}}{xy^{n}}%
\frac{1-(y/x)^{i-1}}{1-y/x}]_{|x|\gg |y|=\rho ,\rho \ll 1}\}
\end{equation*}%
\begin{equation*}
(the\, rule\, of\, linearity)
\end{equation*}
\begin{equation*}
=\sum_{i=2}^{n}\mbox{\bf res}_{xy}\{[\frac{(1-x)^{-n+i-1}(1-y)^{-i-2}}{y^{n}}%
\cdot \frac{1}{y-x}]_{|x|\gg |y|=\rho ,\rho \ll 1}\}
\end{equation*}%
\begin{equation*}
-\sum_{i=2}^{n}\mbox{\bf res}_{xy}\{[\frac{(1-x)^{-n+i-1}(1-y)^{-i-2}}
{x^{i-1}y^{n-i+1}(y-x)}]_{|x|\gg |y|=\rho ,\rho \ll 1}\}.
\end{equation*}%
\begin{equation*}
(change\,\, of\,\, variables\,\,X=x/(1-x),\,Y=y/(1-y)\,\,under\,\,the\,\,sign
\end{equation*}%
\begin{equation*}
of\,\,the\,\,operator\,\,\mbox{\bf res}\,\,in\,\,the\,\,first\,\,and\,\,second\,\,sums)
\end{equation*}%
\begin{equation*}
T_{1}=\sum_{i=2}^{n}\mbox{\bf res}_{xy}\{[\frac{(1+x)^{n-i}(1+y)^{n+i-1}}{y^{n}(x-y)}]_{|x|\gg |y|}\}
\end{equation*}%
\begin{equation*}
-\sum_{i=2}^{\infty }\mbox{\bf res}_{xy}\{[\frac{(1+x)^{n-1}(1+y)^{n}}
{x^{i-1}y^{n-i+1}(x-y)}]_{|x|\gg |y|}\}=
\end{equation*}%
\begin{equation*}
=\sum_{i=2}^{n}\mbox{\bf res}_{y}\{\frac{(1+y)^{n-i}(1+y)^{n+i-1}}{y^{n}}\}
\end{equation*}%
\begin{equation*}
-\mbox{\bf res}_{xy}\{[\frac{(1+x)^{n-1}(1+y)^{n}}{xy^{n-1}(x-y)}\cdot \frac{%
x}{(x-y)}]_{|x|\gg |y|}\}
\end{equation*}%
\begin{equation*}
=\sum_{i=2}^{n}\mbox{\bf res}_{y}\frac{(1+y)^{2n-1}}{y^{n}}-\mbox{\bf res}%
_{xy}\{[\frac{(1+x)^{n-1}(1+y)^{n}}{y^{n-1}(x-y)^{2}}]_{|x|\gg |y|}\}.
\end{equation*}%
Proceeding similarly, we take the first sum res by $y$, and the second by $x$
\begin{equation*}
T_{1}=\sum_{i=2}^{n}{\binom{2n-1}{n-1}}-\mbox{\bf res}_{y}\{[\frac{d}{dx}(%
\frac{(1+x)^{n-1}(1+y)^{n}}{y^{n-1}})]_{x=y}\}
\end{equation*}%
\begin{equation*}
=(n-1){\binom{2n-1}{n-1}}-(n-1)\mbox{\bf res}_{y}\{[\frac{%
(1+x)^{n-2}(1+y)^{n}}{y^{n-1}})]_{x=y}\}
\end{equation*}%
\begin{equation*}
=(n-1){\binom{2n-1}{n-1}}-(n-1)\mbox{\bf res}_{y}\frac{(1+y)^{2n-2}}{y^{n-1}}
\end{equation*}%
\begin{equation*}
=(n-1){\binom{2n-1}{n-1}}-(n-1){\binom{2n-2}{n-2}}=(n-1){\binom{2n-2}{n-1}}.
\end{equation*}%
%\hfill
\end{proof}

Let%
\begin{equation}
T_{2}:=\sum_{i=2}^{n}\sum_{j=i}^{n-1}\binom{i-1+n-j}{n-j}\left(\binom{n-i+j}{j}-\binom{n-i+j}{j+1}\right)=
\label{29}
\end{equation}%
\begin{equation}
=\sum_{i=2}^{n}(\sum_{j=i}^{n}...-\sum_{j=n}...)=S_{1}+S_{2},
\label{30}
\end{equation}%
where%
\begin{equation}
S_{1}:=\sum_{i=2}^{n}\left( \binom{2n-i}{n+1}-\binom{2n-i}{n}\right) ,
\label{31}
\end{equation}%
\begin{equation}
S_{2}:=\sum_{i=2}^{n}\sum_{j=i}^{n}{\binom{i-1+n-j}{n-j}}
\left(\binom{n-i+j}{j}-\binom{n-i+j}{j+1}\right).
 \label{32}
\end{equation}

\begin{lemma}.
\label{Lem3}
For $n=3,4,\ldots $ the following combinatorial identity is
valid
\begin{equation}
\label{33}
S_{1}=\sum_{i=2}^{n}\left( \binom{2n-i}{n+1}-\binom{2n-i}{n}\right)=
\binom{2n-1}{n+2}-\binom{2n-1}{n+1}.
\end{equation}%
\end{lemma}

\begin{proof}
We have%
\begin{equation*}
S_{1}=\sum_{i=2}^{n}\mbox{\bf res}_{x}\left\{(1+x)^{2n-i}/x^{n+2}-(1+x)^{2n-i}/x^{n+1}\right\}=
\end{equation*}%
\begin{equation*}
=\mbox{\bf res}_{x}\sum_{i=2}^{n}\left\{(1+x)^{2n-i}(1-x)/x^{n+2}\right\}=
\end{equation*}%
\begin{equation*}
=\mbox{\bf res}_{x}\left\{\frac{(1+x)^{2n-2}(1-x)}{x^{n+2}}\cdot
\left(\frac { 1-(1+x)^{-(n-1)}}{1-(1+x)^{-1}}\right)
\right\}
\end{equation*}%
\begin{equation*}
=-\mbox{\bf res}_{x}\{(1+x)^{n}(1-x)/x^{n+3}\}+\mbox{\bf res}_{x}\{(1+x)^{2n-1}(1-x)/x^{n+3}\}=
\end{equation*}%
\begin{equation*}
=-0+\mbox{\bf res}_{x}(1+x)^{2n-1}(1-x)/x^{n+3}=
\end{equation*}%
\begin{equation*}
={\binom{2n-1}{n+2}}-{\binom{2n-1}{n+1}}.
\end{equation*}%
%thus,%
%\begin{equation*}
%S_{1}=\sum_{i=2}^{n}\left( {\binom{2n-i}{n+1}}-{\binom{2n-i}{n}}\right) ={\
%\binom{2n-1}{n+2}}-{\binom{2n-1}{n+1}}.%\cdot   \label{33}
%\end{equation*}%
%\hfill
\end{proof}

The following statement establishes the proof of Lemma \ref{Lem3}.

\begin{lemma}.
\label{lem4}
For $n=3,4,\ldots $ the following identity is valid
\begin{equation*}
S_{2}=\sum_{i=2}^{n}\sum_{j=i}^{n}\binom{i-1+n-j}{n-j}\left(\binom{n-i+j}{j}-\binom{n-i+j}{j+1}\right)=
\end{equation*}%
\begin{equation}
=-(n-1) -\binom{2n}{n}+2\binom{2n}{n+1}.
 \label{34}
\end{equation}%
\end{lemma}

From the formulas(\ref{30}), (\ref{33}) and (\ref{34}) follows

\begin{lemma}.
\label{Lem5}
 For $n=3,4,\ldots $ the following identity is valid
%\emph{\
\begin{equation*}
T_{2}:=\sum_{i=2}^{n}\sum_{j=i}^{n-1}\binom{i-1+n-j}{n-j}\left(\binom{n-i+j}{j}-\binom{n-i+j}{j+1}\right)=
\end{equation*}%
\begin{equation}
=\binom{2n-1}{n+2}-\binom{2n-1}{n+1}-(n-1) - \binom{2n}{n}+2\binom{2n}{n+1}.
\label{35}
\end{equation}
\end{lemma}

%\paragraph{The calculation of the sum $T_{3}$}
\subsubsection{The calculation of the sum $T_{3}$}
\begin{lemma}.
\label{lem6}
For\ $n=3,4,\ldots $ the following identity is valid
\begin{equation*}
F_{1}=\sum_{i=2}^{n}\sum_{j=n}\binom{i-1+n-j}{i-1}\mbox{\bf res}_{yu}
\frac{(1+y)^{n-j+i-1}(1+u)^{2j-2i}(1-u^4)}{y^{i}u^{j-i+2}(y-u)}=0.
\end{equation*}%
\end{lemma}

\begin{proof}.
We obtain with the method of coefficients successively
\begin{equation*}
F_{1}=\sum_{i=2}^{n}\sum_{j=n}\binom{i-1+n-j}{i-1}\mbox{\bf res}_{yu}
\frac{(1+y)^{n-j+i-1}(1+u)^{2j-2i}(1-u^4)}{y^{i}u^{j-i+2}(y-u)}=
\end{equation*}%
\begin{equation*}
=\sum_{i=2}^{n}\mbox{\bf res}_{yu}\frac{(1+y)^{i-1}(1+u)^{2n-2i}(1-u^4)}
{y^{i}u^{n-i+2}(y-u)}=
\end{equation*}%
\begin{equation*}
=\sum_{i=2}^{n}\mbox{\bf res}_{yu}\{[\frac{(1+y)^{i-1}(1+u)^{2(n-i)}(1-u^4)}
{y^{i+1}u^{n-i+2}(1-u/y)}]_{|y|\gg |u| =\rho }\}=
\end{equation*}%
\begin{equation*}
\mbox{\bf res}_{u}\{\frac{(1+u)^{2(n-i) }(1-u^4)}{u^{n-i+2}} \{\sum_{i=2}^{n}\mbox{\bf res}_{y}
\frac{(1+y)^{i-1}}{y^{i+1}}\left(1+\sum_{s=0}^{\infty}\left( u/y\right)^{s}\right) \}\}=0
\end{equation*}%
by the definition of $\mbox{\bf res}_{u}.$
\end{proof}

\begin{lemma}.
\label{lem7}
 For $n=3,4,\ldots $ the following identity is valid
%\emph{\
\begin{equation*}
T_{3}=-\sum_{i=2}^{n}\sum_{j=i}^{n-1}\sum_{s=0}^{j-i-3}\binom{i-1+n-j}{i-1}
\binom{i-1+n-j}{s+i}\binom{2j-2i}{j-i-s-3}+
\end{equation*}%
\begin{equation*}
+\sum_{i=2}^{n}\sum_{j=i}^{n-1}\sum_{s=0}^{j-i-1}\binom{i-1+n-j}{i-1}
\binom{i-1+n-j}{s+i}\binom{2j-2i}{j-i-s+1}
\end{equation*}%
\begin{equation}
=\sum_{i=2}^{n}\sum_{j=i}^{n}\binom{i-1+n-j}{n-j}
\mbox{\bf res}_{yu}\frac{(1+y)^{n-j+i-1}(1+u)^{2j-2i}(1-u^{4})}{y^{i}u^{j-i+2}(y-u)}
\label{36}
\end{equation}%
\end{lemma}

The proof will be divided into a series of statements.

\begin{lemma}.
\label{lem8}
Let $|t|=\rho =0,01,\,|u|=10, \,|y|=10^{4}.$ Then%
\begin{equation}
T_{3}=\mbox{\bf res}_{tyu}\left\{\left[\frac{t^{-n+1}(1+y)^{n}(1-u^4)(1-t)^{-2}}{%
\left(y-\left(-1+t\frac{(1+u)^2}{u}\right)\right) yu^{2}(y-u)
\left( y-\frac{t}{1-t} \right) }\right]_{|y|\gg |u| \gg |t|=\rho}\right\}
 \label{38}
\end{equation}%
\begin{equation}
=J_{1}+J_{2}+J_{3}+J_{4},
\label{39}
\end{equation}%
where the integrals are
\begin{equation}
J_{1}:=\mbox{\bf res}_{tu}\{[\frac{(1-u^4)}{t^{n}(1-t)u^{2}\left(u-t(1+u)^2\right)}]_{|u|\gg |t|=\rho }\},
\label{41}
\end{equation}%
\begin{equation}
J_{2}:=\mbox{\bf res}_{tu}\{[\frac{(1-u^4)(1+u)^{n-1}} {t^{n-1}(1-t)
u^{2}\left(u-t(1+u)\right)^{2}}]_{|u|\gg |t|=\rho }\},
\label{42}
\end{equation}%
\begin{equation}
J_{3}:=\mbox{\bf res}_{tu}\{[\frac{1}{t^{n+1}u}\cdot\frac{(1-u)(1-t)^{-n-1}}
{\left(u-\frac{t}{1-t}\right)^{2}\left(u-\frac{1-t}{t}\right)}]_{|u|\gg|t| =\rho },
\label{43}
\end{equation}%
\begin{equation}
J_{4}:=\mbox{\bf res}_{tu}\{[\frac{(1-u^4) (1+u)^{2n-1}u^{-n+1}(1-t)^{-2}}
{(u-t(1+u)^{2}) \left(u-\frac{t}{1-t}\right)^{2}\left(u-\frac{1-t}{t}\right)}]_{|u|\gg |t|=\rho }\}.
\label{44}
\end{equation}
\end{lemma}

\begin{proof}
We have
\begin{equation*}
T_{3}=\sum_{i=2}^{n}\sum_{j=i}^{n}
\left(\mbox{\bf res}_{z}(1+z)^{i-1+n-j}z^{-n+j-1}\times\right.
\end{equation*}%
\begin{equation*}
\left.\times \mbox{\bf res}_{yu}[\frac{(1+y)^{n-j+i-1}(1+u)^{2j-2i}(1-u^4)}
{y^{i}u^{j-i+2}(y-u)}]_{|y|\gg |u|=\rho }\right)=
\end{equation*}%
\begin{equation*}
=\sum_{i=2}^{n}\mbox{\bf res}_{zyu}\left\{\sum_{j=i}^{\infty}
(1+z)^{i-1+n-j}z^{-n+j-1}\times\right.
\end{equation*}%
\begin{equation*}
\left.\times [\frac{(1+y)^{n-j+i-1}(1+u)^{2j-2i}(1-u^4)}{y^{i}u^{j-i+2}(y-u)}]_{|y|\gg |u|=\rho }\right\}
\end{equation*}%
\begin{equation*}
=\sum_{i=2}^{n}\mbox{\bf res}_{zyu}\left\{(1+z)^{n-1}z^{-n+i-1}
\left(1-\frac{z(1+u)^{2}}{(1+z)(1+y)u}\right)^{-1} \cdot\frac{(1+y)^{n-1}(1-u^4)}
{y^{i}u^{2}(y-u)}\right\}
\end{equation*}%
\begin{equation*}
=\mbox{\bf res}_{zyu}\left\{\sum_{i=2}^{\infty }\frac{(1+z)^{n}z^{-n+i-1}}
{(1+z)(1+y)u-z(1+u)^{2}}\cdot \frac{(1+y)^{n}(1-u^{4}}{y^{i}u(y-u)}\right\}=
\end{equation*}%
\begin{equation*}
=\mbox{\bf res}_{zyu}\left\{\frac{(1+z)^{n}z^{-n+1}}{(1+z)(1+y)u-z(1+u)^{2}}
\cdot\frac{(1+y)^{n}(1-u^{4})}{y^{2}u(y-u)(1-z/y)}\right\}=
\end{equation*}%
\begin{equation*}
=\mbox{\bf res}_{zyu}\left\{[\frac{(1+z)^{n}z^{-n+1}(1+y)^{n}(1-u^{4})}
{\left((1+z)(1+y)u -z(1+u)^{2}\right)yu(y-u)(y-z)}]_{|y|\gg |u|\gg|z|=\rho}\right\}.
\end{equation*}

If you make a change here $t=z(1+z)^{-1}\in H_{1}$,
$ z=t(1-t)^{-1}, \, dz/dt=(1-t)^{-2} $, $ y-z=y-t(1-t)^{-1} $, then
\begin{equation*}
T_{3}=\mbox{\bf res}_{zyu}\left\{\frac{(z/(1+z))^{-n+1}(1-u^{4})(1+y)^{n}}
{(u(1+y) -z(1+u)^{2}/(1+z))yu(y-u)(y-z)}\right\}=
\end{equation*}%
\begin{equation*}
=\mbox{\bf res}_{tyu}\left\{\frac{t^{-n+1}(1-u^{4})(1+y)^{n}}
{(1-t)^{2}(u(1+y)-t(1+u)^{2}) yu(y-u)(y-t(1-t)^{-1})}\right\}=
\end{equation*}%
\begin{equation*}
=\mbox{\bf res}_{tyu}\left\{[\frac{t^{-n+1}(1+y)^{n}(1-u^{4})(1-t)^{-2}}
{(y-(-1+t(1+u)^{2}u^{-1})) yu^{2}(y-u)(y-t(1-t)^{-1})}]_{|y|\gg |u| \gg|t|=\rho }\right\}=
\end{equation*}%
\begin{equation*}
=\mbox{\bf res}_{tu}\left\{\frac{(1-u^{4})t^{-n+1}}{(1-t)^{2}u^{2}} %
\mbox{\bf res}_{y}[\frac{(1+y)^{n}(y-u)^{-1}}
{ y \left(y-\frac{t}{1-t}\right) (y-(-1+t(1+u)^{2}u^{-1}))}]_{|u|\gg |t|=\rho}\right\}.
\end{equation*}%
We compute the last integral by the residue theorem, in which the choice is
$|t|=\rho =0,1,\,|u|=10,\,|y|=10^{3}$ we have poles at the points
$y=0,\,y=u,\,y=t/(1-t)$ and $y=-1+t(1+u)^{2}u^{-1}.$ Then
\begin{equation*}
T_{3}=\mbox{\bf res}_{tu}\{\frac{(1-u^{4})t^{-n+1}}{(1-t)^{2}u^{2}}\cdot
\left\{[\frac{(1+y)^{n}\left(y-t(1-t)^{-1}\right)^{-1}}
{y\left(y-(1+t(1+u)^{2}u^{-1})\right)}]_{|u| \gg |t|=\rho ,y=u}\right.+
\end{equation*}%
\begin{equation*}
+[\frac{(1+y)^{n}(y-u)^{-1}}{(y-t(1-t)^{-1})(y-(-1+t(1+u)^{2}u^{-1}))}]_{|u|\gg |t|=\rho ,y=0}+
\end{equation*}%
%\begin{equation*}
%+[\frac{\left( 1+y\right) ^{n}}{y\left( y-t\left( 1-t\right) ^{-1}\right)
%(y-(-1+t\left( 1+u\right) ^{2}u^{-1}))}]_{\left\vert u\right\vert \gg
%\left\vert t\right\vert =\rho ,y=u}
%\end{equation*}%
\begin{equation*}
+[\frac{(1+y)^{n}}{y(y-u)(y-(-1+t(1+u)^{2}u^{-1}))}]_{|u| \gg |t| =\rho ,y=t(1-t)^{-1}}+
\end{equation*}%
\begin{equation*}
\left.+[\frac{(1+y)^{n}}{y(y-u)\left( y-t(1-t)^{-1}\right)}]_{|u| \gg |t| =\rho,y=-1+t(1+u)^{2}u^{-1})}\right\}=
\end{equation*}%
\begin{equation*}
=\mbox{\bf res}_{tu}\{\frac{(1-u^{4})}{t^{n-1}(1-t)^{2}u^{2}}\cdot
\left\{[\frac{1-t}{t(u-t(1+u)^{2})}]_{|u| \gg |t|=\rho }+\right.
\end{equation*}
\begin{equation*}
+[\frac{(1-t)(1+u)^n}{\left(u(1-t)-t\right)(u^{2}+u-t(1+u)^{2})}]_{|u|\gg |t|=\rho}+
\end{equation*}%
\begin{equation*}
+[\frac{u(1-t)^{3-n}}{t(t-u(1-t))(ut+(1-t)u-(1-t)t(1+u)^{2})}]_{|u|\gg |t| =\rho}+
\end{equation*}%
\begin{equation*}
\left.+\left[\frac{(1-t)u^{3}(t(1+u)^{2}u^{-1})^{n}(-u+t(1+u)^{2})^{-1}}
{((u+t(1+u)^{2}-u^{2})(-(1-t)u+(1-t)t(1+u)^{2}-ut)}\right]_{|u|\gg |t| =\rho }\}\right\}=
\end{equation*}%
\begin{equation*}
=J_{1}+J_{2}+J_{3}+J_{4}.
\end{equation*}%
\end{proof}

\begin{lemma}
\label{lem9}
For $|u| \gg |t| =\rho $
the following identity is valid
\begin{equation}
J_{1}=\mbox{\bf res}_{tu}\{[\frac{(1-u^{4})}{(1-t)t^{n}u^{2}(u-t(1+u)^{2})}]_{|u|\gg |t|=\rho }\}=0.
\label{66}
\end{equation}%
\end{lemma}
\begin{proof}
As $|t|=\rho = 0,01,\,|u|=10$ и $\left\vert t(1+u)^{2}/u\right\vert <1$,
the calculation of residues $u$ and $t$ is carried
out using the method of coefficients similar analogously in Lemma \ref{lem7}.
\end{proof}

\begin{lemma}
\label{lem10} For
$|u| \gg |t| =\rho $ and $n=3,4,\ldots $ the following identity is valid
\begin{equation*}
J_{2}=\mbox{\bf res}_{tu}\{[\frac{t^{-n+1}(1-u^{4})(1+u)^{n}}
{(1-t) u^{2}(u(1-t)-t)(u^{2}+u-t(1+u)^{2})}]_{|u| \gg |t| =\rho }\}=
\end{equation*}%
\begin{equation}
=(n-1)-\frac{2}{n}\binom{2n}{n-2}.
 \label{67}
\end{equation}%
\end{lemma}

The proof is analogical to Lemma \ref{lem8}.
\begin{equation*}
J_{3}=\mbox{\bf res}_{tu}\{[\frac{1}{(1-t)^{n+1}t^{n+1}}\cdot
\frac{(1-u^{4})}{u(u-t/(1-t))^{2}(u-(1-t)/t)}]_{|u|\gg |t|=\rho},\}
\end{equation*}%
\begin{equation}
J_{4}=\mbox{\bf res}_{tu}\{[\frac{(1-u^{4})(1+u)^{2n-1}u^{-n+1}(1-t)^{-2}}
{(u-t(1+u)^{2})(u-t/(1-t))^{2}(u-(1-t)/t)}]_{|u|\gg |t|=\rho }\}.
\label{48}
\end{equation}

%\begin{exercise}

\begin{lemma}
\label{lem11}
For $n=3,4,\ldots $ and meeting the conditions
$|u|\gg |t|=\rho $ the following identity is valid
\begin{equation*}
J_{3}=\mbox{\bf res}_{tu}\{[\frac{1}{(1-t)^{n+1}t^{n+1}}\cdot \frac{(1-u^{4})%
}{u\left(u-t/(1-t)\right)^{2}\left(u-(1-t)/t\right)}]_{|u|\gg |t|=\rho }=
\end{equation*}%
\begin{equation}
=2^{2n-1}-{\binom{2n}{n+1}}-{\binom{2n+1}{n}}+{\binom{2n}{n}}.
\label{68}
\end{equation}
\end{lemma}

%\end{exercise}
\begin{proof}
Let $|t|=\rho _{1}=0,01,\,|u|=\rho_{2}=10$ and
$|t|/|(1-t)|<1<\rho_{2}$ and $|(1-t)|/|t|\approx 100>10=\rho_{2},$ then compute
the last integral as a residue for $u$ at $u=0,u=t/(1-t) $,
except for the point $u=(1-t)/t$, we obtain
\begin{equation*}
J_{3}=\mbox{\bf res}_{t}\{\frac{1}{t^{n+1}(1-t)^{n+1}}
([\frac{(1-u^{4})}{(u-t/(1-t))^{2}(u-(1-t)/t)}]_{u=0}+
\end{equation*}%
\begin{equation*}
+\frac{d}{du}[\frac{(1-u^{4})}{u(u-(1-t)/t)}]_{u=t/(1-t)})\}=
\end{equation*}%
\begin{equation*}
=\mbox{\bf res}_{t}\{\frac{1}{(1-t)^{n+1}t^{n+1}}\cdot
\frac{1}{(-t/(1-t))^{2}(-( 1-t)/t)}\}+
\end{equation*}%
\begin{equation*}
+\mbox{\bf res}_{t}\{\frac{t^{-(n+1)}1}{(1-t)^{n+1}}[\frac{-4u^{3}}
{u(u-(1-t)/t)}- \frac{(1-u^{4})(2u-(1-t)/t)}{u^{2}(u-(1-t)/t)^{2}}]_{u=t/(1-t) }\}.
\end{equation*}%
As%
\begin{equation*}
[(1-u^{4})]_{u=t/(1-t)}=[(1-u^{2}) (1+u^{2})]_{u=t/(1-t)}=
\end{equation*}
\begin{equation*}
(1-2t) (2t^{2}-2t+1)/(1-t)^{4},
\end{equation*}%
then
\begin{equation*}
J_{3}=-\mbox{\bf res}_{t}\frac{(1-t)^{-n}}{t^{n+2}}-
\mbox{\bf res}_{t}\frac{4\left( t/(1-t)\right)^{2}}
{(1-t)^{n+1}t^{n+1}(t/(1-t)-(1-t)/t)}=
\end{equation*}%
\begin{equation*}
=-\mbox{\bf res}_{t}\{\frac{1}{(1-t)^{n+1}t^{n+1}}\cdot
\frac{(1-2t)(2t^{2}-2t+1)(2t/(1-t)-(1-t)/t)}
{(1-t)^{4}(t/(1-t))^{2}((t/(1-t)-(1-t)/t)^{2}}\}=
\end{equation*}%
\begin{equation*}
=-\binom{2n}{n+1}+
\mbox{\bf res}_{t}\frac{2t^{4}-2t^{3}+5t^{2}-4t+1}{(1-t)^{n+2}t^{n+2}(1-2t)}.
\end{equation*}%
Thus%
\begin{equation}
J_{3}=-\binom{2n}{n+1}+\mbox{\bf res}_{t}\frac{2t^{4}-2t^{3}+5t^{2}-4t+1}{(1-t)^{n+2}t^{n+2}(1-2t)}.
\label{69}
\end{equation}%
If under the integral sign (\ref{69}) we make the substitution
$t=(1-(1-4w)^{1/2})/2\in H_{1}$, then
\begin{equation*}
w=t(1-t) \in H_{1},\,t=(1-(1-4w)^{1/2})/2
\in H_{1},\,1-2t=(1-4w)^{1/2},
\end{equation*}%
\begin{equation*}
dt/dw=(1-4w)^{-1/2},\,
%(1-4w)^{1/2}=-2\sum_{s=0}^{\infty}\frac{1}{s+1}\binom{2s}{s}w^{s+1},
(1-4w)^{-1/2}=\sum_{s=0}^{\infty }\binom{2s}{s}w^{s}
\end{equation*}%
as a result of a simple calculation gives us the following expression for
$J_{3}$.
\begin{equation*}
J_{3}=-\binom{2n}{n+1}+\mbox{\bf res}_{t}[\frac{2t^{4}-2t^{3}+5t^{2}-4t+1}
{(1-t)^{n+2}t^{n+2}(1-2t)}]_{t=(1-(1-4w)^{1/2})/2}=
\end{equation*}%
\begin{equation*}
=-\binom{2n}{n+1}+\mbox{\bf res}_{w}
\frac{1+10(1-4w)-2(1-4w)^{1/2}-2(1-4w)^{3/2}+(1-4w)^{2}}{8(1-4w)w^{n+2}}=
\end{equation*}%
\begin{equation*}
=-\binom{2n}{n+1}+\frac{1}{8}\mbox{\bf res}_{w}\frac{1}{(1-4w)w^{n+2}}
+\frac{5}{4}\mbox{\bf res}_{w}\frac{1}{w^{n+2}}
\end{equation*}%
\begin{equation*}
-\frac{1}{4}\mbox{\bf res}_{w}\frac{(1-4w)^{-1/2}}{w^{n+2}}-
\frac{1}{4}\mbox{\bf res}_{w}\frac{(1-4w)^{1/2}}{w^{n+2}}+
\mbox{\bf res}_{w}\frac{(1-4w)}{w^{n+2}}=
\end{equation*}%
\begin{equation*}
=-\binom{2n}{n+1}+\frac{1}{8}4^{n+1}+0-
\frac{1}{4}\mbox{\bf res}_{w}\frac{(1-4w)^{-1/2}\left( 1+(1-4w)\right)}{w^{n+2}}+0=
\end{equation*}%
\begin{equation*}
=-\binom{2n}{n+1}+2^{2n-1}-\frac{1}{4}\mbox{\bf res}_{w}
\frac{(1-4w)^{-1/2}(2-4w)}{w^{n+2}}=
\end{equation*}%
\begin{equation}
=-\binom{2n}{n+1}+2^{2n-1}-\frac{1}{2}\mbox{\bf res}_{w}\frac{(1-4w)^{-1/2}}{w^{n+2}}+
\mbox{\bf res}_{w}\frac{(1-4w)^{-1/2}}{w^{n+1}}.
\label{70}
\end{equation}
Using the well-known expansion
%$(1-4w)^{1/2}=-2\sum_{s=0}^{\infty } \frac{1}{s+1}\binom{2s}{s}w^{s+1}$
$(1-4w)^{-1/2}=\sum_{s=0}^{\infty }\binom{2s}{s}w^{s}$
for (\ref{70}), we obtain
\begin{equation*}
J_{3}=-\binom{2n}{n+1}+2^{2n-1}-\frac{1}{2}\binom{2n+2}{n+1}+\binom{2n}{n}=
\end{equation*}%
\begin{equation*}
=2^{2n-1}-\binom{2n}{n+1}-\binom{2n+1}{n}+\binom{2n}{n}.
\end{equation*}
\end{proof}

\begin{lemma}
\label{lem12}
The following identity is valid
\begin{equation}
J_{4}=\mbox{\bf res}_{tu}\{[\frac{(1-u^{4})(1+u)^{2n-1}u^{-n+1}(u-t(1+u)^{2})}
{(u-t/(1-t))^{2}(u-(1-t)/t))(1-t)^{2}}]_{|u| \gg |t| =\rho }\}=0.
\label{71}
\end{equation}
\end{lemma}

\begin{proof}
Let
\begin{equation*}
J_{4}=\mbox{\bf res}_{u}\{\frac{(1-u^{4})(1+u)^{2n-6}}{u^{n-1}}\times
\end{equation*}%
\begin{equation*}
\times \mbox{\bf res}_{t}[\frac{t}{\left( t-u/(1+u)^{2}\right)
\left(t-1/(1+u)\right)(t-u/(1+u))^{2}}]\},
\end{equation*}%
where $|u| \gg |t|=\rho $.

If now, in accordance with the condition $|u|\gg |t| =\rho $ for example,
$|t|=\rho_{1}=0,01,\,|u|=\rho_{2}=10$ and
$\left|u/(1+u)^{2}\right|\approx 10\gg \rho _{1}=0,01,$
$\left|u/(1+u) \right| \approx 1\gg \rho _{1}=0,01,$
$\left|1/(1+u) \right| \approx 0,1\gg \rho _{1}=0,01,$ then the
integrand in $t$ and the integral (\ref{43}) has no singularities inside
$|t|=\rho_{1}=0,01 $, and therefore the integral
(\ref{71}) is zero.
\end{proof}

\begin{remark}
Calculating the sum of $J_{1},J_{2}$ $J_{3} $ in closed form was held
up by appropriate residue theorem, and the results were confirmed by the
numerical test.

Two of them were calculated with the well-known theorem of
residues, while the integrals $J_{1}=J_{4}=0$ are trivial by Theorem
difficult to assess the full amount of the deduction, as the corresponding
residue at infinity is zero.
\end{remark}

From the formulas (\ref{66}) -- (\ref{68}) and (\ref{71}) we immediately
obtain the following formula for $T_{3}=J_{1}+J_{2}+J_{3}+J_{4}.$

\begin{lemma}
\label{lem13} For $n=3,4,\ldots $ the following identity is valid
\begin{equation}
T_{3}=2^{2n-1}+(n-1)-\frac{2}{n}\binom{2n}{n-2}-\binom{2n}{n+1}
-\binom{2n+1}{n}+\binom{2n}{n}.
\label{72}
\end{equation}
\end{lemma}

From the formulas (\ref{28}), (\ref{35}) and (\ref{72}) follows the validity
of the Main theorem.

As a result, using the fundamental theorem, we get the general formula
%(\ref{Main}) for enumeration of $\mathcal{D}$-invariant ideals of ring
(\ref{M1a}) for enumeration of $\mathcal{D}$-invariant ideals of ring
$R_{n}(K,J).$

%\subsection{The Part 2. Applications of the method coefficients in the
%theory of holomorphic functions in $\mathbb{C}^{n} $ }

%\subsection{Applications of the method coefficients in the
%theory of holomorphic functions in $\mathbb{C}^{n} $ and in the theory
%of cubature formulas}

%\subsubsection{4. The new short calculation of the Krivokolesko
%combinatorial sum and its some applications\protect\medskip}

\section{Calculation of multiple combinatorial sums in the theory of
golomorphic functions in $\mathbb{C}^{n} $}

In section 4 we solved the several summation problems. In sections 4.2, 4.3 and section 4.4 we
found
the simple new proof and generalization of the several multiple
combinatorial sums,
which originally arose in the theory of holomorphic functions in
$\mathbb{C}^{n} $ \cite{th, th1, EDK2011}.

\subsection{Introduction}

\begin{definition}
\label{D1}
Domain $G\subset C^{n}$ is called a linearly convex (\cite {AizenYuzh83},\,\S 8),
if for every point $z_{0}$\ its boundary $\partial G$\ there is complex
$(n-1)$-dimensional analytic plane passing through $z_{0}$\ and does not intersect $G$.
\end{definition}

Let in the space $\mathbb{C}^{n}$ set linearly convex polyhedron, ie,
bounded linearly convex domain $G=\{\,z:g^{l}(z,\overline{z})<0,\quad
l=1,\ldots ,N\},$ where the functions $g^{l}(z,\overline{z})$ are twice
continuously differentiable in a neighborhood of this area. The boundary
$\partial G$ of $G$ consists of faces
\begin{equation*}
S^{l}=\{\,z\in \overline{G}:g^{l}(z,\overline{z})=0\},\quad l=1,\ldots ,N.
\end{equation*}

\begin{definition}
\label{D2}
If at any non-empty edge
\begin{equation*}
S^{j_{1}\ldots j_{k}}=S^{j_{1}}\cap \ldots \cap S^{j_{k}}=
\{\,\zeta \in \partial G:g^{j_{1}}(\zeta ,\overline{\zeta })=0,\ldots ,g^{j_{k}}
(\zeta, \overline{\zeta })=0\}
\end{equation*}%
following inequality holds $\overline{\partial }g^{j_{1}}\wedge \ldots
\wedge \overline{\partial }g^{j_{k}}\not=0$ or, that the same,
\begin{equation*}
rank
\begin{pmatrix}
\frac{\partial g^{j_{1}}}{\partial \overline{\zeta }_{1}}
& \ldots & \frac{\partial g^{j_{1}}}{\partial \overline{\zeta }_{n}} \\
\ldots & \ddots & \ldots \\
\frac{\partial g^{j_{k}}}{\partial \overline{\zeta }_{1}}
& \ldots & \frac{\partial g^{j_{k}}}{\partial \overline{\zeta }_{n}}
\end{pmatrix}
=k.
\end{equation*}%
then $G$ has a piecewise regular boundary.
\end{definition}

In \cite{th} we obtained a new integral representation for holomorphic
functions on linear convex domains with piecewise regular boundary of a
bounded linear convex domain.

We received a number of identities, when we considered an example of
integration of holomorphic monomials of a piecewise regular boundary of a
bounded linear convex domain\newline
$G=\{z=(z_{1},z_{2},z_{3})\in \mathbb{C}^{3}:g^{1}(|z|)=-|z_{1}|+1<0,$
$g^{2}(|z|)=-|z_{2}|+1<0,$
$g^{3}(|z|)=-|z_{3}|+1<0,$ $g^{4}(|z|)=a_{1}|z_{1}|+a_{2}|z_{2}|+a_{3}|z_{3}|-r<0,$
$a_{i}>0,$ $i=1,2,3,$ $a_{1}+a_{2}+a_{3}<r\}$\newline
the most difficult of which in the notation $\alpha _{i}=\frac{a_{i}}{r},$ $%
i=1,2,3$ for integers $s_{1},s_{2},s_{3}\geq 0$ has the form:

\begin{theorem}
\label{T1}
For $0<\alpha _{j}<1,j=1,2,3;s_{j}\in Z_{+}$\ following identity
is valid%
\begin{equation*}
(1-\alpha_{2}-\alpha_{3})^{s_{1}+1}
\sum_{k=0}^{s_{2}}\sum_{l=0}^{s_{3}}\frac{(s_{1}+k+l)!}{s_{1}!k!l!}\alpha _{2}^{k}\alpha_{3}^{l}+
\end{equation*}%
\begin{equation*}
+(1-\alpha_{1}-\alpha_{3})^{s_{2}+1}\sum_{k=0}^{s_{1}}
\sum_{l=0}^{s_{3}}\frac{(s_{2}+k+l)!}{s_{2}!k!l!}\alpha_{1}^{k}\alpha_{3}^{l} +
\end{equation*}
\begin{equation*}
+(1-\alpha_{1}-\alpha_{2})^{s_{3}+1}\sum_{k=0}^{s_{1}}
\sum_{l=0}^{s_{2}}\frac{(s_{3}+k+l)!}{s_{3}!k!l!}\alpha_{1}^{k}\alpha_{2}^{l}+
\end{equation*}%
\begin{equation*}
-\frac{(s_{1}+s_{2}+1)!}{s_{1}!s_{2}!}\sum_{m=0}^{s_{2}}
\frac{(-1)^{m}\binom{s_{2}}{m}}{s_{1}+m+1}\left(\left(1-\frac{\alpha_{2}}{1-\alpha_{3}}\right)^{s_{1}+m+1}-
\left(\frac{\alpha_{1}}{1-\alpha _{3}}\right)^{s_{1}+m+1}\right)\times
\end{equation*}%
\begin{equation*}
\times (1-\alpha_{3})^{s_{1}+s_{2}+2}
\sum_{k=0}^{s_{3}}\binom{s_{1}+s_{2}+k+1}{k}\alpha_{3}^{k}+
\end{equation*}%
\begin{equation*}
-\frac{(s_{2}+s_{3}+1) !}{s_{2}!s_{3}!}\sum_{m=0}^{s_{3}}\frac{(-1)^{m}%
\binom{s_{3}}{m}}{s_{2}+m+1} \left( \left( 1-\frac{\alpha _{3}}{1-\alpha _{1}%
}\right) ^{s_{2}+m+1}-\left( \frac{\alpha _{2}}{1-\alpha _{1}}\right)
^{s_{2}+m+1}\right) \times
\end{equation*}%
\begin{equation*}
\times \left( 1-\alpha _{1}\right) ^{s_{2}+s_{3}+2}\sum_{k=0}^{s_{1}}\binom{%
s_{2}+s_{3}+k+1}{k}\alpha _{1}^{k}+
\end{equation*}%
\begin{equation*}
-\frac{\left( s_{1}+s_{3}+1\right) !}{s_{1}!s_{3}!}\sum_{m=0}^{s_{3}}\frac{%
\left( -1\right) ^{m}\binom{s_{3}}{m}}{s_{1}+m+1}\left( \left( 1-\frac{%
\alpha _{3}}{1-\alpha _{2}}\right) ^{s_{1}+m+1}-\left( \frac{\alpha _{1}}{%
1-\alpha _{2}}\right) ^{s_{1}+m+1}\right) \times
\end{equation*}%
\begin{equation*}
\times \left( 1-\alpha _{2}\right) ^{s_{1}+s_{3}+2}\sum_{k=0}^{s_{2}}\binom{%
s_{1}+s_{3}+k+1}{k}\alpha _{2}^{k}+
\end{equation*}%
\begin{equation}
+\frac{\left( s_{1}+s_{2}+s_{3}+2\right) !}{s_{1}!s_{2}!s_{3}!}\int_{\alpha
_{1}}^{1-\alpha _{2}-\alpha _{3}}\int_{\alpha _{2}}^{1-\alpha
_{3}-x}x^{s_{1}}y^{s_{2}}\left( 1-x-y\right) ^{s_{3}}dx\wedge dy\equiv 1.
\label{2F}
\end{equation}
\end{theorem}

If $\alpha _{1}+\alpha _{2}+\alpha _{3}=1,$ then (\ref{2F}) is equivalent to
the identity
%{\footnotesize
\begin{equation*}
\alpha_{1}^{s_{1}+1}\sum_{k=0}^{s_{2}}\sum_{l=0}^{s_{3}}\frac{(s_{1}+k+l)!}{%
s_{1}!k!l!}\alpha_{2}^{k}\alpha_{3}^{l}+
\end{equation*}
\begin{equation*}
+\alpha_{2}^{s_{2}+1}\sum_{k=0}^{s_{1}}\sum_{l=0}^{s_{3}}\frac{(s_{2}+k+l)!}{%
s_{2}!k!l!}\alpha_{1}^{k}\alpha_{3}^{l}+
\end{equation*}%
\begin{equation*}
+\alpha_{3}^{s_{3}+1}\sum_{k=0}^{s_{1}}\sum_{l=0}^{s_{2}}\frac{(s_{3}+k+l) !%
}{s_{3}!k!l!}\alpha_{1}^{k}\alpha_{2}^{l}=1,
\end{equation*}
which allows us to formulate the following theorem:

\begin{theorem}
\label{T2}
If the complex parameters $z_{1},\ldots ,z_{n}$ satisfy the
relation%
\begin{equation}
z_{1}+\ldots +z_{n}=1,
\label{K1}
\end{equation}
\textit{\ then for any values {}{} }$s_{1},\ldots ,s_{n}=0,1,2,\ldots $ the
following identity is valid%
\begin{equation*}
z_{1}^{s_{1}+1}\sum_{j_{2}=0}^{s_{2}}\ldots \sum_{j_{n}=0}^{s_{n}}
\binom{s_{1}+\sum_{i\neq 1}j_{i}}{s_{1},j_{2},\ldots ,j_{n}}z_{2}^{j_{2}}\cdots
z_{n}^{j_{n}}+\ldots +
\end{equation*}%
\begin{equation}
+z_{n}^{s_{n}+1}\sum_{j_{1}=0}^{s_{1}}\ldots \sum_{j_{n-1}=0}^{s_{n-1}}%
\binom{s_{n}+\sum_{i\neq n}j_{i}}{s_{n},j_{1},\ldots ,j_{n-1}}%
z_{1}^{j_{1}}\cdots z_{n-1}^{j_{n-1}}=1.
\label{K2}
\end{equation}
\end{theorem}

In the article of D. Zeilberger \cite{Zeil83} provides the following
identity {\footnotesize
\begin{equation*}
\sum_{i=1}^{k}\underset{j\neq i}{\sum_{0\leq \alpha _{j}\leq n-1}}
\frac{(\alpha _{1}+\ldots +\alpha _{i-1}+(n-1)+\alpha _{i+1}+\ldots \alpha _{k})!}
{\alpha _{1}!\ldots \alpha _{i-1}!(n-1)!\alpha _{i+1}!\ldots \alpha _{k}!}%
p_{1}^{\alpha _{1}}\ldots p_{i-1}^{\alpha _{i-1}}p_{i}^{n}p_{i+1}^{\alpha
_{i+1}}\ldots p_{k}^{\alpha _{k}}=1,
\end{equation*}%
} for $p_{1}+\ldots +p_{k}=1.$

For $n=2$ the identity (\ref{K2}) in a somewhat altered form can be proved
in \cite{Shell82} (V. Shelkovich, 1982).

02 May 2015 professor S. LJ. Damjnovic from Belgrad paid attention of the authors to the articles 
\cite{Koornwinder1} and \cite{Koornwinder2}. From which it follows that identity (\ref{K2}) under 
$n=2$ and real values of parametres $z_1, z_2$ from 1960 was known as identity of Chaundy and Bullard \cite{Chaundy}.
We note that in \cite{Koornwinder2} a detailed story connected with this identity is given.  
%02 мая 2015 года профессор S.LJ. Damjnovic из Белграда обратил внимание одного из авторов на 
%статьи \cite{Koornwinder1} и \cite{Koornwinder2} из которых следует, что тождество (\ref{K2}) при $n=2$ и 
%вещественных значениях параметров $z_1, z_2$ c 1960 известно как тождество Chaundy and Bullard \cite{Chaundy}.
 
%\textbf{Proof of} \textit{Theorem \ref{T1}}
\subsection{Lemmas}

\begin{lemma}
\label{L3}
The following integral representation is valid%
\begin{equation*}
\sum_{k=0}^{s_{3}}\binom{s_{1}+s_{2}+k+1}{k}\alpha ^{k}=\mathbf{res}_{z}%
\frac{(1-z^{-s_{3}-1})}{(z-1)(1-\alpha z)^{s_{1}+s_{2}+2}}=
\end{equation*}
\begin{equation}
=\mathbf{res}_{z}\frac{(1-\alpha z)^{-s_{1}-s_{2}-2}}{z^{s_{3}+1}(1-z)}.
\label{21F}
\end{equation}
\end{lemma}

\begin{proof}
We have directly
\begin{equation*}
\mathbf{res}_{z}(1-\alpha z)^{-s_{1}-s_{2}-2}\frac{z^{-s_{3}-1}}{1-z}=
\end{equation*}
\begin{equation*}
=\mathbf{res}_{z}\left((\sum_{k=0}^{\infty }\alpha^{k}\binom{s_{1}+s_{2}+k+1}{k}
z^{k})(\sum_{k=0}^{\infty }z^{k})z^{-s_{3}-1})\right)
\end{equation*}%
\begin{equation*}
=\mathbf{res}_{z}\sum_{n=0}^{\infty }z^{n}\left(\sum_{k=0}^{n}\alpha^{k}
\binom{s_{1}+s_{2}+k+1}{k}\right) z^{-s_{3}-1}=
\sum_{k=0}^{s_{3}}\binom{s_{1}+s_{2}+k+1}{k}\alpha^{k}.
\end{equation*}%
On the other side%
\begin{equation*}
\sum_{k=0}^{s_{3}}\binom{s_{1}+s_{2}+k+1}{k}\alpha^{k}=
\sum_{k=0}^{s_{3}}\mathbf{res}_{z}\{(1-\alpha z)^{-s_{1}-s_{2}-2}z^{-k-1}\}=
\end{equation*}
\begin{equation*}
=\mathbf{res}_{z}(1-\alpha z)^{-s_{1}-s_{2}-2}\frac{(1-z^{-s_{3}-1})}{z-1}.
\end{equation*}
\end{proof}

We consider (\ref{2F}). It is easy to see that due to the symmetry of
similar terms in the sum of its parameters relative to the latter identity
can be represented in the form:%
\begin{equation*}
S(s_{1},s_{2},s_{3};\alpha_{2},\alpha_{3})+S(s_{2},s_{1},s_{3};\alpha_{1},\alpha_{3})
+S(s_{3},s_{1},s_{2};\alpha_{1},\alpha_{2})+
\end{equation*}
%\textit{\ }%
\begin{equation*}
-(T(s_{1},s_{2},s_{3};\alpha_{1},\alpha_{2},\alpha_{3})
+T(s_{1},s_{3},s_{2};\alpha_{2},\alpha_{3},\alpha_{1})
+T(s_{1},s_{3},s_{2};\alpha_{1},\alpha_{3},\alpha_{2}))+
\end{equation*}%
\begin{equation}
+R(s_{1},s_{2},s_{3})=1,
\label{3F}
\end{equation}%
where%
\begin{equation}
S(s_{1},s_{2},s_{3};\alpha_{2},\alpha_{3})=
(1-\alpha_{2}-\alpha_{3})^{s_{1}+1}\sum_{k=0}^{s_{2}}\sum_{l=0}^{s_{3}}
\frac{(s_{1}+k+l)!}{s_{1}!k!l!}\alpha_{2}^{k}\alpha_{3}^{l},
\label{4F}
\end{equation}%
\begin{equation*}
T(s_{1},s_{2},s_{3};\alpha_{1},\alpha_{2},\alpha_{3})=
(1-\alpha_{3})^{s_{1}+s_{2}+2}\frac{(s_{1}+s_{2}+1)!}{s_{1}!s_{2}!}
\sum_{m=0}^{s_{2}}\frac{(-1)^{m}\binom{s_{2}}{m}}{s_{1}+m+1}\times
\end{equation*}%
{\footnotesize
\begin{equation}
\times \left(\left(1-\frac{\alpha_{2}}{1-\alpha_{3}}\right)^{s_{1}+m+1}-
\left(\frac{\alpha_{1}}{1-\alpha_{3}}\right)^{s_{1}+m+1}\right)
\sum_{k=0}^{s_{3}}\binom{s_{1}+s_{2}+k+1}{k}\alpha _{3}^{k},
\label{7F}
\end{equation}
} {\footnotesize
\begin{equation}
R(s_{1},s_{2},s_{3})=\frac{(s_{1}+s_{2}+s_{3}+2)!}{s_{1}!s_{2}!s_{3}!}
\int_{\alpha_{1}}^{1-\alpha_{2}-\alpha_{3}}
\int_{\alpha_{2}}^{1-\alpha_{3}-x}x^{s_{1}}y^{s_{2}}(1-x-y)^{s_{3}}dx\wedge dy.
\label{10F}
\end{equation}%
} %\medskip \bigskip

%\paragraph{Calculating the sum $S_{1}(s_{1},s_{2},s_{3};\protect\alpha _{2},%
%\protect\alpha _{3}).$}
\subsubsection{Calculation the sum $S_{1}(s_{1},s_{2},s_{3};\alpha_{2},
\alpha_{3}).$}
By direct verification it is easy to verify the validity of the following
formulas.

\begin{lemma}
\label{L1}
The following identities are valid{\footnotesize
\begin{equation*}
S_{1}(s_{1},s_{2},s_{3})=(1-\alpha_{2}-\alpha_{3})
\mathbf{res}_{z_{1},z_{2},z_{3}}
\frac{\left(1-(1-\alpha_{2}-\alpha _{3})z_{1}-\alpha_{2}z_{2}-\alpha_{3}z_{3}\right)^{-1}}
{z_{1}^{s_{1}+1}z_{2}^{s_{2}+1}z_{3}^{s_{3}+1}(1-z_{2})(1-z_{3})}.
\end{equation*}%
}
\begin{equation*}
{\mathbf{S}_{1}}(u_{1},u_{2},u_{3}):=
\sum_{s_{1},s_{2},s_{3}=0}^{\infty }S_{1}(s_{1},s_{2},s_{3})
u_{1}^{s_{1}}u_{2}^{s_{2}}u_{3}^{s_{3}}=
\end{equation*}%
\begin{equation}
=\frac{(1-\alpha_{2}-\alpha_{3})}{(1-u_{2})\cdot
(1-u_{3})\cdot \left(1-(1-\alpha_{2}-\alpha_{3})u_{1}-\alpha_{2}u_{2}-\alpha_{3}u_{3}\right) }.
\label{L1.2}
\end{equation}
\end{lemma}
\begin{proof}
The integral representation for $S_{1}(s_{1},s_{2},s_{3})$ is
\begin{equation*}
S_{1}(s_{1},s_{2},s_{3};\alpha_{2},\alpha_{3})=
\end{equation*}
\begin{equation*}
=(1-\alpha_{2}-\alpha_{3})^{s_{1}+1}\sum_{j_{2}=0}^{s_{2}}\sum_{j_{3}=0}^{s_{3}}
\frac{\left(s_{1}+(s_{2}-j_{2})+(s_{3}-j_{3})\right)!}{s_{1}!(s_{2}-j_{2})!(s_{3}-j_{3})!}
\alpha_{2}^{(s_{2}-j_{2})}\alpha_{3}^{(s_{3}-j_{3})}=
\end{equation*}%
\begin{equation*}
=(1-\alpha_{2}-\alpha_{3}) \mathbf{res}_{z_{1},z_{2},z_{3}}
\sum_{j_{2}=0}^{s_{2}}\sum_{j_{3}=0}^{s_{3}}
\frac{\left(1-(1-\alpha_{2}-\alpha_{3})z_{1}-\alpha_{2}z_{2}-\alpha_{3}z_{3}\right)^{-1}}
{z_{1}^{s_{1}+1}z_{2}^{s_{2}-j_{2}+1}z_{3}^{s_{3}-j_{3}+1}}=
\end{equation*}%
{\footnotesize
\begin{equation*}
=(1-\alpha _{2}-\alpha_{3}) \mathbf{res}_{z_{1},z_{2},z_{3}}
[\frac{\left(1-(1-\alpha_{2}-\alpha_{3})z_{1}-\alpha_{2}z_{2}-\alpha _{3}z_{3}\right)^{-1}}
{z_{1}^{s_{1}+1}z_{2}^{s_{2}+1}z_{3}^{s_{3}+1}}
\sum_{j_{2}=0}^{\infty}z_{2}^{j_{2}}\sum_{j_{3}=0}^{\infty }z_{3}^{j_{3}}]
\end{equation*}
}
\begin{equation}
=(1-\alpha_{2}-\alpha_{3}) \mathbf{res}_{z_{1},z_{2},z_{3}}
\frac{\left(1-(1-\alpha_{2}-\alpha_{3})z_{1}-\alpha_{2}z_{2}-\alpha_{3}z_{3}\right)^{-1}}
{z_{1}^{s_{1}+1}z_{2}^{s_{2}+1}z_{3}^{s_{3}+1}(1-z_{2})(1-z_{3})}.
\label{L1.1}
\end{equation}%
We find the generating function for $S_{1}(s_{1},s_{2},s_{3})$
using the formula (\ref{L1.1})
\begin{equation*}
{\mathbf{S}_{1}}(u_{1},u_{2},u_{3}):=
\sum_{s_{1},s_{2},s_{3}=0}^{\infty }
S_{1}(s_{1},s_{2},s_{3})u_{1}^{s_{1}}u_{2}^{s_{2}}u_{3}^{s_{3}}=
\end{equation*}%
{\footnotesize
\begin{equation*}
=\sum_{s_{1},s_{2},s_{3}=0}^{\infty }(1-\alpha_{2}-\alpha_{3})
\mathbf{res}_{z_{1},z_{2},z_{3}}
\frac{\left(1-(1-\alpha_{2}-\alpha_{3})z_{1}-\alpha_{2}z_{2}-\alpha_{3}z_{3}\right)^{-1}}
{z_{1}^{s_{1}+1}z_{2}^{s_{2}+1}z_{3}^{s_{3}+1}(1-z_{2})(1-z_{3})}
u_{1}^{s_{1}}u_{2}^{s_{2}}u_{3}^{s_{3}}=
\end{equation*}
} (rule of substitution $z_{i}=u_{i})$%
\begin{equation}
=\frac{(1-\alpha _{2}-\alpha _{3})}{(1-u_{2})\cdot
(1-u_{3})\cdot \left(1-( 1-\alpha_{2}-\alpha_{3})u_{1}-\alpha_{2}u_{2}-\alpha _{3}u_{3}\right)}.
\label{L2.0}
\end{equation}
\end{proof}

\begin{lemma}
\label{L2}The following relations are valid
%\frac{(s_{2}+s_{3}+1) !(1-\alpha_{1})}{s_{2}!s_{3}!}\cdot
{\footnotesize
\begin{equation}
{\mathbf{S}_{2}}(u_{1},u_{2},u_{3})=\frac{(1-\alpha_{1}-\alpha_{3})}
{(1-u_{1})\cdot (1-u_{3})\cdot
\left(1-\alpha_{1}u_{1}-(1-\alpha_{1}-\alpha_{3})u_{2}-\alpha_{3}u_{3}\right)},
\label{L2.1}
\end{equation}%
\begin{equation}
{\mathbf{S}_{3}}(u_{1},u_{2},u_{3})=\frac{(1-\alpha_{1}-\alpha_{2})}
{(1-u_{1})\cdot (1-u_{2})\cdot
\left( 1-\alpha_{1}u_{1}-\alpha_{2}u_{2}-(1-\alpha_{1}-\alpha_{2})u_{3}\right)}.
\label{L2.2}
\end{equation}
}
\end{lemma}

The proof is analogous Lemma \ref{L1}.

%\paragraph{Integral representation of and calculation of sums%
%\protect\begin{equation*}

%\subsubsection{Integral representation of and calculation of sums
%$T_{i}(s_{1},s_{2},s_{3};\alpha)$, $i=1,\,2,3.$}
\subsubsection{Calculation of sums
$T_{i}(s_{1},s_{2},s_{3};\alpha)$, $i=1,\,2,3.$}

\begin{lemma}
Integral representation for $T_{1}(s_{1},s_{2},s_{3};\alpha_{1},\alpha_{2},\,\alpha_{3})$ is equal%
{\small
\begin{equation*}
T_{1}=\frac{(s_{2}+s_{3}+1)!(1-\alpha_{1})}{s_{2}!s_{3}!}\cdot\{(1-\alpha_{1}-\alpha_{3})^{s_{2}+1}\int_{0}^{1} t^{s_{2}}
(1-\alpha_{1}-t(1-\alpha_{1}-\alpha_{3}))^{s_{3}}dt+
\end{equation*}
\begin{equation}
\label{L3.1}
-\alpha_{2}^{s_{2}+1}\int_{0}^{1}t^{s_{2}} (1-\alpha_{1}-\alpha_{2}t)^{s_{3}}dt\}
\cdot\mathbf{res}_{z_{1}}(1-\alpha _{1}z_{1})^{-s_{2}-s_{3}-2}
\frac{z_{1}^{-s_{1}-1}}{1-z_{1}}.
\end{equation}}
\end{lemma}

\begin{proof}
\begin{equation*}
T_{1}=\sum_{m=0}^{s_{3}}\frac{(-1)^{m}
\left((1-\alpha_{1}-\alpha_{3})^{s_2+m+1}(1-\alpha_{1})^{s_{3}-m}-\alpha_{2}^{s_{2}+m+1}(1-\alpha_{1})^{s_{3}-m}\right)}
{s_{2}+m+1}
 \times
\end{equation*}%
\begin{equation*}
\times
\binom{s_{3}}{m}\cdot\frac{(s_{2}+s_{3}+1)!(1-\alpha_1)}{s_{2}!s_{3}!}
\sum_{k=0}^{s_{1}}\binom{s_{2}+s_{3}+k+1}{k}\alpha_{1}^{k}=
\end{equation*}%
{\footnotesize
\begin{equation*}
=\frac{(s_{2}+s_{3}+1)!(1-\alpha_{1})}{s_{2}!s_{3}!}\cdot
\left\{\int_{0}^{1}\sum_{m=0}^{s_{3}}\binom{s_{3}}{m}(-1)^{m}t^{s_{2}+m}(1-\alpha_{1}-\alpha_{3})^{s_{2}+1+m}
(1-\alpha_{1})^{s_{3}-m}dt+\right.
\end{equation*}%
\begin{equation*}
-\left.\int_{0}^{1}\sum_{m=0}^{s_{3}}\binom{s_{3}}{m}(-1)^{m}\alpha_{2}^{s_{2}+1+m}t^{s_{2}+m}
(1-\alpha_{1})^{s_{3}-m}dt\right\}\cdot \mathbf{res}_{z_{1}}(1-\alpha_{1}z_{1})^{-s_{2}-s_{3}-2}
\frac{z_{1}^{-s_{1}-1}}{1-z_{1}}=
\end{equation*}
\begin{equation*}
=\frac{(s_{2}+s_{3}+1)!(1-\alpha_{1})}{s_{2}!s_{3}!}\cdot \left\{(1-\alpha_{1}-\alpha_{3})^{s_{2}+1}
\int_{0}^{1}t^{s_{2}}\left(1-\alpha_{1}-t(1-\alpha_{1}-\alpha_{3})\right)^{s_{3}}dt+\right.
\end{equation*}}
\begin{equation*}
-\left.\alpha_{2}^{s_{2}+1}\int_{0}^{1}t^{s_{2}}(1-\alpha_{1}-\alpha_{2}t)^{s_{3}}dt\right\}
\cdot \mathbf{res}_{z_{1}}(1-\alpha _{1}z_{1})^{-s_{2}-s_{3}-2}\frac{z_{1}^{-s_{1}-1}}{1-z_{1}}.
\end{equation*}
\end{proof}

Analogously we find integral representations for
$T_{2}(s_{1},s_{2},s_{3};\alpha_{1},\alpha_{2},\alpha_{3}) $ and
$T_{3}(s_{1},s_{2},s_{3};\alpha_{1},\alpha_{2},\alpha_{3})$:%
\begin{equation*}
T_{2}=\frac{(s_{1}+s_{3}+1)!}{s_{1}!s_{3}!}\{(1-\alpha_{2}-\alpha_{3})^{s_{1}+1}\int_{0}^{1}t^{s_{1}}
(1-\alpha_{2}-t(1-\alpha_{2}-\alpha_{3}))^{s_{3}}dt
\end{equation*}
\begin{equation}
-\alpha_{1}^{s_{1}+1}\int_{0}^{1}t^{s_{1}}(1-\alpha_{2}-t\alpha_{1})^{s_{3}}dt\}\cdot(1-\alpha_{2})\cdot
 \mathbf{res}_{z_{2}}(1-\alpha _{2}z_{2})^{-s_{1}-s_{3}-2}\frac{z_{2}^{-s_{2}-1}}{1-z_{2}}.
\label{L3.2}
\end{equation}
\begin{equation*}
T_{3}=\frac{(s_{1}+s_{2}+1)!}{s_{1}!s_{2}!}(1-\alpha_{3})\{(1-\alpha_{3}-\alpha_{2})^{s_{1}+1}
\int_{0}^{1}t^{s_1}(1-\alpha_{3}-t(1-\alpha_{3}-\alpha_{2}))^{s_{2}}dt
\end{equation*}
\begin{equation}
-\alpha_{1}^{s_{1}+1}\int_{0}^{1}t^{s_{1}}(1-\alpha_{3}-\alpha_{1}t)^{s_{2}}dt\}
\times \mathbf{res}_{z_{3}}(1-\alpha_{3}z_{3})^{-s_{1}-s_{2}-2}\frac{z_{3}^{-s_{3}-1}}{1-z_{3}}.
\label{L3.3}
\end{equation}

\begin{lemma}
\label{L5}
The following identity is valid%
\begin{equation*}
\mathbf{T}_{1}\mathbf{(}u_{1},u_{2},u_{3})=\sum_{s_{1},s_{2},s_{3}=0}^{\infty }T_{1}(s_{1},s_{2},s_{3})
u_{1}^{s_{1}}u_{2}^{s_{2}}u_{3}^{s_{3}}=\frac{(1-\alpha_{1})}{(1-u_{1})}\times
\end{equation*}
{\footnotesize
\begin{equation}
\label{L7.0}
=\times \frac{(1-\alpha_{1}-\alpha_{2}-\alpha_{3})}
{1-\alpha_{1}u_{1}-(1-\alpha_{1}-\alpha_{3})u_{2}-\alpha _{3}u_{3}}\cdot
\frac{1}{1-\alpha_{1}u_{1}-\alpha_{2}u_{2}-u_{3}(1-\alpha_{1}-\alpha_{2})}
\end{equation}
}
\end{lemma}

\begin{proof}
The generating function for $T_{1}(s_{1}s_{2},s_{3})$
\begin{equation*}
\mathbf{T}_{1}(u_{1}, u_{2}, u_{3})=\sum_{s_{1},s_{2},s_{3}=0}^{\infty }T_{1}(s_{1},s_{2},s_{3})
u_{1}^{s_{1}}u_{2}^{s_{2}}u_{3}^{s_{3}}=
\end{equation*}%
\begin{equation*}
=(1-\alpha _{1}) \sum_{s_{1},s_{2},s_{3}=0}^{\infty }\frac{(s_{2}+s_{3}+1)!}{s_{2}!s_{3}!}
\mathbf{res}_{z_{1}}[(1-\alpha_{1}z_{1})^{-s_{2}-s_{3}-2}\frac{z_{1}^{-s_{1}-1}}{1-z_{1}}]\times
\end{equation*}%
{\footnotesize
\begin{equation*}
\times \int_{0}^{1}\{(1-\alpha_{1}-\alpha_{3})^{s_{2}+1}t^{s_{2}}
(1-\alpha_{1}-t(1-\alpha_{1}-\alpha_{3}))^{s_{3}}-\alpha_{2}^{s_{2}+1}t^{s_{2}}
(1-\alpha_{1}-\alpha_{2}t)^{s_{3}}\}dt\cdot u_{1}^{s_{1}}u_{2}^{s_{2}}u_{3}^{s_{3}}=
\end{equation*}%
}
\begin{equation*}
by\text{ }the\text{ }rule\text{ }of\text{ }substitution)
\end{equation*}%
\begin{equation*}
=(1-\alpha_{1})\sum_{s_{2},s_{3}=0}^{\infty }\frac{(s_{2}+s_{3}+1)!}{s_{2}!s_{3}!}\cdot
\frac{(1-\alpha_{1}u_{1})^{-s_{2}-s_{3}-2}}{1-u_{1}}\times
\end{equation*}%
{\small
\begin{equation*}
\times \int_{0}^{1}\{(1-\alpha_{1}-\alpha_{3})^{s_{2}+1}t^{s_{2}}
(1-\alpha_{1}-t(1-\alpha_{1}-\alpha_{3}))^{s_{3}}-\alpha_{2}^{s_{2}+1}t^{s_{2}}
(1-\alpha_{1}-\alpha_{2}t)^{s_{3}}\}dt\cdot u_{2}^{s_{2}}u_{3}^{s_{3}}=
\end{equation*}%
\begin{equation*}
=\frac{(1-\alpha_{1})}{(1-u_{1})(1-\alpha_{1}u_{1})^{2}}
\sum_{s_{2},s_{3}=0}^{\infty}\mathbf{res}_{z_{2},z_{3}}\frac{(1-z_{2}-z_{3})^{-2}}
{z_{2}^{s_{2}+1}z_{3}^{s_{3}+1}}\cdot
\frac{u_{2}^{s_{2}}u_{3}^{s_{3}}}{(1-\alpha_{1}u_{1})^{s_{2}}(1-\alpha_{1}u_{1})^{s_{3}}}\times
\end{equation*}%
\begin{equation*}
\times \int_{0}^{1}\{(1-\alpha_{1}-\alpha_{3})^{s_{2}+1}t^{s_{2}}
(1-\alpha_{1}-t(1-\alpha_{1}-\alpha_{3}))^{s_{3}}-\alpha_{2}^{s_{2}+1}t^{s_{2}}
(1-\alpha_{1}-\alpha_{2}t)^{s_{3}}\}dt=
\end{equation*}%
\begin{equation*}
=\frac{(1-\alpha_{1})(1-\alpha_{1}-\alpha_{3})}
{(1-u_{1})(1-\alpha_{1}u_{1})^{2}}\int_{0}^{1}\left(1-\frac{u_{2}t(1-\alpha_{1}-\alpha_{3})}
{(1-\alpha_{1}u_{1})}-\frac{u_{3}(1-\alpha_{1}-t(1-\alpha_{1}-\alpha_{3}))}
{(1-\alpha_{1}u_{1})}\right)^{-2}dt+
\end{equation*}%
\begin{equation*}
-\frac{(1-\alpha_{1})\alpha_{2}}{(1-u_{1})(1-\alpha_{1}u_{1})^{2}}
\int_{0}^{1}\left(1-\frac{u_{2}\alpha _{2}t}{(1-\alpha _{1}u_{1})}-
\frac{u_{3}(1-\alpha_{1}-\alpha_{2}t)}{(1-\alpha_{1}u_{1})}\right)^{-2}dt=
\end{equation*}%
\begin{equation*}
=\frac{(1-\alpha_{1})}{(1-u_{1})}\cdot \frac{(1-\alpha_{1}-\alpha_{2}-\alpha_{3})}
{1-\alpha_{1}u_{1}-(1-\alpha_{1}-\alpha_{3})u_{2}-\alpha _{3}u_{3}}\cdot
\frac{1}{1-\alpha_{1}u_{1}-\alpha_{2}u_{2}-u_{3}(1-\alpha_{1}-\alpha_{2})}.  %\label{L7.0}
\end{equation*}%
}
\end{proof}

\begin{lemma}
\label{L7}The following relations are valid
{\footnotesize
\begin{equation}
\mathbf{T}_{2}=%\frac{(1-\alpha_{2})}{(1-u_{2})}\cdot
\frac{(1-\alpha_{2})(1-u_{2})^{-1}(1-\alpha_{1}-\alpha_{2}-\alpha_{3})}
{1-(1-\alpha_{2}-\alpha_{3})u_{1}-\alpha_{2}u_{2}-\alpha_{3}u_{3}}
\cdot \frac{1}{1-\alpha_{1}u_{1}-\alpha_{2}u_{2}-u_{3}(1-\alpha_{1}-\alpha_{2})},
\label{L7.1}
\end{equation}%
\begin{equation}
\mathbf{T}_{3}=%\frac{(1-\alpha_{3})}{(1-u_{3})}\cdot
\frac{(1-\alpha_{3})(1-u_{3})^{-1}(1-\alpha_{1}-\alpha_{2}-\alpha_{3})}
{1-(1-\alpha_{2}-\alpha_{3}) u_{1}-\alpha_{2}u_{2}-\alpha_{3}u_{3}}\cdot
\frac{1}{1-\alpha_{1}u_{1}-(1-\alpha_{1}-\alpha_{3}) u_{2}-\alpha_{3}u_{3}}.
\label{L7.2}
\end{equation}
}
\end{lemma}

The proof is analogous Lemma \ref{L5}.

%\paragraph{Calculating the sum $R\left( s_{1},s_{2},s_{3}\right) $}
\subsubsection{Calculation the sum $R(s_{1},s_{2},s_{3}) $}

\begin{lemma}
\label{R1}
The following identities are valid
\begin{equation}
\label{r11}
R(s_{1},s_{2},s_{3})=2\int_{\alpha_{1}}^{1-\alpha_{2}-\alpha_{3}}
\int_{\alpha_{2}}^{1-\alpha_{3}-x}\mathbf{res}_{t_{1},t_{2},t_{3}}
\frac{x^{s_{1}}y^{s_{2}}(1-x-y)^{s_{3}}dxdy}
{(1-t_{1}-t_{2}-t_{3})^{-3} t_{1}^{s_{1}+1}t_{2}^{s_{2}+1}t_{3}^{s_{3}+1}}.
%\frac{(1-t_{1}-t_{2}-t_{3})^{-3}}{t_{1}^{s_{1}+1}t_{2}^{s_{2}+1}t_{3}^{s_{3}+1}}
%x^{s_{1}}y^{s_{2}}(1-x-y)^{s_{3}}dxdy.
\end{equation}
\begin{equation*}
\mathbf{R(}u_{1},u_{2},u_{3})=%\frac{1}{(1-u_{1}(1-\alpha_{2}-\alpha_{3})-\alpha_{2}u_{2}-\alpha_{3}u_{3})}
\frac{1}{(1-\alpha_{1}u_{1}-u_{2}(1-\alpha_{1}-\alpha_{3})-\alpha_{3}u_{3})}
\times
\end{equation*}
\begin{equation}
\label{T2.1}
\times\!\!\frac{(\alpha_{1}+\alpha_{2}+\alpha_{3}-1)^{2}}
{\!(1-u_{1}(1-\alpha_{2}-\alpha_{3})-\alpha_{2}u_{2}-\alpha_{3}u_{3})
(1-\alpha_{1}u_{1}-\alpha_{2}u_{2}-u_{3}(1-\alpha_{1}-\alpha_{2})}.
\end{equation}
\end{lemma}
\begin{proof}
We find an integral representation for $R(s_{1},s_{2},s_{3})$%
\begin{equation*}
R(s_{1},s_{2},s_{3}):=\frac{(s_{1}+s_{2}+s_{3}+2)!}
{s_{1}!s_{2}!s_{3}!}\int\limits_{\alpha_{1}}^{1-\alpha_{2}-\alpha_{3}}
\int\limits_{\alpha_{2}}^{1-\alpha_{3}-x}\!x^{s_{1}}y^{s_{2}}(1-x-y)^{s_{3}}dxdy=
\end{equation*}
\begin{equation*}
=2\cdot \int_{\alpha _{1}}^{1-\alpha_{2}-\alpha_{3}}
\int_{\alpha_{2}}^{1-\alpha _{3}-x}\mathbf{res}_{t_{1},t_{2},t_{3}}
\frac{x^{s_1}y^{s_2}(1-x-y)^{s_3}dx\,dy}
{(1-t_{1}-t_{2}-t_{3})^3 t_{1}^{s_{1}+1}t_{2}^{s_{2}+1}t_{3}^{s_{3}+1}}%
\end{equation*}%
Generating function for $R(s_{1},s_{2},s_{3}) $ is
\begin{equation*}
\mathbf{R(}u_{1},u_{2},u_{3}):=\sum_{s_{1},s_{2},s_{3}=0}^{\infty }
R(s_{1},s_{2},s_{3}) u_{1}^{s_{1}}u_{2}^{s_{2}}u_{3}^{s_{3}}=
\end{equation*}
\begin{equation*}
=2\int_{\alpha_{1}}^{1-\alpha_{2}-\alpha_{3}}
\int_{\alpha_{2}}^{1-\alpha_{3}-x}
\mathbf{res}_{t_{1},t_{2},t_{3}}\sum_{s_{1},s_{2},s_{3}=0}^{\infty }
\frac{x^{s_{1}}y^{s_{2}}(1-x-y)^{s_{3}}u_{1}^{s_{1}}u_{2}^{s_{2}}u_{3}^{s_{3}}dxdy}
{(1-t_{1}-t_{2}-t_{3})^{3}\,t_{1}^{s_{1}+1}t_{2}^{s_{2}+1}t_{3}^{s_{3}+1}}=
\end{equation*}
\begin{equation*}
=2\int_{\alpha_{1}}^{1-\alpha_{2}-\alpha_{3}}
\int_{\alpha_{2}}^{1-\alpha_{3}-x}\frac{dxdy}{(1-xu_{1}-yu_{2}-(1-x-y)u_{3})^{3}}=
\end{equation*}%
\begin{equation*}
=\frac{-1}{(u_{3}-u_{2})}\cdot \int_{\alpha_{1}}^{1-\alpha_{2}-\alpha_{3}}%
\left.\frac{1}{(1-u_{3}+x(u_{3}-u_{1})+y(u_{3}-u_{2}))^{2}}\right|_{\alpha_{2}}^{1-\alpha _{3}-x}dx=
\end{equation*}
\begin{equation*}
=\frac{1}{(1-u_{1}(1-\alpha_{2}-\alpha_{3})-\alpha_{2}u_{2}-\alpha_{3}u_{3})}\times
\end{equation*}
{\footnotesize
\begin{equation*}
\times \frac{(\alpha_{1}+\alpha_{2}+\alpha_{3}-1)^{2}}
{(1-\alpha_{1}u_{1}-u_{2}(1-\alpha_{1}-\alpha_{3})-\alpha_{3}u_{3})
(1-\alpha_{1}u_{1}-\alpha_{2}u_{2}-u_{3}(1-\alpha_{1}-\alpha_{2}))}.
\end{equation*}
}
\end{proof}

\begin{lemma}
\label{L8}
The following formula is valid%
\begin{equation}
\mathbf{S}_{1}+\mathbf{S}_{2}+\mathbf{S}_{3}-(\mathbf{T}_{1}+\mathbf{T}_{2}+%
\mathbf{T}_{3})+\mathbf{R}=\frac{1}{(1-u_{1})(1-u_{2})(1-u_{3})}.
\label{SR}
\end{equation}
\end{lemma}

\begin{proof}
Substitute in (\ref{3F}) formulas (\ref{L2.0})--(\ref{L2.2}),
(\ref{L7.0})--(\ref{L7.2}), (\ref{T2.1}) and get a true equality. We note that methods were
used to overcome the technical difficulties encountered in the proof. Denote%
{\footnotesize
\begin{equation*}
A:=1-(1-\alpha_{2}-\alpha_{3}) u_{1}-\alpha_{2}u_{2}-\alpha_{3}u_{3}=
1-u_{1}+\alpha_{2}(u_{1}-u_{2})+\alpha_{3}(u_{1}-u_{3})=1-u_{1}+A_{1};
\end{equation*}%
\begin{equation*}
B:=1-\alpha_{1}u_{1}-(1-\alpha_{1}-\alpha_{3})u_{2}-\alpha_{3}u_{3}=
1-u_{2}-\alpha_{1}(u_{1}-u_{2})+\alpha_{3}(u_{2}-u_{3})=1-u_{2}+B_{1};
\end{equation*}
\begin{equation*}
C:=1-\alpha_{1}u_{1}-\alpha_{2}u_{2}-(1-\alpha_{1}-\alpha_{2})u_{3}=
1-u_{3}-\alpha_{1}(u_{1}-u_{3})-\alpha_{2}(u_{2}-u_{3})=1-u_{3}+C_{1}.
\end{equation*}}
Then%
\begin{equation*}
\mathbf{S}_{1}+\mathbf{S}_{2}+\mathbf{S}_{3}-(\mathbf{T}_{1}+\mathbf{T}_{2}+%
\mathbf{T}_{3})+\mathbf{R=}\frac{M}{(1-u_{1})(1-u_{2})(1-u_{3})ABC},
\end{equation*}%
where after some simple calculations we get
\begin{equation*}
M=ABC-A_{1}B_{1}C_{1}+
\end{equation*}%
\begin{equation*}
+((-\alpha_{1}-\alpha_{2}-\alpha_{3})(\alpha_{1}(1-u_{2})(1-u_{3})A_{1}+\alpha_{2}(1-u_{1})(1-u_{3})B_{1}+
\alpha_{3}(1-u_{1})(1-u_{2})C_{1})+
\end{equation*}%
\begin{equation*}
+(-\alpha_{2}-\alpha_{3}) (1-u_{1})B_{1}C_{1}+(-\alpha_{1}-\alpha_{3}) (1-u_{2})A_{1}C_{1}+(-\alpha_{1}-
\alpha_{2}) (1-u_{3})A_{1}B_{1}.
\end{equation*}%
It is easy to show that%
\begin{equation*}
\alpha_{1}A_{1}+\alpha_{2}B_{1}+\alpha_{3}C_{1}=0.
\end{equation*}%
Somewhat more complicated was to see that%
\begin{equation*}
(\alpha_{1}+\alpha_{2}+\alpha_{3})((\alpha_{1}(u_{2}+u_{3})A_{1}+\alpha_{2}(u_{1}+u_{3})B_{1}+
\alpha_{3}(u_{1}+u_{2})C_{1})
\end{equation*}%
\begin{equation*}
-(\alpha_{2}+\alpha_{3})
B_{1}C_{1}-(\alpha_{1}+\alpha_{3})A_{1}C_{1}-(\alpha_{1}+\alpha_{2})
A_{1}B_{1}=
\end{equation*}%
\begin{equation*}
=\alpha_{1}\alpha_{2}\alpha_{3}(u_{1}-u_{2}+u_{2}-u_{3}+u_{3}-u_{1})^{2}=0.
\end{equation*}%
After that, representing
\begin{equation*}
A_{1}=u_{1}(\alpha {2}+\alpha_{3})-\alpha_{2}u_{2}-\alpha_{3}u_{3}=pu_{1}-\alpha _{2}u_{2}-\alpha _{3}u_{3};
\end{equation*}%
\begin{equation*}
B_{1}=-\alpha_{1}u_{1}+u_{2}(\alpha_{1}+\alpha_{3})-\alpha_{3}u_{3}=-\alpha_{1}u_{1}+qu_{2}-\alpha_{3}u_{3};
\end{equation*}%
\begin{equation*}
C_{1}=-\alpha_{1}u_{1}-\alpha_{2}u_{2}+u_{3}(\alpha_{1}+\alpha_{2})=-\alpha_{1}u_{1}-\alpha_{2}u_{2}+ru_{3}
\end{equation*}
after some calculations we find that
\begin{equation*}
(-\alpha_{1}-\alpha_{2}-\alpha_{3})(\alpha_{1}u_{2}u_{3}A_{1}+\alpha_{2}u_{1}u_{3}B_{1}+\alpha_{3}u_{1}u_{2}C_{1})+
\end{equation*}%
\begin{equation*}
(\alpha_{2}+\alpha_{3}) u_{1}B_{1}C_{1}+(\alpha_{1}+\alpha_{3}) u_{2}A_{1}C_{1}+(\alpha_{1}+\alpha_{2})
u_{3}A_{1}B_{1}=A_{1}B_{1}C_{1},
\end{equation*}%
which implies that
\begin{equation*}
M=ABC.
\end{equation*}
\end{proof}

\begin{remark}
This lemma allows us to assert a further conclusions about our results that
we did not only check hard combinatorial identities, but also demonstrated a
new multiple combinatorial identity is from the theory of functions in
$\mathbb{C}^{n}.$
\end{remark}

\subsection{Proof of Theorems \ref{T1} and \ref{T2}}

\begin{theorem}
\label{T1.1}
The following formula is valid
\begin{equation*}
R+S_{1}(s_{1},s_{2},s_{3};\alpha_{2},\alpha_{3}) +S_{1}(s_{2},s_{1},s_{3};\alpha_{1},\alpha_{3})+
S_{1}(s_{3},s_{1},s_{2};\alpha_{1},\alpha_{2}) -
\end{equation*}%
\begin{equation*}
-(T_{1}(s_{1},s_{2},s_{3};\alpha_{1},\alpha_{2},\alpha_{3})
+ T_{1}(s_{1},s_{3},s_{2};\alpha_{2},\alpha_{3},\alpha_{1})+
  T_{1}(s_{1},s_{3},s_{2};\alpha_{1},\alpha_{3},\alpha_{2}))=1
\end{equation*}
\end{theorem}
\begin{proof}[Proof Therem \ref{T1}]
Using the generating function (\ref{SR}) we have
\begin{equation*}
\left( \sum_{i=1}^{3}S_{i}-\sum_{j=1}^{3}T_{j}+R\right)_{s}(\alpha )=%
\mathbf{res}_{u_{1}u_{2}u_{3}}\frac{u_{1}^{-s_{1}-1}u_{2}^{-s_{2}-1}u_{3}^{-s_{3}-1}}{(1-u_{1})(1-u_{2})(1-u_{3})}=
\end{equation*}%
\begin{equation*}
=\mathbf{res}_{u_{1}u_{2}u_{3}}\left( 1+\sum_{j_{1}=1}^{\infty}u_{1}^{j_{1}}\right)
\left( 1+\sum_{j_{2}=1}^{\infty }u_{2}^{j_{3}}\right)
\left( 1+\sum_{j_{3}=1}^{\infty }u_{3}^{j_{3}}\right)
u_{1}^{-s_{1}-1}u_{2}^{-s_{2}-1}u_{3}^{-s_{3}-1}=1.
\end{equation*}%
\medskip \medskip
\end{proof}
%\subsection{Proof of Theorem} \ref{T2} and}
%\textbf{Proof of Theorem} \ref{T2}

\begin{theorem}
If the complex parameters $z_{1},\ldots ,z_{n}$ satisfy the relation
\begin{equation*}
z_{1}+\ldots +z_{n}=1,
\end{equation*}%
then for any values\textit{\ }$s_{1},\ldots ,s_{n}=0,1,2,\ldots $ the
following identity is valid%
\begin{equation*}
z_{1}^{s_{1}+1}\sum_{j_{2}=0}^{s_{2}}\ldots \sum_{j_{n}=0}^{s_{n}}\binom{s_{1}+\sum_{i\neq 1}j_{i}}
{s_{1},j_{2},\ldots ,j_{n}}z_{2}^{j_{2}}\cdots
z_{n}^{j_{n}}+\ldots +
\end{equation*}
\begin{equation*}
+z_{n}^{s_{n}+1}\sum_{j_{1}=0}^{s_{1}}\ldots \sum_{j_{n-1}=0}^{s_{n-1}}%
\binom{s_{n}+\sum_{i\neq n}j_{i}}{s_{n},j_{1},\ldots ,j_{n-1}}%
z_{1}^{j_{1}}\cdots z_{n-1}^{j_{n-1}}=1.
\end{equation*}
\end{theorem}

\begin{proof}[Proof Therem \ref{T2}]
Denoting the sum on the left side (\ref{K2}) by
$S_{s}(z):=S_{s_1 \ldots s_n}(z_{1},\ldots ,z_{n}),$ we find
in closed form the generating function for the sequence
$\{S_{s}(z) \}_{s_{i}\geq 0}$
\begin{equation}
T(t, z):=\sum_{s_{i}\geq 0}S_{s}(z)t_{1}^{s_{1}}\ldots t_{n}^{s_{n}}.
\label{K3}
\end{equation}%
To do this, using the coefficients we find the beginning of integral
representation for each summand of the left side (\ref{K2}).
For example,%
\begin{equation*}
R_{1}:=z_{1}^{s_{1}+1}\sum_{j_{2}=0}^{s_{2}}\ldots \sum_{j_{n}=0}^{s_{n}}
\binom{s_{1}+\sum_{i\neq 1}j_{i}}{s_{1},j_{2},\ldots ,j_{n}}
z_{2}^{j_{2}}\cdots z_{n}^{j_{n}}=
\end{equation*}%
\begin{equation*}
=z_{1}^{s_{1}+1}\sum_{j_{2}=0}^{s_{2}}\ldots \sum_{j_{n}=0}^{s_{n}}
\binom{s_{1}+\sum_{i\neq 1}(s_{i}-j_{i})}{s_{1},s_{2}-j_{2},\ldots, s_{n}-j_{n}}z_{2}^{s_{2}-j_{2}}\cdots z_{n}^{s_{n}-j_{n}}=
\end{equation*}%
\begin{equation*}
=z_{1}\sum_{j_{2}=0}^{\infty }\ldots \sum_{j_{n}=0}^{\infty }
\mathbf{res}_{x_{1}\ldots x_{n}}\frac{\left( 1-\sum_{i}z_{i}x_{i}\right)^{-1}}
{x_{1}^{s_{1}+1}x_{2}^{s_{2}-j_{2}+1}\ldots x_{n}^{s_{n}-j_{n}+1}}=
\end{equation*}%
(summation over the indices $j_{2},\ldots ,j_{n}$ by the formula of the sum of a
geometric progression)%
\begin{equation}
=z_{1}\mathbf{res}_{x_{1}\ldots x_{n}}\frac{1}{1-\sum_{i}z_{i}x_{i}}%
\prod\nolimits_{i\neq 1}(1-x_{i})^{-1}\times\prod\nolimits_{i}x_{i}^{-s_{i}-1}.
\label{K4}
\end{equation}%
Then%
\begin{equation*}
P_{1}(t) :=\sum_{s_{i}\geq 0}R_{1}t_{1}^{s_{1}}\ldots t_{n}^{s_{n}}=
\end{equation*}%
\begin{equation*}
=z_{1}\sum_{s_{i}\geq 0}\mathbf{res}_{x_{1}\ldots x_{n}}\frac{1}{1-\sum_{i}z_{i}x_{i}}
\prod\nolimits_{i\neq 1}(1-x_{i})^{-1}\times \prod\nolimits_{i}x_{i}^{-s_{i}-1}t_{i}^{s_{i}}=
\end{equation*}%
(summation over the indices $s_{1},\ldots ,s_{n}:$ rule of substitution,
replace $x_{1}=t_1,\ldots,\,x_{n}=t_{n}$)
\begin{equation}
=\frac{z_{1}}{1-\sum_{i}z_{i}t_{i}}\prod\nolimits_{i\neq 1}(1-t_{i})^{-1}.
\label{K5}
\end{equation}%
Thus%
\begin{equation*}
T(t) =\frac{\{z_{1}(1-t_{1}) +\ldots +z_{n}(1-t_{n})\}}
{1-\sum_{i}z_{i}t_{i}}\prod\nolimits_{i}(1-t_{i})^{-1}
\end{equation*}%
\begin{equation*}
=\frac{\sum_{i}z_{i}-\sum_{i}z_{i}t_{i}}{1-\sum_{i}t_{i}z_{i}}\prod\nolimits_{i}(1-t_{i})^{-1}
\end{equation*}%
\begin{equation*}
(considering\text{ }(\ref{K1}))
\end{equation*}%
\begin{equation}
=\frac{1-\sum_{i}z_{i}t_{i}}{1-\sum_{i}z_{i}t_{i}}\prod\nolimits_{i}(1-t_{i})^{-1}=
\prod\nolimits_{i}(1-t_{i})^{-1}.
\label{K6}
\end{equation}%
Thus, in accordance with (\ref{K2}) and (\ref{K6}) we have for all $s_{i}\geq 0,$ $i=1,\ldots ,n,$
\begin{equation*}
S_{s}(z):=\mathbf{res}_{t_{1}\ldots t_{n}}T( t,z)t_{1}^{-s_{1}-1}\ldots t_{n}^{-s_{n}-1}=
\mathbf{res}_{t_{1}\ldots t_{n}}\prod\nolimits_{i}(1-t_{i})^{-1}t_{1}^{-s_{i}-1}=
\end{equation*}%
\begin{equation*}
=\mathbf{res}_{t_{1}\ldots t_{n}}\left( 1+\sum_{j_{1}=1}^{\infty}t_{1}^{j_{1}}\right) \ldots
\left( 1+\sum_{j_{n}=1}^{\infty}t_{n}^{j_{n}}\right) \left( 1+\sum_{j_{1}=1}^{\infty }t_{1}^{j_{1}}\right)
\prod\nolimits_{i}t_{1}^{-s_{i}-1}=1.\cdot
\end{equation*}
\end{proof}
\begin{remark}
%\noindent
The Theorem \ref{T1} is a special case of the Theorem \ref{T2} for $n=3.$%
\end{remark}
\bigskip

\subsection{The new short calculation of Zeilberger and Krivokolesko combinatorial sums and its some applications}

\noindent

In \cite{th} a theory of integral representations of holomorphic functions
in a linearly convex domains $D\subset \mathbb{C}^{n}$ with a piecewise
regular boundary, which allowed V.P. Krivokolesko in \cite{th1} to find a
series of combinatorial identities for a certain family of integral
parameters of $D.$ For example,%
\begin{equation*}
\alpha _{1}^{s_{1}+1}\sum_{k=0}^{s_{2}}\sum_{l=0}^{s_{3}}\frac{(s_{1}+k+l)!}{%
s_{1}!k!l!}\alpha _{2}^{k}\alpha _{3}^{l}+\alpha
_{2}^{s_{2}+1}\sum_{k=0}^{s_{1}}\sum_{l=0}^{s_{3}}\frac{(s_{2}+k+l)!}{%
s_{2}!k!l!}\alpha _{1}^{k}\alpha _{3}^{l}+
\end{equation*}%
\begin{equation}
+\alpha _{3}^{s_{3}+1}\sum_{k=0}^{s_{1}}\sum_{l=0}^{s_{2}}\frac{(s_{3}+k+l)!%
}{s_{3}!k!l!}\alpha _{1}^{k}\alpha _{2}^{l}=1,
\label{A1}
\end{equation}%
where numeric parameters $\alpha _{1},\alpha _{2},\,\alpha _{3}$ satisfy
equality $\alpha _{1}+\alpha _{2}+\alpha _{3}=1.$ V.P. Krivokolesko raised
the question of finding a simple proof of the following identity (\ref{A2}),
which generalizes the identity (\ref{A1}):
\begin{equation*}
z_{1}^{s_{1}+1}\sum_{j_{2}=0}^{s_{2}}\ldots \sum_{j_{n}=0}^{s_{n}}\binom{%
s_{1}+\sum_{i\neq 1}j_{i}}{s_{1},j_{2},\ldots ,j_{n}}z_{2}^{j_{2}}\cdots
z_{n}^{j_{n}}+\ldots
\end{equation*}%
\begin{equation}
+z_{n}^{s_{n}+1}\sum_{j_{1}=0}^{s_{1}}\ldots \sum_{j_{n-1}=0}^{s_{n-1}}%
\binom{s_{n}+\sum_{i\neq n}j_{i}}{s_{n},j_{1},\ldots ,j_{n-1}}%
z_{1}^{j_{1}}\cdots z_{n-1}^{j_{n-1}}=1,
\label{A2}
\end{equation}%
for all values {}{}of the parameters\textit{\ }$s_{1},\ldots
,s_{n}=0,1,2,\ldots ,$ where the numerical parameters $z_{1},\ldots ,z_{n}$
satisfy the equation%
\begin{equation}
z_{1}+\ldots +z_{n}=1.
 \label{A3}
\end{equation}%
Particular cases (\ref{A2}) is the identity of V.L. Shelkovich in quantum
field theory \cite{Shell82} and probability of identity in the theory of D.
Zeilberger dilations (the wavelet theory) \cite{Zeil83}.

Here in Theorem 1 using the coefficients we found elegant formula integral
representation (generating function) for the sum of (\ref{A4}), which is a
natural generalization of the sum in the left-hand side of (\ref{A2}). It is
possible to find a short analytic proof of the desired identity (\ref{A2})
(Lemma \ref{le2}). Moreover found some interesting recurrence relations for the
calculation of multiple combinatorial sums of different types (Lemma \ref{le3} and
\ref{le4}).

\begin{theorem}
\label{theo1}
Let
\begin{equation*}
S_{s}(z;\alpha):=z_{1}^{s_{1}+1}\sum_{j_{2}=0}^{s_{2}}\ldots
\sum_{j_{n}=0}^{s_{n}}\binom{\alpha +s_{1}+\sum_{i\neq 1}j_{i}}
{\alpha,s_{1},j_{2},\ldots ,j_{n}}z_{2}^{j_{2}}\cdots z_{n}^{j_{n}}+\ldots +
\end{equation*}%
\begin{equation}
+z_{n}^{s_{n}+1}\sum_{j_{1}=0}^{s_{1}}\ldots \sum_{j_{n-1}=0}^{s_{n-1}}%
\binom{\alpha +s_{n}+\sum_{i\neq n}j_{i}}
{\alpha ,s_{n},j_{1},\ldots ,j_{n-1}}z_{1}^{j_{1}}\cdots z_{n-1}^{j_{n-1}},
\label{A4}
\end{equation}%
where\, the\, real\, parameter $\alpha \geq 0$ and\, the\, complex\, parameters\,
$z_{1},\ldots ,z_{n}$\, satisfy\, the\, condition.
(\ref{A3})
Then the generating function
\begin{equation}
T_{\alpha }(t) :=\sum_{s_{i}\geq 0}S_{s}(z;\alpha)t_{1}^{s_{1}}\ldots
t_{n}^{s_{n}}
\label{A5}
\end{equation}
for the sequence $\{S_{s}(z;\alpha)\}_{s_{i}\geq 0}$ has the form
\begin{equation}
T_{\alpha }(t)
=(1-\sum_{i}z_{i}t_{i})^{-\alpha}\prod\limits_{i}(1-t_{i})^{-1}.
\label{A6}
\end{equation}
\end{theorem}

\begin{proof}
Following the scheme of computation method of integral first find an
integral representation for each summand $R_{k}(z),k=1,...,n$ в
(\ref{A4}). For example,
\begin{equation*}
R_{1}(z):=z_{1}^{s_{1}+1}\sum_{j_{2}=0}^{s_{2}}\ldots \sum_{j_{n}=0}^{s_{n}}
\binom{\alpha +s_{1}+\sum_{i\neq 1}j_{i}}{\alpha ,s_{1},j_{2},\ldots ,j_{n}}
z_{2}^{j_{2}}\cdots z_{n}^{j_{n}}=
\end{equation*}%
\begin{equation*}
=z_{1}^{s_{1}+1}\sum_{j_{2}=0}^{s_{2}}\ldots \sum_{j_{n}=0}^{s_{n}}
\binom{\alpha +s_{1}+\sum_{i\neq 1}(s_{i}-j_{i})}
{\alpha,s_{1},s_{2}-j_{2},\ldots ,s_{n}-j_{n}}z_{2}^{s_{2}-j_{2}}\cdots
z_{n}^{s_{n}-j_{n}}=
\end{equation*}%
\begin{equation*}
=z_{1}\sum_{j_{2}=0}^{s_{2}}\ldots \sum_{j_{n}=0}^{s_{n}}
\mbox{\bf res}_{x}\frac{(1-\sum_{i}z_{i}x_{i})^{-\alpha -1}}
{x_{1}^{s_{1}+1}x_{2}^{s_{2}-j_{2}+1}\ldots x_{n}^{s_{n}-j_{n}+1}}=
\end{equation*}%
\begin{equation*}
(the\,rule\,linearly\,\mbox{\bf res}_{x}:\mathbf{\ }\,entry\, under\, the \,sign\, of\, the\,
sum\, sign\, \mbox{\bf res}_{x})
\end{equation*}%
\begin{equation*}
=z_{1}\mbox{\bf res}_{x}\{\sum_{j_{2}=0}^{\infty }\ldots
\sum_{j_{n}=0}^{\infty }\frac{(1-\sum_{i}z_{i}x_{i})^{-\alpha -1}}
{x_{1}^{s_{1}+1}x_{2}^{s_{2}-j_{2}+1}\ldots x_{n}^{s_{n}-j_{n}+1}}\}=
\end{equation*}%
\begin{equation*}
(summation\, over\, indices\,j_{2},\ldots ,j_{n}\,formula\, of\, the\, sum\,of\, a \,geometric\, progression)
\end{equation*}%
\begin{equation*}
=z_{1}\mbox{\bf res}_{x}(1-\sum_{i}z_{i}x_{i})^{-\alpha
-1}\prod\limits_{i\neq 1}(1-x_{i})^{-1}\prod\limits_{i}x_{i}^{-s_{i}-1}.
\end{equation*}
Thence%
\begin{equation*}
P_{1}(t,z):=\sum_{s_{i}\geq 0}R_{1}(z)t_{1}^{s_{1}}\ldots t_{n}^{s_{n}}=
\end{equation*}%
\begin{equation*}
=z_{1}\sum_{s_{i}\geq 0}\mbox{\bf res}_{x}(1-\sum_{i}z_{i}x_{i})^{-\alpha
-1}\prod\nolimits_{i\neq
1}(1-x_{i})^{-1}\prod\nolimits_{i}x_{i}^{-s_{i}-1}t_{i}^{s_{i}}=
\end{equation*}%
\begin{equation*}
(summation\, over\, indices\, s_{1},\ldots ,s_{n}:
\end{equation*}%
\begin{equation*}
rule\,of\, substitution\, for\,the\,operator\,
\mbox{\bf res}\,change\, x_{i}=t_{i},\,i=1,...,n)
\end{equation*}%
\begin{equation*}
=z_{1}(1-\sum_{i}z_{i}t_{i})^{-\alpha -1}\prod\nolimits_{i\neq 1}(1-t_{i})^{-1}
=\frac{z_{1}(1-t_{1})}{1-\sum_{i}z_{i}t_{i}}%
\prod\nolimits_{i}(1-t_{i})^{-1}.
\end{equation*}%
Thus,%
\begin{equation*}
T(t,z):=\sum_{i}P_{i}(t,z)=\frac{z_{1}(1-t_{1})+\ldots+z_{n}(1-t_{n})}
{(1-\sum_{i}z_{i}t_{i})^{\alpha +1}}\prod\nolimits_{i}(1-t_{i})^{-1}=
\end{equation*}%
\begin{equation}
=\frac{\sum_{i}z_{i}-\sum_{i}z_{i}t_{i}}{(1-\sum_{i}z_{i}t_{i})^{\alpha +1}}
\prod\nolimits_{i}(1-t_{i})^{-1},
\label{A6a}
\end{equation}%
and subject to (\ref{A3}) we obtain the required formula (\ref{A6}):
\begin{equation*}
T_{\alpha }(t):=[T(t,z)]_{z_{1}+\ldots z_{n}=1}=\frac{1-\sum_{i}z_{i}t_{i}}
{(1-\sum_{i}z_{i}t_{i})^{\alpha +1}}\prod\nolimits_{i}(1-t_{i})^{-1}=
\end{equation*}%
\begin{equation*}
=(1-\sum_{i}z_{i}x_{i})^{-\alpha }\prod\nolimits_{i}(1-t_{i})^{-1}.
\end{equation*}
\end{proof}

\begin{lemma}
\label{le2}
The identity (\ref{A2}) is valid.
\end{lemma}

\begin{proof}
If we denote the left-hand side of the identity(\ref{A2}) after$S_{s}(z),$
then $S_{s}(z):=S_{s}(z;0)$, and according to$(\ref{A6}) $ we have for all $
s_{i}\geq 0,\,i=1,\ldots ,n,$
\begin{equation*}
S_{s}(z):=\mbox{\bf res}_{t}T_{0}(t)t_{1}^{-s_{1}-1}\ldots t_{n}^{-s_{n}-1}=%
\mbox{\bf res}_{t} \prod\nolimits_{i}(1-t_{i})^{-1}t_{1}^{-s_{i}-1}=
\end{equation*}%
\begin{equation*}
=\mbox{\bf res}_{t}(1+\sum_{j_{1}=1}^{\infty}t_{1}^{j_{1}})\cdot
\ldots\cdot(1+\sum_{j_{n}=1}^{\infty}t_{n}^{j_{n}})
\prod\nolimits_{i}t_{1}^{-s_{i}-1}=1.
\end{equation*}
\end{proof}

\begin{lemma}
\label{le3}
The following identity is valid for $S_{s}(z;t):$
\begin{equation*}
S_{s}(z;\alpha +1)-z_{1}S_{s_{1}-1,s_{2},\ldots ,s_{n}}(z;\alpha+1)-
\ldots -z_{n}S_{s_{1},s_{2},\ldots ,s_{n}-1}(z;\alpha +1)=
\end{equation*}
\begin{equation}
=S_{s}(z;\alpha) ,\,\alpha=0,1,2,\ldots .
\label{A7}
\end{equation}
In particular, (\ref{A7}) for $\alpha =0$ by \ref{le2}
and relation $S_{s}(z;0):=S_{s}(z) =1$ we have the following recursive
formula\textit{:}%
\begin{equation}
S_{s}(z;1)=1+z_{1}S_{s_{1}-1,s_{2},\ldots ,s_{n}}(z;1) +\ldots
+z_{n}S_{s_{1},s_{2},\ldots ,s_{n}-1}(z;1).
\label{A8}
\end{equation}
\end{lemma}

\begin{proof}
According to $(\ref{A5})$ and $(\ref{A6})$ for any $s_{i}=0,1.2,\ldots $ we
have
\begin{equation*}
S_{s}(z;\alpha ):=\mbox{\bf res}_{t}T_{\alpha }(t)\left(\prod\nolimits_{i}t_{i}^{-s_{i}-1}\right)=
\end{equation*}%
\begin{equation}
=\mbox{\bf res}_{t_{1},\ldots ,t_{n}}(1-\sum_{i}z_{i}t_{i})^{-\alpha}
\prod\nolimits_{i}(1-t_{i})^{-1}t_{i}^{-s_{i}-1}=
\label{A9}
\end{equation}
\begin{equation*}
=\mbox{\bf res}_{t_{1},\ldots ,t_{n}}\{(1-\sum_{i}z_{i}t_{i})^{-\alpha-1}
\prod\nolimits_{i}(1-t_{i})^{-1}\}
(1-\sum_{i}z_{i}t_{i})(\prod\nolimits_{i}t_{i}^{-s_{i}-1})=
\end{equation*}
\begin{equation*}
=\mbox{\bf res}_{t_{1},\ldots ,t_{n}}T_{\alpha +1}(t)
(1-\sum_{i}z_{i}t_{i})(\prod\nolimits_{i}t_{i}^{-s_{i}-1})=
\end{equation*}%
\begin{equation*}
(by\,\, linearity\,\,of\,\,the\,\,operator\,\,\mbox{\bf res}_{t})
\end{equation*}
\begin{equation*}
=\mbox{\bf res}_{t_{1},\ldots ,t_{n}}T_{\alpha+1}(t)
\prod\nolimits_{i}t_{i}^{-s_{i}-1}-z_{1}\mbox{\bf res}_{t_{1},\ldots,t_{n}}
T_{\alpha +1}(t)t_{1}^{-(s_{1}-1)-1}
(\prod\nolimits_{i\neq 1}t_{i}^{-s_{i}-1})-\ldots -
\end{equation*}%
\begin{equation*}
-z_{n}\mbox{\bf res}_{t_{1},\ldots ,t_{n}}T_{\alpha+1}(t)t_{n}^{-(s_{n}-1)-1}
(\prod\nolimits_{i\neq n}t_{i}^{-s_{i}-1})=
\end{equation*}%
\begin{equation*}
(according\text{ }to\text{ }the\text{ }definition\text{ }(\ref{A9}))
\end{equation*}%
\begin{equation*}
=S_{s}(z;\alpha +1)-z_{1}S_{s_{1}-1,s_{2},\ldots ,s_{n}}(z;\alpha +1)-\ldots
-z_{1}S_{s_{1},s_{2},\ldots ,s_{n}-1}(z;\alpha +1).
\end{equation*}%
\smallskip
\end{proof}

\begin{lemma}
\label{le4}
\textit{Let }$S_{s}(z;\alpha ,\beta) $\textit{\ be the numbers
determined by the formula (\ref{A4}), where the parameter }
$\alpha \geq 0$%
\textit{, and the complex parameters }$z_{1},\ldots ,z_{n},\beta $\textit{\
satisfy the condition}%
\begin{equation}
z_{1}+\ldots +z_{n}=\beta .
\label{A10a}
\end{equation}%
\textit{Obviously,}%
\begin{equation}
S_{s}(z;\alpha ,1):=S_{s}(z;\alpha) ,\text{ }%
S_{s}(z;0,1):=S_{s}(z),
\label{A11}
\end{equation}%
\begin{equation}
S_{s}(z_{1},\ldots z_{n};\alpha ,\beta)=S_{s}
(z_{1}/\beta,\ldots z_{n}/\beta ;\alpha ),\text{ }if\text{ }\beta \neq 0.
\label{A12}
\end{equation}%
\textit{Then the generating function }%
\begin{equation}
T_{\alpha ,\beta }(t):=\sum_{s_{i}\geq 0}S_{s}
(z;\alpha,\beta) t_{1}^{s_{1}}\ldots t_{n}^{s_{n}}
\label{A13}
\end{equation}%
\textit{for the sequence }$\{S_{s}(z;\alpha ,\beta)\}_{s_{i}\geq 0}$\textit{\ has the form}%
\begin{equation}
T_{\alpha ,\beta }(t) =(\beta-\sum_{i}z_{i}t_{i})
(1-\sum_{i}z_{i}t_{i})^{-\alpha-1}\prod\nolimits_{i}(1-t_{i})^{-1},
\label{A14}
\end{equation}%
\textit{and the following formula for the number of }$S_{s}(z;\alpha,\beta) :$
\begin{equation}
S_{s}(z;\alpha ,\beta)=(\beta -1)S_{s}(z;\alpha +1,1)
+S_{s}(z;\alpha ,1) ,\text{ }\alpha=1,2,\ldots .
\label{A15}
\end{equation}
\textit{In other words, if }$\beta \notin \{0,1\},$\textit{\ then in view of
(\ref{A11}) and (\ref{A12}) the following recursive formula for the number of}
$S_{s}(z;\alpha)$\textit{\ is valid:}
\begin{equation}
S_{s}(\beta z;\alpha +1)=\frac{1}{\beta -1}\left(S_{s}(z;\alpha)-S_{s}(\beta z;\alpha)\right),
\text{ }\alpha =1,2,\ldots .
\label{A16}
\end{equation}%
\smallskip
\end{lemma}

\begin{proof}
The formula (\ref{A14}) immediately from (\ref{A6a}) using the equation (\ref{A10a}).
 Thus, we have
\begin{equation*}
T_{\alpha ,\beta }(t)=(\beta-\sum_{i}z_{i}t_{i})
(1-\sum_{i}z_{i}t_{i})^{-\alpha-1}\prod\nolimits_{i}(1-t_{i})^{-1}=
\end{equation*}%
\begin{equation*}
=(\beta -1)(1-\sum_{i}z_{i}t_{i})^{-\alpha-1}
\prod\nolimits_{i}(1-t_{i})^{-1}+(1-\sum_{i}z_{i}t_{i})^{-\alpha}
\prod\nolimits_{i}(1-t_{i})^{-1}:=
\end{equation*}
(by formula (\ref{A6}))
\begin{equation}
=(\beta -1)T_{\alpha +1}(t)+T_{\alpha }(t).
\label{A17}
\end{equation}%
Equating the coefficients of the monomials $t_{1}^{s_{1}}\ldots
t_{n}^{s_{n}} $ in the right and left sides of (\ref{A17}), we obtain
(\ref{A15}).
\end{proof}

%\section{ 5. Calculation of multiple combinatorial sums in the theory of
%golomorphic functions in $\mathbb{C}^{n} $\protect\medskip}
%\section{\textit{6. Integral representation and computation a multiple sum
%in the theory of cubature formulas}\protect\bigskip }
\section{Integral representation and computation of a multiple sum
in the theory of cubature formulas}

Let $\alpha =(\alpha_{0},\alpha_{1},\ldots ,\alpha_{d}),$
$\beta=(\beta_{0},\beta_{1},\ldots ,\beta_{d})$ be vectors from
$E^{d+1}$ with integer non-negative coordinats, and the vector
$\gamma=(\gamma_{0},\gamma_{1},\ldots ,\gamma_{d})\in E^{d+1}.$
Denote
\begin{equation*}
|\mathbf{\alpha }|:=\alpha_{0}+\alpha_{1}+\ldots +\alpha_{d}=2s+1,\,
\mathbf{\alpha }!:=\alpha_{0}!\alpha_{1}!\ldots \alpha_{d}!,\,
\binom{\mathbf{\alpha }}{\mathbf{\beta }}:=
\binom{\alpha_{0}}{\beta_{0}}\ldots \binom{\alpha_{d}}{\beta_{d}},
\end{equation*}
 where
\begin{equation*}
\binom{a}{b}:=\frac{\Gamma (a+1)}{\Gamma (b+1)\,\Gamma (a-b+1)},
\text{and}\,\binom{a}{b}:=0,\,\text {if}\, b\geq a+1.
\end{equation*}

Moreover, we write $\mathbf{\alpha }-1/2:=(\alpha _{0}-1/2,\alpha_{1}-1/2,\ldots ,\alpha_{d}-1/2).$

Heo S. and Xu Y. (\cite{XeoXu(2000)}, the identity (2.9), p.~631--635) with
the help of the method of generating functions and the
difference-differential operators of various type, found not simple
proof the following multiple combinatorial identity (\cite{XeoXu(2000)}, identity (2.9)):
%{XeoXu(2000)}
\begin{equation*}
2^{2s}\mathbf{\alpha}!\binom{\alpha +\gamma}{\alpha}=
\sum_{j=0}^{s}(-1)^{j}\binom{d+\sum_{i=0}^{d}(\alpha_{i}+\gamma_{i})}{j}\times
\end{equation*}
\begin{equation}
\times\sum_{\beta_{0}+\beta_{1}+ldots+\beta_{d}=s-j}
\prod_{i=0}^{d}\binom{\beta_{i}+\gamma_{i}}{\beta_{i}}(2\beta_{i}+\gamma_{i}+1)^{\alpha_{i}}.
\label{KK1}
\end{equation}

The purpose of the given section is finding a new simple proof of identity
(\ref{KK1}) by means of the method of coefficients and multiple applying of
known theorem on the total sum of residues.

\subsection {The proof of the identity (\ref{KK1})}
%\paragraph{\textbf{The proof of the identity\ }(\protect\ref{KK1})}

It is possible to copy the identity (\ref{KK1})  in the form of
\begin{equation*}
\sum_{j=0}^{s}(-1)^{j}\binom{d+\sum_{i=0}^{d}(\alpha_{i}+\gamma_{i})}{j}
\times
\end{equation*}%
\begin{equation}
\!\times\sum_{\beta_{0}+\beta_{1}+...+\beta_{d}=s-j}\prod_{i=0}^{d}\binom{\beta_{i}+\gamma_{i}}{\beta_{i}}
\frac{(2\beta_{i}+\gamma_{i}+1)^{\alpha _{i}}}{(\alpha_{i})!}
\!=2^{2s}\prod_{i=0}^{d}\binom{\alpha_{i}+\gamma_{i}}{\alpha_{i}}.
\label{KK2}
\end{equation}%
Let's enter use following designations for the right part of identity (\ref{KK2}):
\begin{equation}
T(s; \alpha,\beta):=\sum_{j=0}^{s}(-1)^{j}\binom{\sum_{i=0}^{d}
(\alpha_{i}+\gamma_{i})+d}{j}\cdot S_{j},
\label{KK3}
\end{equation}%
where%
\begin{equation}
S_{j}:=\sum_{\beta_{0}+\beta_{1}+...+\beta_{d}=s-j}\,\prod_{i=0}^{d}
\binom{\beta_{i}+\mu_{i}}{\beta_{i}}\frac{(2\beta_{i}+\mu_{i}+1)^{\alpha_{i}}}
{(\alpha_{i})!}.
\label{KK4}
\end{equation}%
Then by means of the method of coefficients we receive%
\begin{equation*}
S_{j}=\sum_{|\beta|=s-j}\,\prod_{i=0}^{d}\binom{\beta_{i}+\gamma_{i}}
{\beta_{i}}\frac{(2\beta_{i}+\gamma_{i}+1)^{\alpha_{i}}}{(\alpha_{i})!}=
\end{equation*}%
\begin{equation*}
=\sum_{\beta_{0}=0}^{\infty}\ldots \sum_{\beta_{d}=0}^{\infty}
\mbox{\bf res}_{z_{0},\ldots,z_{d},t}(t^{-s+j-1}\prod_{i=0}^{d}
(1-tz_{i})^{-\gamma_{i}-1}z_{i}^{-\beta _{i}-1})\times
\end{equation*}%
\begin{equation*}
\times \mbox{\bf res}_{w_{0},\ldots ,w_{d}}\left(\prod_{i=0}^{d}
w_{i}^{-\alpha_{i}-1}\exp (w_{i}(2\beta_{i}+\gamma_{i}+1)\right)=
\end{equation*}%
\begin{equation*}
=\mbox{\bf res}_{w_{0},\ldots,w_{d},t}\{t^{-s+j-1}
(\prod_{i=0}^{d}w_{i}^{-\alpha_{i}-1}\exp (w_{i}(\gamma_{i}+1))\times
\end{equation*}%
\begin{equation*}
\times \prod_{i=0}^{d}\left(\sum_{\beta_{i}=0}^{\infty}\exp (\beta_{i}(2w_{i}))
\mbox{\bf res}_{z_{i}}\left(1-tz_{i})^{-\gamma_{i}-1}z_{i}^{-\beta_{i}-1}\right)\right)\}=
\end{equation*}%
\begin{equation*}
(\text{the\, summation\,by\,each}\,\beta_{i},\text{and}\,\mbox{\bf res}_{z_i}\,i=0,\ldots,\,d:
\end{equation*}%
\begin{equation*}
\text{the\,\, substitution\, of\, the\, rule\,of\,\,changes}\,\, z_{i}=\exp (2w_{i}),\,i=0,\ldots ,d)
\end{equation*}%
\begin{equation*}
=\mbox{\bf res}_{w_{0},\ldots
,w_{d},t}\{t^{-s+j-1}(\prod_{i=0}^{d}w_{i}^{-\alpha_{i}-1}\exp (w_{i}(\gamma _{i}+1))
\cdot\prod_{i=0}^{d}(1-t\exp (2w_{i}))^{-\gamma _{i}-1}\}=
\end{equation*}%
\begin{equation*}
=\mbox{\bf res}_{w_{0},\ldots, w_{d},t}\{t^{-s+j-1}
\prod_{i=0}^{d}w_{i}^{-\alpha_{i}-1}\left(\exp(-w_{i})-t\exp (w_{i})\right)^{-\gamma_{i}-1}\},
\end{equation*}%
i.e.%
\begin{equation}
S_{j}=\mbox{\bf res}_{w_{0},\ldots ,w_{d},t}\{t^{-s+j-1}
\prod_{i=0}^{d}w_{i}^{-\alpha_{i}-1}(\exp(-w_{i})-t\exp (w_{i}))^{-\gamma _{i}-1}\}.
\label{KK5}
\end{equation}%
According to (\ref{KK3})--(\ref{KK5}) we received%
\begin{equation*}
T(s;\alpha,\beta) =\sum_{j=0}^{s}\mbox{\bf res}_{w_{0},\ldots ,w_{d},t}
\{t^{-s+j-1}\prod_{i=0}^{d}w_{i}^{-\alpha_{i}-1}(\exp (-w_{i})-t\exp (w_{i}))^{-\gamma _{i}-1}\}
\times
\end{equation*}%
\begin{equation*}
\times \mbox{\bf res}_{x}\{x^{-j-1}(1-x)^{d+\sum_{i=0}^{d}(\alpha_{i}+\gamma_{i})}\}=
\end{equation*}%
\begin{equation*}
=\mbox{\bf res}_{w_{0},\ldots, w_{d},t}\{t^{-s-1}\prod_{i=0}^{d}
w_{i}^{-\alpha_{i}-1}(\exp(-w_{i})-t\exp (w_{i}))^{-\gamma_{i}-1}\times
\end{equation*}%
\begin{equation*}
\times (\sum_{j=0}^{\infty }t^{j}\mbox{\bf res}_{x}\{x^{-j-1}
(1-x)^{d+\sum_{i=0}^{d}(\alpha_{i}+\gamma_{i})})\}=
\end{equation*}%
\begin{equation*}
(the\text{ }summation\text{ }by\text{ }j,\text{ }and\text{ }\mbox{\bf res}_{x}:
the\text{ }substitution\text{ }rule,\text{ }the\text{ }change\text{ }x=t)
\end{equation*}%
\begin{equation*}
=\mbox{\bf res}_{w_{0},\ldots ,w_{d},t}\{t^{-s-1}\prod_{i=0}^{d}w_{i}^{-\alpha _{i}-1}
(\exp(-w_{i})-t\exp (w_{i}))^{-\gamma_{i}-1}(1-t)^{d+\sum_{i=0}^{d}(\alpha _{i}+\gamma _{i})}\}.
\end{equation*}%
Thus we proved

\begin{lemma}
\label{AprL1}
Let parameters $s,\alpha_{0},\alpha_{1},\ldots ,\alpha_{d},\beta_{0},\beta_{1},\ldots ,\beta_{d}$ be the integer non-negative
numbers, for which $\alpha_{0}+\ldots +\alpha_{d}=2s+1,$ \textit{and the }%
vector $(\mu_{0},\mu_{1},\ldots ,\mu_{d})\in \mathbb{R}^{d+1}.$ Then the following integral formula is valid:
\begin{equation*}
\sum_{j=0}^{s}(-1)^{j}\binom{d+\sum_{i=0}^{d}(\alpha_{i}+\gamma _{i})}
{j}\sum_{\beta_{0}+\beta_{1}\ldots +\beta_{d}=s-j}\prod_{i=0}^{d}
\binom{\beta_{i}+\gamma_{i}}{\beta _{i}}\frac{(2\beta_{i}+\gamma_{i}+1)^{\alpha_{i}}}
{(\alpha_{i})!}=
\end{equation*}%
\begin{equation*}
=\mbox{\bf res}_{w_{0},\ldots,w_{d},t}\left\{t^{-s-1}\prod_{i=0}^{d}w_{i}^{-\alpha_{i}-1}
(\exp(-w_{i})-t\exp (w_{i}))^{-\gamma _{i}-1}\right.\times
\end{equation*}%
\begin{equation}
\left.\times (1-t)^{d+\sum_{i=0}^{d}
(\alpha_{i}+\gamma_{i})}\right\}.
\label{KK6}
\end{equation}
\end{lemma}

\begin{remark}. It is easy to see, that a calculation of multiple integral in the right
part (\ref{KK6}) on variables $t,\,w_{0}...,w_{d}$
sequentially gives the multiple sum of the left part (\ref{KK6}).

We will spend a new proof of identity (\ref{KK2})
similarly by calculation of multiple residue of a zero point in the right part
of the formula (\ref{KK6}) it is on each variable $w_{0},\ldots,w_{d}$ and $t$
sequentially (see lemmas \ref{AprL1}--\ref{AprL3} and the theorem \ref{AprT1}).
\end{remark}

Let's enter necessary designations. Denote
\begin{equation}
f=f( w,\,t):=e^{-w}-te^{w},\,g=g(w,\,t)
:=e^{-w}+te^{w},
\label{KK7}
\end{equation}%
where $\alpha $ is the fixed integer and $\gamma \in \mathbb{R}.$ Obviously
\begin{equation}
f':=\frac{df}{dw}=-g,\,g':=\frac{dg}{dw}=-f,\,
g^{2}-f^{2}=4t,\,(f{-\gamma })'=\gamma f{-\gamma -1}g,\,
 \label{KK11}
\end{equation}%
\begin{equation}
(g^{\alpha })'=-\alpha g^{\alpha -1}f,\,f(0) =1-t,\, g(0)=1+t.
 \label{KK12}
\end{equation}
%\qquad \qquad

\begin{lemma}
\label{AprL2}
If $s$ is the integer non-negative number and $\gamma \in \mathbb{R},$
in designations (\ref{KK7}) and (\ref{KK11}) the following expansion is a derivative%
\begin{equation}
(f^{-\gamma -1})^{(\alpha) }_{w}=(\gamma +1)\cdot\ldots \cdot(\gamma +\alpha) f^{-\gamma -\alpha -1}g^{\alpha}+
\sum_{k=1}^{[\alpha/2]}c_{k}(\gamma) f^{-\gamma +2k-\alpha-1}g^{\alpha -2k},
\label{KK13}
\end{equation}%
\textit{with integer coefficients }$c_{1},c_{2},\ldots ,c_{[\alpha /2]}$
\textit{is valid.}

According (\ref{KK12})\ the formula (\ref{KK13})\ generates
the following formula%
\begin{equation*}
\mbox{\bf res}_{w}w_{i}^{-\alpha -1}(\exp (-w_{i})-t\exp(w_{i}))^{-\gamma -1}:=
[\left(\exp (-w_{i})-t\exp (w_{i})\right)^{-\gamma -1}]_{w=0}^{(\alpha) }/\alpha !=
\end{equation*}%
\begin{equation}
=\binom{\alpha +\gamma }{\alpha }(1-t)^{-\gamma -\alpha-1}
(1+t)^{\alpha }(1+\sum_{k=1}^{[\alpha /2]}h_{k}(\alpha,\gamma) (1-t)^{2k}(1+t)^{-2k}),
\label{KK14}
\end{equation}%
where the rational coefficients $h_{k}(\alpha ,\gamma) :=c_{k}(\gamma) /\alpha !,$ $k=1,\ldots ,[\alpha /2].$
\end{lemma}

\begin{proof}
The formula (\ref{KK13}) most easier to prove an induction
on parameter $\alpha $. Really with the help (\ref{KK11})
we have for initial values $\alpha =1,2,3$:%
\begin{equation*}
(f^{-\gamma -1})'=-(\gamma+1) f^{-\gamma -2}f'=(\gamma +1) f^{-\gamma -2}g.
\end{equation*}%
\begin{equation*}
(f^{-\gamma -1})^{\prime \prime }=((f^{-\gamma -1})')'=
((\gamma +1) f^{-\gamma -2}g)'=(\gamma+1) (f^{-\gamma -2})'g+(\gamma +1)f^{-\gamma -2}f=
\end{equation*}%
\begin{equation*}
=(\gamma +1) (\gamma +2) f^{-\gamma -3}g^{2}-(\gamma +1) f^{-\gamma -1}.
\end{equation*}%
\begin{equation*}
(f^{-\gamma -1})^{\prime \prime \prime}=((f^{-\gamma -1})^{^{\prime\prime }})'=
((\gamma +1)( \gamma +2)f^{-\gamma -3}g^{2}-(\gamma +1) f^{-\gamma -1})'=
\end{equation*}%
\begin{equation*}
=(\gamma +1)(\gamma +2)(\gamma +3)f^{-\gamma -4}g^{3}+
\end{equation*}%
\begin{equation*}
+(\gamma +1)(\gamma +2)f^{-\gamma -3}2gf-(\gamma +1)^{2}f^{-\gamma -2}g=
\end{equation*}%
\begin{equation*}
=(\gamma+1)(\gamma +2)(\gamma +3)f^{-\gamma -4}g^{3}-(\gamma +1)(3\gamma +2)f^{-\gamma -2}g.
\end{equation*}%
Further on an induction if the formula (\ref{KK13}) is
valid for current value $\alpha $, with the help (\ref{KK11}) we have%
\begin{equation*}
(f^{-\gamma -1})^{(\alpha +1)}=((f^{-\gamma -1})^{(\alpha) })'=((\gamma +1)\cdot \ldots \cdot
(\gamma +\alpha) f^{-\gamma -\alpha -1}g^{\alpha}+
\end{equation*}%
\begin{equation*}
+\sum_{k=1}^{[\alpha /2]}c_{k}(\gamma)(f^{-\gamma +2k-\alpha-1}g^{\alpha -2k})'=
\end{equation*}%
\begin{equation*}
=(\gamma +1) \times \ldots \times (\gamma +\alpha+1)
 f^{-\gamma -\alpha -2}g^{\alpha +1}-(\gamma +1) \times
\ldots \times (\gamma +\alpha) f^{-\gamma -\alpha }\alpha g^{\alpha -1}+
\end{equation*}%
\begin{equation*}
+\sum_{k=1}^{[\alpha /2]}c_{k}(\gamma) ( \gamma -2k+\alpha+1)f^{-\gamma +2k-\alpha -2}
g^{\alpha -2k+1}-\sum_{k=1}^{[\alpha/2]}c_{k}(\gamma) f^{-\gamma +2k-\alpha }(\alpha-2k)
g^{\alpha -2k-1}=
\end{equation*}%
\begin{equation*}
(the\text{ }replacement\text{ }in\text{ }the\text{ }first\text{ }sum\text{ }%
of\text{ }an\text{ }index\text{ }k-1\text{ }on\text{ }k)
\end{equation*}%
\begin{equation*}
=(\gamma +1) \ldots (\gamma +\alpha +1) f^{-\gamma-\alpha -2}g^{\alpha +1}+(c_{1}(\gamma)
(\gamma +\alpha-1) -\alpha ( \gamma +1) \ldots (\gamma +\alpha) )f^{-\gamma -\alpha +1}g^{\alpha -1}+
\end{equation*}%
\begin{equation*}
+\sum_{k=0}^{[\alpha /2]-1}c_{k+1}(\gamma ) ( \gamma-2k+\alpha -1)
f^{-\gamma+2k-\alpha }g^{\alpha-2k-1}-\sum_{k=1}^{[\alpha /2]}c_{k}(\gamma)
f^{-\gamma+2k-\alpha }(\alpha -2k) g^{\alpha -2k-1}=
\end{equation*}%
\begin{equation*}
=(\gamma +1)\ldots (\gamma +\alpha +1)f^{-\gamma
-\alpha -2}g^{\alpha +1}+(c_{1}(\gamma)(\gamma +\alpha-1)-c_{[\alpha /2]}\times
\end{equation*}%
\begin{equation*}
\times \left( \gamma \right) f^{-\gamma +2[\alpha /2]-\alpha }\left( \alpha
-2[\alpha /2]\right) g^{\alpha -2[\alpha /2]-1}+
\end{equation*}%
\begin{equation*}
+\sum_{k=1}^{[\alpha /2]-1}(c_{k+1}\left( \gamma \right) \left( \gamma
-2k+\alpha -1\right) -\left( \alpha -2k\right) c_{k}\left( \gamma \right)
)f^{-\gamma +2k-\alpha }g^{\alpha -2k-1},
\end{equation*}%
as was shown.
\end{proof}

\begin{lemma}
\label{AprL3}If $s$ \textit{is the integer non-negative number then the}
following formulas are valid:%
\begin{equation}
J=\mbox{\bf res}_{t}(1-t)^{-1}(1+t)^{2s+1}t^{-s-1}=2^{2s},
\label{KK16}
\end{equation}%
\begin{equation}
J_{k}=\mbox{\bf res}_{t}(1-t)^{k-1}(1+t)^{2s-k+1}t^{-s-1}=0,\text{ }\forall k=1,\ldots ,2s.
\label{KK17}
\end{equation}
\end{lemma}

\begin{proof}
We have%
\begin{equation*}
J=\mbox{\bf res}_{t}(1-t)^{-1}(1+t)^{2s+1}t^{-s-1}:=\mbox{\bf res}_{t=0}(1-t)^{-1}(1+t)^{2s+1}t^{-s-1}=
\end{equation*}%
\begin{equation*}
(the\, theorem\, of\, the\,full\, sum\, of\, residuis)
\end{equation*}%
\begin{equation*}
=-\mbox{\bf res}_{t=1}(1-t)^{-1}(1+t)^{2s+1}t^{-s-1}-\mbox{\bf res}_{t=\infty }(1-t)^{-1}(1+t)^{2s+1}t^{-s-1}=
\end{equation*}%
\begin{equation*}
(directly\text{ }by\text{ }definition\text{ }of\text{ }a\text{ }deduction%
\text{ }in\text{ }a\text{ }corresponding\text{ }point)
\end{equation*}%
\begin{equation*}
=[(1+t)^{2s+1}t^{-s-1}]_{t=1}-res_{t=0}(1-1/t)^{-1}(1+1/t)^{2s+1}(1/t)^{-s}(-1/t)^{2}=
\end{equation*}%
\begin{equation*}
=2^{2s+1}-\mbox{\bf res}_{t=0}(1-t)^{-1}(1+t)^{2s+1}t^{-s}=
2^{2s+1}-J\Leftrightarrow J=2^{2s+1}-J\Rightarrow J=2^{2s}.
\end{equation*}%
Let $k$ be the any fixed number from set $\{1,\ldots ,2s\}$. Acting just as
in the previous case, we have%
\begin{equation*}
J_{k}=\mbox{\bf res}_{t}(1-t)^{k-1}(1+t)^{2s-k+1}t^{-s-1}:=
\end{equation*}%
\begin{equation*}
\mbox{\bf res}_{t=0}\frac{(1-t)^{k-1}(1+t)^{2s-k+1}}{t^{s+1}}-
\mbox{\bf res}_{t=\infty}\frac{(1-t)^{k-1}(1+t)^{2s+1}}{t^{s+1}}=
\end{equation*}%
\begin{equation*}
(\text{ as\,\,powers\,\,of\,\,binomials} (1-t)^{k-1}\,\text{and}\,(1+t)^{2s-k+1}t^{-s-1}\,\text{non\, negative}
\end{equation*}%
\begin{equation*}
\text{at\, any}\,k\,\text{from\,\,set}\, \{1,\ldots ,2s\})
\end{equation*}%
\begin{equation*}
=0-\mbox{\bf res}_{t=0}(1-1/t)^{k-1}(1+1/t)^{5}(1/t)^{-3}(-1/t)^{2}=
\end{equation*}%
\begin{equation*}
=-\mbox{\bf res}_{t=0}(1-t)^{-s-1}(1+t)^{2s-k+1}t^{-s-1}=
-J_{k}\Longleftrightarrow J_{k}=-J_{k}\Longleftrightarrow
J_{k}=0._{{}}
\end{equation*}
\end{proof}

\begin{theorem}
\label{AprT1} The identity (\ref{KK2}) is valid .
\end{theorem}

\begin{proof}
By (\ref{KK6}) we have following integrated representation
for sum $T(s; \alpha,\,\beta) $ in the left part of identity (\ref{KK2}):%
\begin{equation*}
T( s;\alpha,\,\beta) =\mbox{\bf res}_{w_{0},\ldots ,w_{d},t}\left\{\frac{(1-t)^{d+\sum_{i=0}^{d}(\alpha _{i}+\gamma _{i})}}
{t^{s+1}}\prod_{i=0}^{d}
\frac{\left(\exp (-w_{i})-t\exp (w_{i})\right)^{-\mu _{i}-1}}{w_{i}^{\alpha_{i}+1}}\right\}=
\end{equation*}%
\begin{equation*}
=\mbox{\bf res}_{t}\{t^{-s-1}(1-t)^{d+\sum_{i=0}^{d}(\alpha
_{i}+\gamma _{i})}(\prod_{i=0}^{d}\mbox{\bf res}_{w_{i}}w_{i}^{-\alpha_{i}-1}
\left( \exp (-w_{i})-t\exp (w_{i})\right) ^{-\mu _{i}-1})\}.
\end{equation*}%
Calculating in last expression each of deductions on variables $w_{0},\ldots
,w_{d}$ by the formula \textbf{(}\ref{KK14}\textbf{)} we have%
\begin{equation*}
T(s;\,\alpha,\,\beta) =\mbox{\bf res}_{t}\{t^{-s-1}(1-t)^{d+\sum_{i=0}^{d}(\alpha _{i}+\gamma
_{i})}\times
\end{equation*}%
\begin{equation*}
\times \prod_{i=0}^{d}\binom{\alpha _{i}+\gamma _{i}}{\alpha_{i}}
(1-t)^{-\gamma _{i}-\alpha _{i}-1}(1+t)^{\alpha_{i}}(\sum_{k=1}^{[\alpha _{i}/2]}h_{k}(\alpha _{i},\gamma _{i})
(1-t) ^{2k}(1+t)^{-2k})\}=
\end{equation*}%
%\begin{equation*}
(trivial reductions in an $ sing $ of product
%\end{equation*}%
\begin{equation*}
\prod_{i=0}^{d}\ldots \text{ in\,\,according\,\,of\,\,the\,\,assumption}
\sum_{i=0}^{d}\alpha_{i}=2s+1)
\end{equation*}%
\begin{equation}
=\binom{\alpha +\gamma}{\alpha}\mbox{\bf res}_{t}\{\frac{(1+t)^{2s+1}}{t^{s+1}(1-t)}
\prod_{i=0}^{d}(1+\sum_{k=1}^{[\alpha _{i}/2]}h_{k}(\alpha_{i},\gamma_{i}) (1-t)^{2k}(1+t)^{-2k})\}.
\label{KK18}
\end{equation}%
As $[\alpha _{0}/2]+[\alpha _{1}/2]+\ldots +[\alpha _{d}/2]\leq \lbrack
\sum_{i=0}^{d}\alpha_{i}/2]=s,$
it is easy, that after disclosing brackets
and reduction of similar members of product%
\begin{equation*}
\prod_{i=0}^{d}(1+\sum_{k=1}^{[\alpha _{i}/2]}h_{k}(\alpha _{i},\gamma_{i}) ( 1-t)^{2k}(1+t)^{-2k}),
\end{equation*}%
in a sign of $\mbox{\bf res}_{t}$ in (\ref{KK18})\textbf{\ }is\textbf{\ }%
representable in the form of a multinominal of a kind%
\begin{equation*}
1+\sum_{k=2}^{2s}\lambda_{k}(1-t)^{k}(1+t)^{-k},
\end{equation*}%
where coefficients $\lambda_{1},\ldots ,\lambda_{2s-1}$ are some fixed
rational numbers. Thus
\begin{equation*}
T(s;\alpha,\beta)=\binom{\alpha+\gamma}{\alpha}\mbox{\bf res}_{t}
\left\{\frac{(1+t)^{2s+1}}{t^{s+1}(1-t)}
(1+\sum_{k=1}^{2s}\lambda_{k}(1-t)^{k}(1+t)^{-k})\right\}=
\end{equation*}%
\begin{equation*}
=\binom{\alpha +\gamma}{\alpha}\left(\mbox{\bf res}_{t}\frac{(1+t)^{2s+1}}{(1-t)t^{s+1}}
+\sum_{k=1}^{2s}\lambda_{k}\mbox{\bf res}_{t}\frac{(1-t)^{k-1}(1+t)^{2s-k+1}}{t^{s+1}}\right)=
\end{equation*}%
%\begin{equation*}
(calculation of residius in last expression under formulas (\ref{KK16}) and (\ref{KK17}))
%\end{equation*}%
\begin{equation*}
=\binom{\alpha +\gamma}{\alpha}\{2^{2s}+\sum_{k=1}^{2s}\lambda_{k}\times 0\}=\binom{\alpha +\gamma}{\alpha}2^{2s}.
\end{equation*}
\end{proof}

\begin{remark}
It is to our interest to answer this question
%The answer to a question is of interest, what additional information the
%knowledge of the integral representation of expression
\begin{equation*}
J=\mbox{\bf res}_{w_{0},...,w_{d},t}\left\{t^{-s-1}\prod_{i=0}^{d}w_{i}^
{-\alpha_{i}-1}\left(\exp (-w_{i})-t\exp (w_{i})\right)^{-\mu _{i}-1}\times\right.
\end{equation*}%
\begin{equation}
\left.\times(1-t)^{d+\sum_{i=0}^{d}(\alpha_{i}+\gamma_{i})}\right\}/\mathbf{\alpha}!,
\label{KK20}
\end{equation}%
in the left hand of initial identity (\ref{KK1}). For example, the
integral (\ref{KK20}) can be resulted to in the following kind%
\begin{equation}
J=\mbox{\bf res}_{t}\{(t^{-s-1}\prod_{i=0}^{d}\mbox{\bf res}%
_{w_{i}}w_{i}^{-\alpha _{i}-1}\left( \exp (-\lambda _{i}w_{i})-t\exp
(\lambda _{i}w_{i})\right)^{-\gamma _{i}-1}\}/\mathbf{\alpha }!.
\label{KK21}
\end{equation}%
The calculation of integral (\ref{KK21}) is connected with studying
hyperbolic $t$-sine \cite{FoataHan(2010)}
\begin{equation}
\sinh _{t}(x):=(\exp (-x)-t\exp (x))/2,
\label{KK22}
\end{equation}%
and the functions $\sinh_{t}^{-\gamma }(x),$ $\gamma \in \mathbb{N},$ \textit{and}
\begin{equation*}
J_{a,\gamma }(t):=\mbox{\bf res}_{z}(z^{-\alpha -1}(\exp(-z)-t\exp (z))^{-\gamma -1})/\alpha!=
\end{equation*}%
\begin{equation}
=\mbox{\bf res}_{z}(z^{-\alpha-1}(\exp (-z)-t\exp (z))^{-\gamma -1})/\alpha !.
\label{KK23}
\end{equation}%
In my opinion, the study of these functions is interesting, including
their combinatorial interpretation and various relations with them.
\end{remark}

%\subsection{References\protect\medskip}

%%\begin{thebibliography}{9}
%%\bibitem{AbrTs97} \textit{Abramov S.A. and Tsarev S.P.} (1997). On
%%peripheral factorization of linear ordinary operators, Programming \&
%%Computer Software, No.~1, p.~59--67.

%\bibitem{AizenYuzh83} \textit{Aizenberg L.A. and Yuzhakov A.P.} (1979).
%Integral representation and residues in multidimensional complex analysis.
%Nauka, Novosibirsk (in Russian).

%%\bibitem{Andrews70} \textit{Andrews G.E.} (1970). On the foundations of
%%combinatorial theory. IV. Finite vector space and Eulerian generating
%%functions. Stud. Appl. Math. \textbf{49}, p.~239--258.

%\bibitem{XeoXu(2000)} \textit{Hео S. and Хu Y.} Invariant
%cubature formulae for spheres and balls by combinatorial methods. SIAM J.
%Numer. Anal. \textbf{38}(2), 626--638.

%%\bibitem{Davl(2011)} Davletshin M.N.\textit{\ }(2011). Enumeration of the
%%D-ideals of ring $R_{n}(K,J$). Journal of Siberian Federal University 4(3),10--26.

%%\end{thebibliography}

%Mathematics \& Physics \textbf{4}(3), Krasnoyarsk, 10--26 (in Russian).
%\begin{thebibliography}{9}

\addcontentsline{toc}{section}{The bibliography}
\renewcommand{\refname}{The bibliography}

\end{document}